\documentclass[12pt,a4paper]{article}
\usepackage{amsmath,amssymb,amsthm,eucal,url,cite}

\usepackage{dsfont} 
\usepackage{bigints} 

\allowdisplaybreaks[3]

\usepackage{enumitem}

\numberwithin{equation}{section}
\theoremstyle{plain}
\newtheorem{thm}{Theorem}[section]
\newtheorem{lem}[thm]{Lemma}

\newtheorem{prop}[thm]{Proposition}
\newtheorem{cor}[thm]{Corollary}

\newtheorem{definition}[thm]{Definition}
\theoremstyle{remark}
\newtheorem{remark}[thm]{Remark}
\newtheorem{rem}[thm]{Remark}

\theoremstyle{plain}

\DeclareMathOperator{\sign}{sgn}

\newcommand{\ess}{\mathrm{ess}}

\newcommand{\RR}{\mathbb{R}}
\newcommand{\CC}{\mathbb{C}}
\newcommand{\NN}{\mathbb{N}}
\newcommand{\ZZ}{\mathbb{Z}}
\newcommand{\TT}{\mathbb{T}}

\newcommand{\one}{\mathds{1}}

\newcommand{\cC}{\mathcal{C}}
\newcommand{\cD}{\mathcal{D}}

\newcommand{\cH}{\mathcal{H}}
\newcommand{\cG}{\mathcal{G}}
\newcommand{\cK}{\mathcal{K}}

\newcommand{\cN}{\mathcal{N}}

\newcommand{\cP}{\mathcal{P}}

\newcommand{\cT}{\mathcal{T}}

\newcommand{\dd}{\,\mathrm{d}}
\newcommand{\rmi}{\mathrm{i}}

\newcommand{\dsum}{\mathop{\dot{+}}}

\newcommand{\dom}{\mathop{\mathrm{dom}}}
\newcommand{\ran}{\mathop{\mathrm{ran}}}
\newcommand{\pv}{\mathop{\mathrm{p.v.}}}

\DeclareMathOperator{\supp}{supp}

\newcommand{\spec}{\mathop{\mathrm{spec}}\nolimits}

\newcommand{\res}{\mathop{\mathrm{res}}\nolimits}

\usepackage[usenames,dvipsnames]{color}
\newcommand{\kp}[1]{{#1}}

\begin{document}

\title{\bf Two-dimensional Dirac operators with~singular interactions supported on~closed curves}

\author{
\sc Jussi Behrndt\\
\small Institut f\"{u}r Angewandte Mathematik, Technische Universit\"{a}t Graz\\[-\smallskipamount]
\small Steyrergasse 30, 8010 Graz, Austria\\[-\smallskipamount]
\small E-mail: \url{behrndt@tugraz.at}\\[-\smallskipamount]
\small Webpage: \url{http://www.math.tugraz.at/~behrndt/}\\[\medskipamount]
\sc Markus Holzmann\\
\small Institut f\"{u}r Angewandte Mathematik, Technische Universit\"{a}t Graz\\[-\smallskipamount]
\small Steyrergasse 30, 8010 Graz, Austria\\[-\smallskipamount]
\small E-mail: \url{holzmann@math.tugraz.at}\\[-\smallskipamount]
\small Webpage: \url{http://www.math.tugraz.at/~holzmann/}\\[\medskipamount]
\sc Thomas Ourmi\`eres-Bonafos\\
\small CNRS \& Universit\'e Paris-Dauphine, PSL  University, CEREMADE,\\[-\smallskipamount]
\small Place de Lattre de Tassigny, 75016 Paris, France\\[-\smallskipamount]
\small E-mail: \url{ourmieres-bonafos@ceremade.dauphine.fr}\\[-\smallskipamount]
\small Webpage: \url{http://www.ceremade.dauphine.fr/~ourmieres/}\\[\medskipamount]
\sc Konstantin Pankrashkin\\
\small Laboratoire de Math\'ematiques d'Orsay, Univ.~Paris-Sud, CNRS,\\[-\smallskipamount]
\small Universit\'e Paris-Saclay, 91405 Orsay, France\\[-\smallskipamount]
\small E-mail: \url{konstantin.pankrashkin@math.u-psud.fr}\\[-\smallskipamount]
\small Webpage: \url{http://www.math.u-psud.fr/~pankrashkin/}
}

\date{}

\maketitle

\begin{abstract}

\kp{
We study the two-dimensional Dirac operator with an arbitrary combination of electrostatic and Lorentz scalar $\delta$-interactions of constant strengths supported on a smooth closed curve. For any combination of the coupling constants a rigorous description of the self-adjoint realizations of the operators is given and the qualitative spectral properties are described. The analysis covers also all so-called critical combinations of coupling constants, for which there is a loss of regularity in the operator domain. In this case, if the mass is non-zero, the resulting operator has an additional point in the essential spectrum, and the position of this point inside the central gap can be made arbitrary by a suitable choice of the coupling constants. The analysis is based on a combination of the extension theory of symmetric operators with a detailed study of boundary integral operators viewed as periodic pseudodifferential operators.
}
\end{abstract}

%
\tableofcontents
%


\section{Introduction}

\subsection{Motivations and state of the art}

Initially introduced to model the effects of special relativity on the behavior of quantum particles of spin $\frac12$ (such as electrons),
the Dirac operator also comes into play as an effective operator when studying low-energy electrons in a single layered material like graphene.
In order to model the interaction of the particles with external forces, the Dirac operator is coupled to a potential,
and the understanding of the spectral features of the resulting Hamiltonian translates into dynamical properties of the quantum system. 

In the last few years a class of singular potentials has been extensively studied in this relativistic setting. These potentials, which are usually called $\delta$-interactions, are supported on sets of Lebesgue measure zero and used as idealized replacements for regular potentials localized in thin neighborhoods of the interaction supports in the ambient Euclidean space.
In nonrelativistic quantum mechanics, these interactions were successfully studied in the case of Schr\"odinger operators with point interactions in \cite{AGHH} or  with $\delta$-interactions supported on hypersurfaces in $\mathbb{R}^d$, e.g., in \cite{BLL13, BEKS94, E08}. In the relativistic setting, the one-dimensional Dirac operators with $\delta$-potentials supported on points
are well studied, see \cite{AGHH, CMP, GS, PR}.
The case of potentials supported on surfaces in $\mathbb{R}^3$ was recently discussed
in \cite{AMV14, AMV15, AMV15bis, BEHL17, BEHL19, BH, DES89, hobp, OP, OV}. We also mention a recent contribution in the two-dimensional case is \cite{PV}
for a class of interactions with a non-smooth support.
In the above works, it was observed that there are critical interaction strengths
for which the standard elliptic regularity fails, and the self-adjoint realization of the operator shows
a loss of regularity in the operator domain. As as a result, the spectral properties of the
operator may be different from what was observed for the non-critical case \cite{BH}, but
no exhaustive study for all critical interaction strengths is available so far.

In this paper we are considering Dirac operators in $\mathbb{R}^2$ with electrostatic and Lorentz scalar $\delta$-potentials supported on smooth closed curves,
and we provide a systematic approach combining the general theory of boundary triples with some elements of the pseudodifferential calculus for matrix-valued singular
integral operators. A similar combination of methods was used successfully in \cite{CPP} to study a class of sign-changing Laplacians.
Our main advance is that we show the self-adjointness of the resulting operators and discuss spectral properties for \emph{all} possible combinations of interaction strengths, which includes
all critical cases. This answers fully the question of \cite[Open Problem 11]{OP} in dimension two.

Let us introduce the problem setting in greater detail. To set the stage, let $\Sigma$ be a smooth planar loop, i.e.
a closed non-self-intersecting $C^\infty$-smooth curve in $\RR^2$. It splits $\mathbb{R}^2$ into a bounded domain $\Omega_+$
and an unbounded domain $\Omega_-$, and we denote by $\nu=(\nu_1,\nu_2)$ the unit normal to $\Sigma$ pointing outwards of $\Omega_+$.
For a function $f$ defined on $\mathbb{R}^2$ we will often use the notation $f_\pm := f \upharpoonright \Omega_\pm$
with $\upharpoonright \Omega_\pm$ meaning the restriction to $\Omega_\pm$. 
If a function $f$ has suitably defined Dirichlet traces on the both sides of $\Sigma$, we define the distribution $\delta_\Sigma f$ by
\begin{equation*}
	\langle\delta_\Sigma f , \varphi\rangle:= \int_\Sigma \frac{1}{2} \,\big( \mathcal{T}_+^D f_+ + \mathcal{T}_-^D f_- \big) \cdot \varphi \dd s, \quad \varphi\in C^\infty_0(\RR^2),
\end{equation*}
where $\mathcal{T}_\pm^D f_\pm$ denotes the Dirichlet trace of $f_\pm$ at $\Sigma$ and $\dd s$
means the integration with respect to the arc-length.
We are going to study Dirac operators $A_{\eta, \tau}$ in $L^2(\mathbb{R}^2; \mathbb{C}^2)$ given by the formal differential expression
\begin{equation*}
  D_{\eta, \tau} := -\rmi \big( \sigma_1 \partial_1 + \sigma_2 \partial_2 \big) + m \sigma_3 + (\eta \sigma_0 + \tau \sigma_3) \delta_\Sigma,
\end{equation*}
where $\sigma_0$ is the identity matrix in $\mathbb{C}^{2 \times 2}$, $\sigma_1, \sigma_2, \sigma_3$ are the $\mathbb{C}^{2 \times 2}$-valued Pauli spin matrices defined in~\eqref{def_Pauli_matrices}, and $m, \eta, \tau \in \mathbb{R}$.
\kp{Following the standard language \cite{T92} one may interpret $\eta$ and $\tau$ as the strengths of the electrostatic and Lorentz scalar interactions on $\Sigma$, respectively, while the parameter $m$ is usually interpreted as the mass. }
Integration by parts shows that if the distribution $D_{\eta, \tau} f$ is generated by an $L^2$-function, then the function $f$ has to fulfil (at least formally)
the transmission condition  
\begin{equation}
   \label{transm00}
-\rmi \,(\sigma_1 \nu_1+\sigma_2\nu_2) \, (\mathcal{T}_+^D f_+ - \mathcal{T}_-^D f_-) = \frac12(\eta \sigma_0 + \tau \sigma_3)(\mathcal{T}_+^D f_+ + \mathcal{T}_-^D f_-).
\end{equation}
Our goal is to make this observation rigorous and to show that there is a unique reasonably defined self-adjoint operator $A_{\eta, \tau}$ in $L^2(\mathbb{R}^2; \mathbb{C}^2)$ for this transmission condition
and then to study its qualitative spectral properties.

\kp{
Our approach is to consider $A_{\eta,\tau}$ as an extension of a suitably chosen symmetric operator. This allows one to make use of the standard machinery of boundary triples
\cite{BHS19, BGP, DM91, DM95} and to reformulate the main questions in terms of operators on $\Sigma$.
While a similar approach was used in in \cite{BEHL17, BH, CMP, PR}, the main new ingredient is provided by an additional detailed study of various
integral operators arising in the construction. Closely related objects already appeared e.g. in \cite{AMV14, AMV15, AMV15bis, BEHL17, BEHL19, BH, OV}
for the three-dimensional case, but for the two-dimensional case we manage to provide a more detailed analysis with the help of the periodic pseudodifferential calculus, which is an important finding in the present paper.
}

\subsection{Main results}
Let us pass to the formulation and discussion of the main results of this paper. To define the operator $A_{\eta, \tau}$ rigorously, we introduce for an open set $\Omega \subset \RR^2$
\begin{gather*}
H(\sigma,\Omega)=\big\{ f\in L^2(\Omega;\CC^2):\, ( \sigma_1 \partial_1 + \sigma_2 \partial_2 ) f\in L^2(\Omega;\CC^2)\big\}.
\end{gather*}
One can show that functions $f_\pm$ in $H(\sigma, \Omega_\pm)$ admit Dirichlet traces $\mathcal{T}^D_\pm f_\pm$ in $H^{-\frac{1}{2}}(\Sigma; \mathbb{C}^2)$. With these notations in hand we define now, following~\eqref{transm00}, for $\eta, \tau \in \mathbb{R}$ the operator $A_{\eta, \tau}$ in $L^2(\mathbb{R}^2; \mathbb{C}^2)$ by
\begin{equation}\label{dirdelta}
\begin{aligned}
  A_{\eta, \tau} f &:= \big(-\rmi (\sigma_1 \partial_1 + \sigma_2 \partial_2) + m \sigma_3\big) f_+ \oplus \big(-\rmi (\sigma_1 \partial_1 + \sigma_2 \partial_2) + m \sigma_3\big) f_-, \\
    \dom A_{\eta, \tau} &:= \Big\{ f = f_+ \oplus f_- \in H(\sigma,\Omega_+)\oplus H(\sigma,\Omega_-):\\ 
 	{}-\rmi\,(&\sigma_1 \nu_1  + \sigma_2 \nu_2)\big( \cT_+^D f_+ - \cT_-^D f_-\big) = \frac12\,(\eta \sigma_0 + \tau \sigma_3)\big(\cT_+^D f_+ + \cT_-^D f_-\big)\Big\}.
 	  \end{aligned}
\end{equation}
It turns out that the value $\eta^2 - \tau^2$ plays a special role. More precisely, if $\eta^2 - \tau^2 =4$ we will say that we are in a \emph{critical} case,
while all the cases with $\eta^2 - \tau^2 \ne 4$ will be referred to as \emph{non-critical} ones.
We also remark that at some combinations of coupling constants the boundary condition in \eqref{dirdelta} appears to be decoupling, i.e.
the operator $A_{\eta, \tau}$ becomes the direct sum of two operators acting in $\Omega_\pm$, see Lemma~\ref{lemma_confinement} below.

It appears that the non-critical case is easier to deal with, and the results for $A_{\eta, \tau}$ are summarized as follows:
\begin{thm}[Non-critical case] \label{theorem_noncritical}
  Let $\eta, \tau \in \mathbb{R}$ with $\eta^2 - \tau^2 \neq 4$. Then $A_{\eta, \tau}$ is self-adjoint in $L^2(\mathbb{R}^2; \mathbb{C}^2)$
	with $\dom A_{\eta, \tau} \subset H^1(\mathbb{R}^2 \setminus \Sigma; \mathbb{C}^2)$, its essential spectrum is
  \begin{equation*}
    \spec_\ess A_{\eta, \tau} = \big(-\infty, -|m|\big] \cup \big[|m|, +\infty\big),
  \end{equation*}
  and the discrete spectrum of $A_{\eta, \tau}$ is finite. 
\end{thm}

The proof of Theorem~\ref{theorem_noncritical} is given in Section~\ref{sec-noncrit}. There, also some additional properties of $A_{\eta, \tau}$ like a Krein-type resolvent formula, an abstract version of the Birman-Schwinger principle, and some symmetry relations in the point spectrum of $A_{\eta, \tau}$ are shown. Similar results are known in the three-dimensional case, see \cite{BEHL19}. 

Our main results in the critical case $\eta^2 - \tau^2 = 4$ are stated in the following theorem. In particular, this shows that there is a loss of regularity in the domain of $A_{\eta, \tau}$ and that there is an additional point in the essential spectrum if $m\ne 0$, which is in contrast to  the non-critical case.

\begin{thm}[Critical case] \label{theorem_critical}
  Let $\eta, \tau \in \mathbb{R}$ with $\eta^2 - \tau^2 = 4$. Then $A_{\eta, \tau}$ is self-adjoint in $L^2(\mathbb{R}^2; \mathbb{C}^2)$ 
	and the restriction of $A_{\eta, \tau}$ onto $\dom A_{\eta, \tau} \cap H^1(\mathbb{R}^2 \setminus \Sigma; \mathbb{C}^2)$
	is essentially self-adjoint, while
	$\dom A_{\eta, \tau} \not\subset H^s(\mathbb{R}^2 \setminus \Sigma; \mathbb{C}^2)$ for any $s > 0$. The essential spectrum of $A_{\eta, \tau}$ is
  \begin{equation*}
    \spec_\ess A_{\eta, \tau} = \big(-\infty, -|m|\big] \cup \left\{ -\frac{\tau}{\eta} m \right\} \cup \big[|m|, +\infty\big).
  \end{equation*}
\end{thm}

Theorem~\ref{theorem_critical} is the main result of this paper, and it is proved in Section~\ref{sec-crit}. There, also a Krein type resolvent formula, a Birman Schwinger principle, and several symmetry relations in the point spectrum of $A_{\eta, \tau}$ are shown. We would like to point out that some analogs in three dimensions are only known in the case of purely electrostatic interactions,
i.e. when $\eta = \pm 2$ and $\tau = 0$, see \cite{BH, OV}. We remark that the additional point $-\frac{\tau}{\eta} m$ can take any value in the gap $\big(-|m|, |m|\big)$
under a suitable choice of $\eta$ and $\tau$, and this effect was not observed in previous works. Several papers addressed the question of presence of a non-empty essential
spectrum for Dirac operators in bounded domains with various boundary conditions, see e.g. \cite{BFSV,freitas,kms}, and our results can also be regarded as a contribution in this direction.

By a minor modification of the argument, one can also deal with an interaction supported on several loops.
Let $N\geq1$ and consider a family of  non-intersecting smooth loops $\Sigma_1, \dots, \Sigma_N$  with normals $\nu_j$, $j\in\{1,\dots,N\}$.
We set $\Sigma:=\bigcup_{j=1}^N \Sigma_j$, and for $f \in H(\sigma,\RR^2\setminus\Sigma)$ we denote its Dirichlet traces
on the two sides of $\Sigma_j$ as $\mathcal{T}_{\pm, j}^Df$,
where $-$ corresponds to the side to which  $\nu_j$ is directed.
In addition, consider a family of pairs of real parameters
\[
\mathcal{P} := \big((\eta_j,\tau_j)\big)_{j\in\{1,\dots,N\}}, \quad
\eta_j,\,\tau_j\in\RR,
\] and define the associated operator $A_{\Sigma,\mathcal{P}}$ by
\begin{equation}\label{dirdeltamult}
\begin{aligned}
  A_{\Sigma,\mathcal{P}} f &:= \big({}-\rmi (\sigma_1 \partial_1 + \sigma_2 \partial_2) + m \sigma_3\big) f \qquad \text{ in } \RR^2\setminus\Sigma, \\
    \dom A_{\Sigma,\mathcal{P}} &:= \Big\{ f\in H(\sigma,\RR^2\setminus\Sigma): \\
 	 {}-\rmi\,&(\sigma_1 \nu_{j,1} + \sigma_2 \nu_{j,2})\big( \cT_{+,j}^D f - \cT_{-,j}^D f\big) = \frac12\,(\eta \sigma_0 + \tau \sigma_3)\big(\cT_{+,j}^D f + \cT_{-,j}^D f\big)\\
	&\qquad \text{ for each } j\in\{1,\dots,N\}\Big\}.
 	  \end{aligned}
\end{equation}
Then the preceding results can be extended as follows:
\begin{thm}[Interaction supported on several loops]\label{sevloops} Denote
\[
\mathcal{I}_{\rm crit} := \big\{ j \in \{1,\dots,N\} : \eta_j^2 - \tau_j^2 = 4\big\}.
\]
Then the following is true:
\begin{enumerate}
	\item[$\textup{(i)}$] If $\mathcal{I}_{\rm crit} = \emptyset$, then $A_{\Sigma,\mathcal{P}}$ is self-adjoint with $\dom A_{\Sigma,\mathcal{P}} \subset H^1(\mathbb{R}^2 \setminus \Sigma; \mathbb{C}^2)$, the essential spectrum of $A_{\Sigma,\mathcal{P}}$ is
  \begin{equation*}
    \spec_\ess A_{\Sigma,\mathcal{P}} = \big(-\infty, -|m|\big] \cup \big[|m|, \infty\big),
  \end{equation*}
  and the discrete spectrum of $A_{\Sigma,\mathcal{P}}$ in $(-|m|, |m|)$ is finite.
	\item[$\textup{(ii)}$] If $\mathcal{I}_{\rm crit} \neq \emptyset$, then $A_{\Sigma,\mathcal{P}}$ is self-adjoint and the restriction of $A_{\Sigma,\mathcal{P}}$ onto $\dom A_{\Sigma,\mathcal{P}} \cap H^1(\mathbb{R}^2 \setminus \Sigma; \mathbb{C}^2)$ is essentially self-adjoint in $L^2(\mathbb{R}^2; \mathbb{C}^2)$, but
	 $\dom A_{\Sigma,\mathcal{P}} \not\subset H^s(\mathbb{R}^2 \setminus \Sigma; \mathbb{C}^2)$ for any $s > 0$.
	The essential spectrum of $A_{\Sigma,\mathcal{P}}$ is
	\[
	\spec_\ess A_{\Sigma,\mathcal{P}}=\big(-\infty,-|m|\big] \bigcup_{j \in \mathcal{I}_{\rm crit}} \Big\{-\frac{\tau_j}{\eta_j}m\Big\} \cup \big[|m|,+\infty\big).
	\]
\end{enumerate}
\end{thm}
\kp{
In particular, one easily observes that if $\Sigma$ has $N$ connected components, then for any finite set $\Xi\subset \big(-|m|,|m|\big)$ with $\#\Xi\le N$
it is possible to find a combination of parameters $\cP$ such that the essential spectrum of $A_{\Sigma,\cP}$ in $\big(-|m|,|m|\big)$ coincides with~$\Xi$.
Necessary modifications for the proof of Theorem~\ref{sevloops} are sketched in Subsection~\ref{secsevloops}.
}

%
%

\subsection{Structure of the paper} Let us shortly describe the structure of the paper. First, in Section~\ref{section_preliminaries} we recall some well-known facts on periodic pseudodifferential operators on curves, boundary triples, and Schur complements of block operator matrices. With that we study then in Section~\ref{section_boundary_triple} integral operators, which are associated to the Green function corresponding to the free Dirac operator in $\mathbb{R}^2$, and construct a boundary triple which is suitable to study the properties of $A_{\eta, \tau}$.
\kp{The two sections~\ref{section_preliminaries} and~\ref{section_boundary_triple} occupy an important portion of the text, which is due to the big number of tools from various
domains which are put together and which are rarely (if at all) used simultaneously. We believe that the construction can be of use for other two-dimensional boundary value problems
with the help of the boundary triple machinery. } 
Finally, Section~\ref{section_delta_op} is devoted to the proofs of the main results of this paper, Theorems~\ref{theorem_noncritical}--\ref{dirdeltamult}.

\subsection{Notations}

\kp{We use the convention $0\notin \NN$ and set $\NN_0:=\NN\,\cup\,\{0\}$.}

We denote the $2\times 2$ identity matrix by $\sigma_0$ and the $2\times 2$-Pauli spin matrices by
\begin{equation} \label{def_Pauli_matrices}
  \sigma_1 := \begin{pmatrix} 0 & 1 \\ 1 & 0 \end{pmatrix}, \quad
  \sigma_2 := \begin{pmatrix} 0 & -\rmi \\ \rmi & 0 \end{pmatrix}, \quad
  \sigma_3 := \begin{pmatrix} 1 & 0 \\ 0 & -1 \end{pmatrix}.
\end{equation}
They fulfil
\begin{equation} \label{anti_commutation}
  \sigma_j \sigma_k + \sigma_k \sigma_j = 2 \delta_{j k} \sigma_0, \quad j, k \in \{ 1, 2, 3 \}.
\end{equation}
For $x=(x_1,x_2) \in \mathbb{R}^2$ we write $\sigma \cdot x: = \sigma_1 x_1 + \sigma_2 x_2$ and, in the same spirit, $\sigma \cdot \nabla := \sigma_1 \partial_1 + \sigma_2 \partial_2$.

Next, $\Sigma \subset \mathbb{R}^2$ is always a $C^\infty$-loop of length $\ell > 0$, which splits $\mathbb{R}^2$ into a bounded domain $\Omega_+$ and an unbounded domain $\Omega_-$ with common boundary $\Sigma$. By $\nu$ we denote the unit normal vector field at $\Sigma$ which points outwards of $\Omega_+$, and $\bf t$ denotes the unit tangent vector at $\Sigma$. If $\gamma: [0, \ell] \rightarrow \mathbb{R}^2$ is an arc length parametrization of $\Sigma$ with positive orientation, then we have ${\bf t} = \gamma'$ and $\nu = (\gamma_2', - \gamma_1')$. We sometimes identify the vector ${\bf t} \in \mathbb{R}^2$ with the complex number $T = {\bf t}_1 + \rmi {\bf t}_2$.

If $\Omega$ is a measurable set, we write, as usual, $L^2(\Omega)$ for the classical $L^2$-spaces and $L^2(\Omega; \mathbb{C}^2) := L^2(\Omega) \otimes \mathbb{C}^2$. If $\Omega = \Sigma$, then $L^2(\Sigma)$ is based on the inner product, where the integrals are taken with respect to the arc-length. By $H^s(\Omega)$ we denote the Sobolev spaces of order $s \in \mathbb{R}$ on $\Omega$, and the Sobolev spaces on the curve $\Sigma$ are reviewed in Section~\ref{section_PPDO}.

Next, we denote
\begin{equation*}
  \mathbb{T} := \mathbb{R} / \mathbb{Z}.
\end{equation*}
Then $C^\infty(\mathbb{T})$ can be identified with the space of all $1$-periodic $C^\infty(\mathbb{R})$-functions.
For $\alpha \in \mathbb{R}$ we denote the set of periodic pseudodifferential operators of order $\alpha$ on $\mathbb{T}$ by $\Psi^\alpha$ and the set of periodic pseudodifferential operators
of order $\alpha$ on $\Sigma$ by $\Psi^\alpha_\Sigma$ (see Definitions~\ref{definition_PPDO} and~\ref{definition_PPDO_Sigma} below).

For a linear operator $A$ in a Hilbert space $\mathcal{H}$ we write $\dom A$, $\ran A$, and $\ker A$ for its domain, range, and kernel, respectively. The identity operator is often denoted by $\one$. If $A$ is self-adjoint, then we denote by $\res A, \spec A, \spec_\textup{p} A$, and $\spec_\ess A$ its resolvent set, spectrum, point, and essential spectrum, respectively. If $A$ is self-adjoint and bounded from below, then $\mathcal{N}(A, z)$ is the number of eigenvalues smaller than $z$ taking multiplicities into account. For $z > \inf \spec_\ess A$ this is understood as $\mathcal{N}(A, z)=\infty$.

Finally, $K_j$ stands for the modified Bessel function of the second kind and order~$j$.

\subsection*{Acknowledgments}
Thomas Ourmi\`eres-Bonafos and Konstantin Pankrashkin were supported in part by the PHC Amadeus 37853TB funded by the French Ministry of Foreign Affairs and the French Ministry of Higher Education, Research and Innovation.
Jussi Behrndt and Markus Holzmann were supported by the Austrian Agency for International Cooperation in Education and Research (OeAD) within the project FR 01/2017.
Thomas Ourmi\`eres-Bonafos was also supported by the ANR "D\'efi des autres savoirs (DS10) 2017" programm, reference ANR-17-CE29-0004, project molQED.

\section{Preliminaries} \label{section_preliminaries}

In this section we provide some preliminary material from functional analysis and operator theory. First, in Section~\ref{section_PPDO} we 
recall the definition and some properties of periodic pseudodifferential operators on smooth curves and some 
special integral operators of this form. Afterwards, in Section~\ref{ssec-schur} a theorem on the Schur complement of block operator matrices is 
recalled and finally, in Section~\ref{ssec-bt} boundary triples and their $\gamma$-fields and Weyl functions are briefly discussed.

\subsection{Sobolev spaces and periodic pseudodifferential operators on closed curves} \label{section_PPDO}

In this section some properties of periodic pseudodifferential operators on closed curves are discussed. 
Special realizations of such operators will play an important role in the analysis of 
Dirac operators with singular interactions later. The presentation in this section follows closely the one in \cite[Chapters~5 and~7]{SV}.

\kp{Throughout this section $\Sigma \subset \mathbb{R}^2$ is always a $C^\infty$-smooth loop of length $\ell$. 
Recall that $\mathbb{T} := \mathbb{R} / \mathbb{Z}$.  By $\gamma: \ell\mathbb{T} \rightarrow \Sigma$ we denote
a fixed arc-length parametrization of $\Sigma$, i.e. a $C^\infty$-function with $|\gamma'(\cdot)|\equiv 1$ and $\gamma(\ell\mathbb{T})=\Sigma$.
}
First, we recall the construction of the Sobolev spaces on $\Sigma$. For that we recall some constructions for Sobolev spaces of periodic functions 
on the unit interval. 
For a distribution\footnote{In \cite{SV} the notation $\mathcal{D}_1'(\mathbb{R})$ is used instead of $\cD'(\mathbb{T})$. The subindex $1$ means the $1$-periodicity.}
$f \in \cD'(\mathbb{T}) := C^\infty(\mathbb{T})'$ we write, as usual,
\begin{equation*}
\widehat f (n):=\langle f,e_{-n}\rangle_{\mathcal{D}'(\mathbb{T}),\mathcal{D}(\mathbb{T})}\in\CC, \quad e_n(t)=e^{2\pi n \rmi  t}, \qquad n\in\ZZ,
\end{equation*}
for its Fourier coefficients. Recall that a distribution $f \in \mathcal{D}'(\mathbb{T})$ can be reconstructed from its Fourier coefficients by
\begin{equation} \label{equation_Fourier}
  f = \sum_{n \in \mathbb{Z}} \widehat{f}(n) e_n,
\end{equation}
where the series converges in $\mathcal{D}'(\mathbb{T})$, see \cite[Theorem~5.2.1]{SV}. For two distributions $f,g\in\cD'(\TT)$ we denote by $f\star g$ their
convolution which is defined (via its Fourier coefficients) by
\begin{equation*}
\widehat{f\star g} (n)=\widehat f(n)\, \widehat g(n), \quad n\in\NN.
\end{equation*}
In particular, for $f,g\in L^1(\TT)$ one simply has
\begin{equation*}
f\star g = \int_{\TT} f(s)g(\cdot-s)\dd s.
\end{equation*}
For convenience we set
\begin{equation*}
  \underline{n} := \begin{cases} 1, & n=0, \\ |n|, &n \neq 0, \end{cases} \qquad n \in \mathbb{Z}.
\end{equation*}
Then for $s \in \mathbb{R}$, the Sobolev space $H^s(\TT)$ consists of the distributions $f\in \cD'(\TT)$ with
\begin{equation*}
  \| f \|_{H^s(\TT)}^2 := \sum_{n \in \mathbb{Z}} \underline{n}^{2 s} \big|\widehat{f}(n)\big|^2 < \infty.
\end{equation*}
The set $H^s(\mathbb{T})$ endowed with the above norm becomes a Hilbert space.
If $s < t$, then $H^t(\mathbb{T})$ is compactly embedded into $H^s(\mathbb{T})$.

Having the definition of Sobolev spaces on $\mathbb{T}$, we can translate this to Sobolev spaces of 
order $s \in \mathbb{R}$ on $\Sigma$. For that we define on $\mathcal{D}'(\Sigma) := C^\infty(\Sigma)'$ the linear map
\begin{equation} \label{def_U}
U:\mathcal{D}'(\Sigma)\to \mathcal{D}'(\TT),
\quad
(U f)(\varphi)=f\left(\ell^{-1}\varphi(\ell^{-1} \gamma^{-1}(\cdot))\right), \quad \varphi \in C^\infty(\mathbb{T}).
\end{equation}
It is not difficult to verify that
\begin{equation}\label{ul}
U f(t) = f(\gamma(\ell t)),\quad f \in L^1(\Sigma),\,\,t\in\TT;
\end{equation}
this property will often be used. 
For $s \in \mathbb{R}$ we define the space
\begin{equation*}
  H^s(\Sigma) := \big\{ f \in \mathcal{D}'(\Sigma): U f \in H^s(\mathbb{T}) \big\},
\end{equation*}
which, endowed with the norm
\begin{equation*}
  \| f \|_{H^s(\Sigma)} := \| U f \|_{H^s(\mathbb{T})},\quad f\in H^s(\Sigma),
\end{equation*}
is a Hilbert space. By construction, the induced map 
\begin{equation}\label{uunitary}
U:H^s(\Sigma)\rightarrow H^s(\mathbb{T}),\quad s \in \mathbb{R},
\end{equation}
is unitary. 
For $f\in H^0(\Sigma)$ it is useful to observe that
\begin{equation*}
\Vert f\Vert_{H^0(\Sigma)}^2 =\| U f \|^2_{H^0(\mathbb{T})}= \sum_{n \in \mathbb{Z}}  \big|(U f,e_n)_{L^2(\TT)}\big|^2 
=\| U f \|^2_{L^2(\mathbb{T})} = \ell^{-1} \Vert f\Vert^2_{L^2(\Sigma)}.
\end{equation*}
Note also that the definition of $H^s(\Sigma)$ implies that $C^\infty(\Sigma)$ is dense in $H^s(\Sigma)$ for all $s \in \mathbb{R}$.

Next, we recall the definition of periodic pseudodifferential operators on $\mathbb{T}$ and translate this concept to 
periodic pseudodifferential operators on $\Sigma$.
Define first the linear operator $\omega$ acting on mappings $F:\ZZ\to \CC$ by
\begin{equation} \label{difference_op}
  (\omega F)(n) := F(n+1)-F(n), \quad n\in\ZZ.
\end{equation}

\begin{definition} \label{definition_PPDO}
  A linear operator $H$ acting on $C^\infty(\mathbb{T})$ is called a \emph{periodic pseudodifferential operator of order $\alpha \in \RR$}, 
  if there exists a function $h: \TT \times \ZZ\rightarrow\CC$
  with $h(\cdot, n) \in C^\infty(\TT)$ for each $n\in\ZZ$
  and 
  \begin{equation}
	       \label{eqopa}
    H u(t) = \sum_{n \in \ZZ}\, h(t, n) \,\widehat{u}(n) \,\,e_n(t) \text{ for all }
		u \in C^\infty(\TT),
  \end{equation}
  and for all $k, l \in \NN_0$ there exist constants $c_{k, l}>0$ such that
  \begin{equation*}
    \left| \frac{\partial^k}{\partial t^k} \,\omega_n^l h(t, n) \right| \leq c_{k,l}
		\,\underline{n}^{\alpha - l} \quad \text{for all } n \in \mathbb{Z};
  \end{equation*}
	where $\omega_n$ means the application of $\omega$ to the second argument of $h$.
	The class of all periodic pseudodifferential operators of order $\alpha$ is denoted by $\Psi^\alpha$, and we set
	\begin{equation*}
	\Psi^{-\infty}:=\bigcap_{\alpha\in\RR}\Psi^\alpha.
	\end{equation*}
\end{definition}
	
We note that one has the obvious inclusions $\Psi^\alpha\subset\Psi^\beta$ for $\alpha<\beta$.
Moreover, in the spirit of~\eqref{equation_Fourier} the periodic pseudodifferential operator $H$ is determined by its Fourier coefficients
\begin{equation*}
  \widehat{H u}(m) = \sum_{n \in \mathbb{Z} } \widehat{u}(n) \big\langle h(\cdot,n) e_n, e_{-m} \big\rangle_{\mathcal{D}'(\mathbb{T}),\mathcal{D}(\mathbb{T})}.
\end{equation*}
In particular, if $h$ is independent of $t$, then we simply have $\widehat{H u}(n) = h(n) \widehat{u}(n)$.
The following properties of periodic pseudodifferential operators can be found in \cite[Theorem~7.3.1 and Theorem~7.8.1]{SV}.

\begin{prop} \label{proposition_properties_PPDO}
 
  \begin{itemize}
    \item[\textup{(i)}] 
		 Let  $H \in \Psi^\alpha$. Then for any $s\in\RR$ the operator $H$
		uniquely extends by continuity to a bounded operator $H^s(\TT) \rightarrow H^{s-\alpha}(\TT)$; this extension will be denoted by the same symbol $H$.
    \item[\textup{(ii)}] For any  $H\in\Psi^\alpha$ and $G\in\Psi^\beta$ one has
		\[
		H + G \in \Psi^{\max\{ \alpha, \beta\} }, \quad H G\in\Psi^{\alpha+\beta}, \quad HG-GH\in\Psi^{\alpha+\beta-1}.
		\]
  \end{itemize}
\end{prop}

Having the definition of periodic pseudodifferential operators on $\mathbb{T}$ and the bijective map $U$ in~\eqref{def_U}  at hand,
it is now straightforward to define periodic pseudodifferential operators on the loop $\Sigma$.

\begin{definition} \label{definition_PPDO_Sigma}
  A linear map $H:C^\infty(\Sigma)\to \cD'(\Sigma)$ is called
  a periodic pseudodifferential operator of order $\alpha\in\RR$ on $\Sigma$, if there exists
  a periodic pseudodifferential operator $H_0$ of order $\alpha$ on $\TT$ such that
  \begin{equation*}
    H=U^{-1}H_0 U.
  \end{equation*}
  We denote by $\Psi_\Sigma^\alpha$ the linear space of all
  periodic pseudodifferential operators of order $\alpha\in\RR$ on $\Sigma$ and set
  \begin{equation*}
    \Psi_\Sigma^{-\infty}:=\bigcap_{\alpha\in\RR} \Psi^\alpha_\Sigma.
  \end{equation*}
\end{definition}

In view of Proposition~\ref{proposition_properties_PPDO} and the fact that $U$ in \eqref{uunitary} is unitary 
it is clear that each $H \in \Psi_\Sigma^\alpha$ induces a unique bounded operator $H: H^s(\Sigma) \rightarrow H^{s-\alpha}(\Sigma)$.

In what follows we discuss several special periodic pseudodifferential operators and their mapping properties which will play an important 
role in the analysis in the main part of this paper. First, let $c_0 > 0$ be a constant and consider the operator
\begin{equation}\label{def_L}
L^\alpha u(t) = \sum_{n \in \ZZ}\, \bigl(c_0^2 + |n|\bigr)^{\frac{\alpha}{2}} \,\widehat{u}(n) \,\,e_n(t),\quad
		u \in C^\infty(\TT),\quad\alpha\in\RR,
\end{equation}
on $C^\infty(\TT)$. Note that the Fourier coefficients of $L^\alpha u$ are $\widehat{L^\alpha u}(n) = (c_0^2 + |n|)^{\frac{\alpha}{2}} \,\widehat{u}(n)$ for $n\in\ZZ$.
One can show that $L^\alpha \in \Psi^{\frac{\alpha}{2}}$ and hence $L^\alpha$ induces an isomorphism from $H^{s}(\TT)$ to $H^{s-\frac{\alpha}{2}}(\TT)$ for any $s\in\RR$. 
The operator $L=L^1$ will be of particular importance in the following.

Using the operator $U$ from \eqref{def_U} we introduce 
\begin{equation} \label{def_Lambda}
 \Lambda^\alpha:= U^{-1} L^\alpha U \in \Psi_\Sigma^{\frac{\alpha}{2}},\quad\alpha\in\RR,
\end{equation}
and conclude that $\Lambda^\alpha:H^s(\Sigma)\rightarrow H^{s-\frac{\alpha}{2}}(\Sigma)$ is an isomorphism for any 
$\alpha,s \in \mathbb{R}$. Moreover, the above definition of $\Lambda$ implies 
that $\Lambda^\alpha \Lambda^\beta = \Lambda^{\alpha + \beta}$ for all $\alpha, \beta \in \mathbb{R}$. 
We note that the realization of $\Lambda=\Lambda^1$ for $s=\frac{1}{2}$ is viewed as an unbounded 
self-adjoint operator in $L^2(\Sigma)$ satisfying $\Lambda \geq c_0$. 
In particular, by varying $c_0$ we get that $\Lambda$ is a uniformly positive operator and that its lower bound can be arbitrarily large.

With the aid of $\Lambda$ we can prove now the following lemma.

\begin{lem}\label{sympsi}
For $H\in \Psi^\alpha_\Sigma$ consider the induced linear operator in $L^2(\Sigma)$ defined by
\begin{gather*}
H_\infty u=H u, \quad \dom H_\infty=C^\infty(\Sigma),
\end{gather*}
and assume that $H_\infty$ is symmetric. Then the adjoint $H_\infty^*$ is given by 
\begin{equation*}
  H_\infty^*f=Hf, \quad \dom H_\infty^*=\big\{f\in L^2(\Sigma): \, Hf\in L^2(\Sigma)\big\}.
\end{equation*}
\end{lem}

\begin{proof} 
The result is trivial for $\alpha\le 0$ due to the boundedness of $H_\infty$; cf. Proposition~\ref{proposition_properties_PPDO}. 
Hence, we may assume that $\alpha>0$.
Recall that $f \in \dom H_\infty^*$ if and only if the mapping
\begin{equation} \label{def_adjoint}
  C^\infty(\Sigma) \ni u \mapsto (H_\infty u, f)_{L^2(\Sigma)}
\end{equation}
can be extended to a bounded functional on $L^2(\Sigma)$.

Let $f\in L^2(\Sigma)$ and $f_n\in C^\infty(\Sigma)$ such that $f_n\to f$ in $L^2(\Sigma)$.
For $u\in C^\infty(\Sigma)$ and the map $U$ from~\eqref{def_U}-\eqref{uunitary} one has
\begin{equation*}
\begin{split}
 ( H_\infty u, f)_{L^2(\Sigma)}&=\lim_{n\to\infty} (H_\infty u, f_n )_{L^2(\Sigma)}=\lim_{n\to\infty} (u,  H f_n)_{L^2(\Sigma)} =\lim_{n\to\infty} \ell (U u, U H f_n)_{L^2(\mathbb{T})}\\
&=\lim_{n\to\infty} \ell ( L^{2\alpha} U u,  L^{-2\alpha} U H f_n)_{L^2(\mathbb{T})}=\ell ( L^{2\alpha} U u, L^{-2\alpha} U H f)_{L^2(\TT)},
\end{split}
\end{equation*}
where we have used in the last step that $L^{-2\alpha} U H=L^{-2\alpha} U HU^{-1} U$ gives rise to a bounded operator 
from $L^2(\Sigma)\rightarrow L^2(\mathbb{T})$ due to $L^{-2\alpha}\in\Psi^{-\alpha}$, $U HU^{-1} \in\Psi^{\alpha}$, and 
Proposition~\ref{proposition_properties_PPDO}. Therefore, if $f \in L^2(\Sigma)$ is such that $H f \in L^2(\Sigma)$, then 
$$\ell ( L^{2\alpha} U u, L^{-2\alpha} U H f)_{L^2(\TT)}=\ell ( U u, U H f)_{L^2(\TT)}=(u, H f)_{L^2(\Sigma)}$$
and
the functional 
in~\eqref{def_adjoint} is bounded,
\begin{equation*}
\vert (H_\infty u, f)_{L^2(\Sigma)}\vert =  \vert  (  u,  H f)_{L^2(\Sigma)}\vert
\leq \Vert u\Vert_{L^2(\Sigma)}\Vert H f\Vert_{L^2(\Sigma)},
\end{equation*}
and hence, $f \in \dom H_\infty^*$ and $H_\infty^* f = H f$.

On the other hand, for $f\in\dom H_\infty^*$ and every $u\in C^\infty(\Sigma)$ the functional in~\eqref{def_adjoint} is bounded. For the special choice
\begin{equation*}
 u_k = \sum_{|n| \leq k} \widehat{U H f}(n) U^{-1} e_n \in C^\infty(\Sigma),\quad k\in\NN,
\end{equation*}
one has $\widehat{U u_k}(n) = \widehat{U H f}(n)$ for $ |n| \leq k$ and $\widehat{U u_k}(n)=0$ for $|n|>k$, and hence
\begin{equation*}
\begin{split}
 ( H_\infty u_k, f)_{L^2(\Sigma)}&=\ell ( L^{2\alpha} U u_k, L^{-2\alpha} U H f)_{L^2(\TT)}\\
 &=\ell\sum_{n\in\ZZ} \widehat{L^{2\alpha} U u_k}(n) \overline{\widehat{L^{-2\alpha} U H f}(n)} \\
&=\ell\sum_{n\in\ZZ} (c_0^2+|n|)^\alpha\widehat{U u_k}(n) \overline{(c_0^2+|n|)^{-\alpha} \widehat{U H f}(n)}\\
&=\ell\sum_{|n| \leq k} \bigl\vert \widehat{U H f}(n)\bigr\vert^2.
\end{split}
\end{equation*}
Sending $k \rightarrow \infty$ we see that a necessary condition for the functional in~\eqref{def_adjoint} to be bounded on $L^2(\Sigma)$ is given by
\begin{equation*}
  \sum_{n\in\ZZ} \big|\widehat{U H f}(n) \big|^2 < \infty,
\end{equation*}
i.e. $U H f\in L^2(\TT)$, and hence
$H f \in L^2(\Sigma)$. We have shown that $f \in \dom H_\infty^*$ if and only if $H f \in L^2(\Sigma)$, which finishes the proof.
\end{proof}

Next, we discuss that several types of integral operators on $\mathbb{T}$ are in fact periodic pseudodifferential operators, which allows us to deduce their mapping properties from the general theory. Note that via the isomorphism $U$ from \eqref{def_U} the results can be translated to integral operators on $\Sigma$. To formulate the following first result, recall the definition of the map $\omega$ from~\eqref{difference_op}; the proof of this proposition can be found in \cite[Theorem~7.6.1]{SV}.

\begin{prop} \label{proposition_integral_op_PPDO}
  Let $\alpha \in \mathbb{R}$ and 
	$\kappa \in \cD'(\mathbb{T})$ such that for any $j \in \NN_0$
	there exists $c_j > 0$ with $\big|\omega^j \widehat{\kappa}(n)\big| \leq c_j \underline{n}^{\alpha-j}$ for all $n \in \ZZ$. Let $h \in C^\infty(\TT^2)$ and 
	the operator $H$ be defined on $C^\infty(\TT)$ by
  \begin{equation}
	     \label{eq-AA}
    (H u)(t) := \Big(\kappa \star \big(h(t,\cdot) u\big)\Big) (t), \quad u \in C^\infty(\mathbb{T}).
  \end{equation}
	Then $H\in \Psi^\alpha$.
	%
\end{prop}

We remark that for $\kappa\in L ^1(\TT)$ the operator $H$ in \eqref{eq-AA} is an integral operator,
\begin{equation*}
 (H u)(t) := \int_\TT \kappa(t - s) h(t,s) u(s)\dd s, \quad u \in C^\infty(\mathbb{T}).
\end{equation*}
As a corollary we obtain:

\begin{cor} \label{corollary_smooth_PPDO}
  Let $h \in C^\infty(\TT^2)$. Then the integral operator acting as
  \begin{equation*}
    H u(t) := \int_\TT h(t, s) u(s) \dd s, \quad u \in C^\infty(\TT),
  \end{equation*}
  belongs to $\Psi^{-\infty}$.
\end{cor}

In the following proposition we discuss a class of integral operators that appear quite frequently in our applications. 

\begin{prop}\label{prop15}
Let $m\in\NN_0$, let
\begin{equation*}
  a:\TT^2\to \CC \quad \text{and} \quad \rho:\TT\to  \CC
\end{equation*}
be $C^\infty$-functions, where $\rho$ is injective with $\rho'(t) \neq 0$ for all $t \in \mathbb{T}$.
Set
\[
\kappa_m(z) := z^m \log |z| \text{  for $z \in \CC\setminus\{0\}$}
\]
and define the integral operator
  \begin{equation*}
    H_m u(t) := \int_\mathbb{T} \kappa_m\big(\rho(t)-\rho(s)\big) \,a(t,s)\, u(s) \,\dd s,
		\quad u \in C^\infty(\TT).
  \end{equation*}
	Then $H_m \in \Psi^{-m-1}$. Furthermore, in the special case $a\equiv 1$ and $m = 0$ one has
	\begin{equation}
	    \label{eql}
	\one + 2 L H_0 L \in\Psi^{-1},
	\end{equation} 
	where the operator $L$ is defined by~\eqref{def_L}.
\end{prop}
\begin{proof}
First, we treat the case $m = 0$. For that we introduce the auxiliary function 
$\chi_0: \mathbb{T} \rightarrow \mathbb{R}$ by $\chi_0(t) := \log \big|\sin(\pi t)\big|$. Then the Fourier coefficients of $\chi_0$ are
\begin{equation}
   \label{klog}
\widehat{\chi_0}(n) = \begin{cases} -\log 2, & n=0,\\ -\dfrac{1}{2 |n|}, & n \neq 0, \end{cases}
\end{equation}
see \cite[Example~5.6.1]{SV}.
Next, one has
  \begin{equation}
	   \label{eq-a0}
							\log\Big(\big|\rho(t)-\rho(s)\big|\Big)
					= \log \Big|\sin \big(\pi (t-s)\big) \Big|
					+ a_0(t,s)
  \end{equation}
  with 
  \begin{equation*}
	a_0(t,s)=\log\left(\left| \frac{\rho(t)-\rho(s)}{\sin (\pi (t-s))} \right|\right),\,\,\,t\not=s,\quad\text{and}\quad 
	a_0(t,t)=\log\left(\frac{\vert \rho'(t)\vert }{\pi} \right).
  \end{equation*}
  Using Taylor series expansions one sees that there exist smooth functions $f_1$ and $f_2$ such that
  \begin{equation*}
   \frac{1}{\sin (\pi (t-s))}= \frac{1}{\pi (t-s)}f_1(t,s)\quad\text{and}\quad \rho(t)-\rho(s)=(t-s) f_2(t,s),
  \end{equation*}
   and since $\rho$ is injective, we have $\big(\rho(t)-\rho(s)\big)/\sin \big(\pi (t-s)\big) \neq 0$.
  One concludes that $a_0:\TT^2\to \CC$ is a $C^\infty$-function.
	Now we decompose $H_0= C_0 + D_0$, where
	\begin{align*}
	C_0 u(t)&=\int_\TT \chi_0(t-s)\, a(t,s)\, u(s)\dd s = (\chi_0 \star (a(t, \cdot) u))(t),\\
	D_0 u(t)&=\int_\TT a_0(t,s) \,a(t,s)\,u(s)\, \dd s.
	\end{align*}
	It follows from~\eqref{klog} and Proposition~\ref{proposition_integral_op_PPDO} that 
	$C_0\in \Psi^{-1}$ and by Corollary~\ref{corollary_smooth_PPDO} we have $D_0\in \Psi^{-\infty}$. Hence $H_0\in \Psi^{-1}$ 
	by Proposition~\ref{proposition_properties_PPDO}.
	
	To show \eqref{eql} consider $L H_0 L = L C_0 L + L D_0 L$ and note that the second term in the sum belongs to $\Psi^{-\infty}$.
	Furthermore, for $a \equiv 1$ the Fourier coefficients of $C_0Lu$ are given by
	$$
	\widehat{C_0 L u}(n)=\widehat{\chi_0}(n)\widehat{L u}(n)=\widehat{\chi_0}(n)(c_0^2 + |n|)^{\frac{1}{2}} \widehat{u}(n),
	$$
	and hence
	one finds with the aid of~\eqref{klog}
	\begin{equation*}
	\widehat{L C_0 L u}(n)=\big(c_0^2 + |n|\big)^{\frac{1}{2}} \widehat{\chi_0}(n) (c_0^2 + |n|)^{\frac{1}{2}} \widehat{u}(n) = b(n) \widehat{u}(n)
	\end{equation*}
	with
	\begin{equation*}
	  b(n)= \big(c_0^2 + |n|\big)\widehat{\chi_0}(n)=\begin{cases} -c_0^2 \log 2, & n=0,\\[\smallskipamount] -\dfrac{1}{2} - \dfrac{c_0^2}{2 |n|}, & n \neq 0, \end{cases}
	\end{equation*}
	which shows that the action of the operator $K:=\one+2L C_0 L$ is determined by
	\begin{equation*}
	  \widehat{K u}(n) = k(n) \widehat{u}(n) \quad \text{with} \quad k(n) = \begin{cases} 1 - 2 c_0^2 \log 2, & n=0,\\ - \dfrac{c_0^2}{|n|}, & n \neq 0. \end{cases}
	\end{equation*}
	Therefore, one can show with the help of Proposition~\ref{proposition_integral_op_PPDO} that $K \in \Psi^{-1}$.
	
To study the case $m\ge 1$ we consider 
\[
\rho(t)-\rho(s)=\big(e^{-2\pi\rmi(t-s)}-1\big)\, a_1(t,s) 
\]
with 
$$
a_1(t,s)=\dfrac{\rho(t)-\rho(s)}{e^{-2\pi\rmi(t-s)}-1},\,\,\,t\not=s,\quad\text{and}\quad 
	a_1(t,t)=\frac{\rho'(t) }{-2\pi \rmi}
$$
and note, as for $a_0$, that $a_1\in C^\infty(\TT^2)$. Then using the decomposition \eqref{eq-a0} we write
\begin{multline*}
(\rho(t)-\rho(s))^m \log (|\rho(t)-\rho(s)|)\\
= \big(e^{-2\pi\rmi(t-s)}-1\big)^m\, \log (|\sin (\pi (t-s)) |) a_1(t,s)^m\\
+\big(e^{-2\pi\rmi(t-s)}-1\big)^m a_0(t,s)a_1(t,s)^m.
\end{multline*}
This shows that $H_m=C_m+D_m$, where $C_m$ and $D_m$ are integral operators
\begin{align*}
C_m u(t)&=\int_\TT \big(e^{-2\pi\rmi(t-s)}-1\big)^m\, \log (|\sin (\pi (t-s)) |) a_1(t,s)^m\, a(t,s) u(s)\dd s,\\
D_m u(t)&=\int_\TT \big(e^{-2\pi\rmi(t-s)}-1\big)^m a_0(t,s)a_1(t,s)^ma(t,s)\, u(s)\dd s.
\end{align*}
The integral kernel of $D_m$ is smooth, which implies by Corollary~\ref{corollary_smooth_PPDO} that $D_m\in\Psi^{-\infty}$. It is remains to show
that $C_m\in \Psi^{-(m+1)}$. For that consider the function
\[
\chi_m: \mathbb{T} \rightarrow \mathbb{C}, \quad \chi_m(t) := \big(e^{-2\pi\rmi t}-1\big)^m\log \Big(\big|\sin(\pi t)\big|\Big).
\]
Using  the map $\omega$ from \eqref{difference_op} and $\chi_0$
one obtains that $\widehat{\chi_m}(n)=\big(\omega^m \widehat{\chi_0}\big)(n)$.
With the help of \eqref{klog} it follows that 
\begin{equation*}
  |\omega^j \widehat{\chi_m}(n)| = |\omega^{m+j} \widehat{\chi_0}(n)| \leq c_j \underline{n}^{-m-1-j}.
\end{equation*}
By Proposition~\ref{proposition_integral_op_PPDO} this yields $C_m\in \Psi^{-(m+1)}$, which completes the proof of this proposition.
\end{proof}

Next, recall that the \emph{Hilbert transform} $T_0$ on $\TT$ is defined by
\begin{equation} \label{def_Hilbert_transform}
  T_0 u(t) := \rmi\, \pv \int_\TT \cot \big(\pi (t-s) \big) u(s) \textup{d} s = (\kappa \star u )(t),
	\quad \kappa = \rmi \pv \cot(\pi \cdot),
\end{equation}
where $\pv$ means the principal value of the integral.
By \cite[Section~5.7]{SV} the distribution $\kappa$ satisfies
\begin{equation*}
  \widehat{\kappa}(n) = \sign n = \begin{cases} -1, & n < 0, \\ 0, &n=0, \\ 1, & n>0. \end{cases}
\end{equation*}
It follows that $\widehat {T_0^2 u}(n) = (1-\delta_{0,n})\widehat u(n)$, and
\begin{equation}
  \label{ht0}
	T_0\in \Psi^0, \quad T_0^2-\one\in \Psi^{-\infty}.
\end{equation}
In the following assume that $a\in C^\infty(\TT^2)$. Then the operator
\begin{equation*}
  (T_1 u)(t)=\rmi\, \pv \int_\TT \cot \big(\pi (t-s) \big) \,a(s,t)\,u(s) \dd s
\end{equation*}
satisfies for $a_0(t):=a(t,t)$ the relation
\begin{equation}
  \label{tt0}
T_1-a_0 T_0\in\Psi^{-\infty},
\end{equation}
see Section 7.6.2 in \cite{SV}.
Since the commutator $T_2:=a_0 T_0-T_0 a_0$, which acts as
\begin{equation*}
T_2 u(t)=\rmi\, \pv \int_\TT \cot \big(\pi (t-s) \big) \,\big(a(t,t)-a(s,s)\big)\,u(s) \dd s,
\end{equation*}
has a $C^\infty$-smooth integral kernel, the principal value can be dropped, as the integral is convergent, 
and Corollary~\ref{corollary_smooth_PPDO} implies that $T_2\in \Psi^{-\infty}$. Hence,
we also have
\begin{equation}
  \label{tt1}
T- T_0a_0\in\Psi^{-\infty}.
\end{equation}

\begin{cor} \label{corollary_Hilbert_transform}
Let $\rho:\TT\to\CC$ be $C^\infty$-smooth and injective with $\rho'(t)\ne 0$ for all $t\in\TT$.
Then the operator $C$ given by
  \begin{equation*}
    C u(t) = \dfrac{\rmi}{\pi}\,\pv \int_\mathbb{T} \frac{u(s)}{\rho(t)-\rho(s)} \dd s, \quad u \in C^\infty(\mathbb{T}),
  \end{equation*}
	satisfies
	\begin{equation}
	  \label{th1}
	C -\dfrac{1}{\rho'} \,T_0 \in \Psi^{-\infty} \qquad \text{and} \qquad
	  	C-T_0 \,\dfrac{1}{\rho'}\in \Psi^{-\infty}.
  \end{equation}
\end{cor}
\begin{proof}
  We write
	\[
	\dfrac{1}{\pi}\,\dfrac{1}{\rho(t)-\rho(s)}=\cot \big(\pi(t-s)\big) a(t,s) \quad \text{with} \quad
	a(t,s)=\dfrac{1}{\pi}\,\dfrac{\tan \big(\pi(t-s)\big)}{\rho(t)-\rho(s)},\quad t\not=s,
	\]
	and $a(t,t)=1/\rho'(t)$.
Then $a\in C^\infty(\TT^2)$ and $a_0(t)=a(t,t)=1/\rho'(t)$.
	Thus \eqref{th1} follows from \eqref{tt0} and \eqref{tt1}.
\end{proof}

Finally we introduce the \emph{Cauchy transform $C_\Sigma$ on $\Sigma$}. For that we identify $\mathbb{R}^2$ 
with $\mathbb{C}$ and use the notation
\begin{equation*}
 \begin{split}
  \mathbb{R}^2 \ni x&=(x_1, x_2) \sim x_1 + \rmi x_2 =: \xi \in \mathbb{C},\\
  \mathbb{R}^2 \ni y&=(y_1, y_2) \sim y_1 + \rmi y_2 =: \zeta \in \mathbb{C}.
 \end{split}
\end{equation*} 
Then
\begin{equation} \label{def_Cauchy_transform}
  C_\Sigma u(\xi) := \frac{\rmi}{\pi} \pv \int_\Sigma \frac{u(\zeta)}{\xi - \zeta} \text{d} \zeta, 
  \quad u \in C^\infty(\Sigma),\, \xi \in \Sigma,
\end{equation}
where the complex line integral is understood as its principal value. With 
an arc-length parametrization $\gamma$ of $\Sigma$ and $x = \gamma(t), y=\gamma(s)$  it follows that $C_\Sigma$ acts as
\begin{equation*}
  C_\Sigma u\big(\gamma(t)\big) = \frac{\rmi}{\pi} \pv \int_0^\ell \frac{\big(\gamma_1'(s) + \rmi \gamma_2'(s)\big) u\big(\gamma(s)\big)}{\big(\gamma_1(t) + \rmi \gamma_2(t)\big) - \big(\gamma_1(s) + \rmi \gamma_2(s)\big)} \dd s. 
\end{equation*}
Recall that for the tangent vector field $\bf t$ at $\Sigma$ and 
$y = \gamma(s) \in \Sigma$ we use the notation $T(y) := {\bf t}_1(y) + \rmi {\bf t}_2(y) = \gamma_1'(s) + \rmi \gamma_2'(s)$. We shall also view
$y\mapsto T(y)$ as a function on $\Sigma$ or $s\mapsto T(\gamma(s))$ as a function on $[0,\ell]$. The same holds for the function $\overline T(y) := {\bf t}_1(y) - \rmi {\bf t}_2(y) = 
\gamma_1'(s) - \rmi \gamma_2'(s)$, and we will also denote the corresponding multiplication operators by $T$ and $\overline T$.
With this we see for $u \in C^\infty(\Sigma)$ and $x = \gamma(t) \in \Sigma$ that
\begin{equation} \label{relation_Cauchy_transform}
  \begin{split}
    (C_\Sigma \overline{T} u)(x) &= \frac{\rmi}{\pi} \pv \int_0^\ell \frac{\big(\gamma_1'(s) + \rmi \gamma_2'(s)\big) \big(\gamma_1'(s) - \rmi \gamma_2'(s)\big) u\big(\gamma(s)\big)}{\big(\gamma_1(t) + \rmi \gamma_2(t)\big) - \big(\gamma_1(s) + \rmi \gamma_2(s)\big)} \dd s \\
    &= \frac{\rmi}{\pi} \pv \int_\Sigma \frac{u(y)}{(x_1 + \rmi x_2) - (y_1 + \rmi y_2)} \dd s(y).
  \end{split}
\end{equation}

In our considerations also the formal dual $C_\Sigma'$ of $C_\Sigma$ in $L^2(\Sigma)$, which acts as
\begin{equation} \label{def_Cauchy_transform_dual}
  \begin{split}
    C_\Sigma' u(\gamma(t)) 
    =  \frac{\rmi}{\pi} \pv \int_0^\ell \frac{\big(\gamma_1'(t) - \rmi \gamma_2'(t)\big) u\big(\gamma(s)\big)}{\big(\gamma_1(t) - \rmi \gamma_2(t)\big) - \big(\gamma_1(s) - \rmi \gamma_2(s)\big)} \dd s 
  \end{split}
\end{equation}
for $u \in C^\infty(\Sigma)$ and $x = \gamma(t) \in \Sigma$ will play an important role. 
Note that $C_\Sigma'$ is the operator which satisfies $ (C_\Sigma u, v)_{L^2(\Sigma)} = (u, C_\Sigma'  v)_{L^2(\Sigma)}$ for all $u, v \in C^\infty(\Sigma)$. 
Similarly as in~\eqref{relation_Cauchy_transform} we have
\begin{equation} \label{relation_dual_Cauchy_transform}
  \begin{split}
    (T C_\Sigma' u)(x) &= \frac{\rmi}{\pi} \pv \int_0^\ell \frac{\big(\gamma_1'(t) + \rmi \gamma_2'(t)\big) \big(\gamma_1'(t) - \rmi \gamma_2'(t)\big) u\big(\gamma(s)\big)}{\big(\gamma_1(t) - \rmi \gamma_2(t)\big) - \big(\gamma_1(s) - \rmi \gamma_2(s)\big)} \dd s \\
    &= \frac{`rmi}{\pi} \pv \int_\Sigma \frac{u(y)}{(x_1 - \rmi x_2) - (y_1 - \rmi y_2)} \dd s(y).
  \end{split}
\end{equation}

In the following proposition we summarize the basic properties of $C_\Sigma$ and $C_\Sigma'$ which are needed for our further considerations. They basically follow directly from~\eqref{relation_Cauchy_transform}, \eqref{relation_dual_Cauchy_transform}, Corollary~\ref{corollary_Hilbert_transform}, and~\eqref{ht0}.

\begin{prop} \label{proposition_Cauchy_transform}
  Let $C_\Sigma$ and $C_\Sigma'$ be defined by~\eqref{def_Cauchy_transform} and~\eqref{def_Cauchy_transform_dual}, let $U$ be given by~\eqref{def_U}, and let the Hilbert transform $T_0$ be defined by~\eqref{def_Hilbert_transform}. Then the following is true:
  \begin{itemize}
    \item[\textup{(i)}] $C_\Sigma - U^{-1} T_0 U \in \Psi_\Sigma^{-\infty}$. In particular, $C_\Sigma \in \Psi_\Sigma^0$ and for all $s \in \mathbb{R}$ the operator $C_\Sigma$ gives rise to a bounded operator in $H^s(\Sigma)$.
    \item[\textup{(ii)}] $C_\Sigma' - U^{-1} T_0 U \in \Psi_\Sigma^{-\infty}$. In particular, $C_\Sigma' \in \Psi_\Sigma^0$ and for all $s \in \mathbb{R}$ the operator $C_\Sigma'$ gives rise to a bounded operator in $H^s(\Sigma)$.
  \end{itemize} 
  Furthermore, one has $C_\Sigma' C_\Sigma - \one\in \Psi_\Sigma^{-\infty}$ and $C_\Sigma C_\Sigma' - \one \in \Psi_\Sigma^{-\infty}$.
\end{prop}

\begin{proof} Let us prove \textup{(i)}. Note first that the multiplication operators $T$ and $\overline T$ that multiply with the 
functions $s\mapsto T(\gamma(s)) = \gamma_1'(s) + \rmi \gamma_2'(s)$ and 
$s\mapsto \overline{T}(\gamma(s))= \gamma_1'(s) - \rmi \gamma_2'(s)$  belong to $\Psi_\Sigma^0$, see \cite[Section~7.2]{SV}. Hence \textup{(i)} is equivalent to
\[
C_\Sigma\overline{T} - U^{-1}T_0 U \overline{T} = C_\Sigma\overline{T} - U^{-1}T_0 \overline{T}\big(\gamma( \ell\cdot)\big)U \in \Psi_\Sigma^{-\infty},
\]
which in turn is equivalent, by definition, to
\[
	U C_\Sigma \overline{T} U^{-1} - T_0 \overline{T}(\gamma(\ell \cdot)) \in \Psi^{-\infty}.
\]
For $v\in C^\infty(\mathbb{T})$ and $t\in \mathbb{T}$, we compute $\big(U C_\Sigma \overline{T} U^{-1}v \big) (t)$. Remark that for $x =(x_1,x_2) \in \Sigma$ and $w(x) := (U^{-1} v) (x)$, \eqref{ul} and~\eqref{relation_Cauchy_transform} give
\begin{align*}
	(C_\Sigma\overline{T}w)(x) & =  \frac{\rmi}{\pi} \pv \int_0^\ell \frac{w(\gamma(s))}{(x_1 + \rmi x_2) - \big(\gamma_1(s) + \rmi\gamma_2(s)\big)}\dd s\\
	& = \frac{\rmi}{\pi} \pv \int_0^\ell \frac{v(\ell^{-1}s)}{(x_1 + \rmi x_2) - \big(\gamma_1(s) + \rmi\gamma_2(s)\big)}\dd s.
\end{align*}
Hence, a change of variable yields
\[
	(U C_\Sigma \overline{T} U^{-1}v)(t) = \ell \frac{\rmi}{\pi}\pv\int_{\mathbb{T}} \frac{v(s)}{\rho(t) - \rho(s)}\dd s
\]
with $\rho(t) := \gamma_1 (\ell t) + \rmi \gamma_2 (\ell t)$. Remark that for all $t\in\mathbb{T}$ we have $\rho'(t) = \ell T\big(\gamma(\ell t)\big) \neq 0$ and $1/\rho'(t) = \ell^{-1}\overline{T}\big(\gamma(\ell t)\big)$. Corollary \ref{corollary_Hilbert_transform} gives
\[
	\ell^{-1} U C_\Sigma \overline{T} U^{-1} -\ell^{-1} T_0 \overline{T}(\ell \cdot) \in \Psi^{-\infty}
\]
which completes the proof of \textup{(i)}. Item \textup{(ii)} is proved in a similar fashion and the last statement is a consequence of \textup{(i)}, \textup{(ii)}, 
and \eqref{ht0}. This can be seen by the equivalences
\[
T_0^2 - \one \in \Psi^{-\infty} \quad \text{iff} \quad U C_\Sigma'U^{-1} U C_\Sigma U^{-1} - \one \in \Psi^{-\infty} 
\quad \text{iff}\quad C_\Sigma'C_\Sigma - \one \in \Psi_\Sigma^{-\infty},
\]
and a similar argument shows $C_\Sigma C_\Sigma' - \one \in \Psi_\Sigma^{-\infty}$.
This completes the proof.
\end{proof}

\subsection{Schur complement of block operators}\label{ssec-schur}

Let $W_{jk}$, $j,k\in\{1,2\}$, be  closable densely defined operators in a Hilbert space $\cH$.
Define a linear operator $W$ in $\cH\oplus\cH$ by
\begin{gather*}
      W := \begin{pmatrix} W_{11} & W_{12} \\ W_{21} & W_{22}\end{pmatrix}, \quad
    \dom W = (\dom W_{11} \cap \dom W_{21}) \oplus (\dom W_{12} \cap \dom W_{22}).
  \end{gather*}
Assume that $\dom W_{11} \subset \dom W_{21}$ and that $W_{11}$ is invertible. Then
one can define
the Schur complement $\mathcal{S}(W)$ of  $W$ as an operator in $\cH$ by
\begin{equation} \label{def_Schur_complement}
  \mathcal{S}(W) := W_{22} - W_{21} W_{11}^{-1} W_{12},
\end{equation}
and one has the factorization
  \begin{equation} \label{schur-fact}
    W = \begin{pmatrix} \one & 0 \\ W_{21} W_{11}^{-1} & \one \end{pmatrix}
		\begin{pmatrix} W_{11} & 0 \\ 0 & \mathcal{S}(W) \end{pmatrix}
		\begin{pmatrix} \one & W_{11}^{-1} W_{12} \\ 0 & \one \end{pmatrix}.
  \end{equation}
We will use the following facts, which follow from Theorem 2.2.14 and Theorem~2.4.6 in the monograph~\cite{Tre}.
\begin{prop}\label{prop-block}
Assume that $0\in \res W_{11}$, that $\dom W_{11} \subset \dom W_{21}$,
and that $W_{11}^{-1} W_{12}$ is  bounded on $\dom W_{12}$.
Then $W$ is closable/closed if and only if its Schur complement $\mathcal{S}(W)$ is closable/closed,
with
  \begin{equation*} 
    \overline{W} = \begin{pmatrix} \one & 0 \\ W_{21} W_{11}^{-1} & \one \end{pmatrix} \begin{pmatrix} W_{11} & 0 \\ 0 & \overline{\mathcal{S}(W)} \end{pmatrix} \begin{pmatrix} \one & \overline{W_{11}^{-1} W_{12}}\, \\ 0 & \one \end{pmatrix},
  \end{equation*}
	and
	\[
	\dom \overline{W}=\left\{
	(x_1, x_2)\in\cH\times\cH:
	x_1+ \overline{W_{11}^{-1} W_{12}} \,x_2\in \dom W_{11},
	\ x_2\in \dom \overline{\mathcal{S}(W)}
	\right\}.
	\]
	Moreover, if $\overline{W}$ is self-adjoint, then $0\in\spec_\ess \overline{W}$ if and only if $0\in \spec_{\textup{ess}} \overline{\mathcal{S}(W)}$.
\end{prop}

\subsection{Boundary triples and their Weyl functions}\label{ssec-bt}

We recall some basic facts about boundary triples following the first chapter of the paper \cite{BGP}, 
in which the proofs for all statements can be found. We also refer the reader to \cite{DM91,DM95} and the monographs \cite{BHS19,DM19} for more details
and applications. 
Throughout this abstract section $\mathcal{H}$ is always a separable Hilbert space.

\begin{definition}
  Let $S$ be a closed densely defined symmetric operator in $\cH$.
  A \emph{boundary triple} for $S^*$ is a triple $\{\cG,\Gamma_0, \Gamma_1\}$
  consisting of a Hilbert space $\cG$ and two linear maps $\Gamma_0,\Gamma_1: \dom S^* \rightarrow \cG$
  satisfying the following two conditions:
  \begin{itemize}
    \item[\textup{(i)}] For all $f, g \in \dom S^*$ there holds
		\begin{equation*}
		( S^* f, g )_\cH - ( f, S^* g )_\cH = (\Gamma_1 f, \Gamma_0 g)_\cG - (\Gamma_0 f, \Gamma_1 g)_\cG.
        \end{equation*}
 		\item[\textup{(ii)}] The map $\dom S^*\ni f\mapsto (\Gamma_0 f, \Gamma_1 f )\in \cG\times\cG$ is surjective. 
  \end{itemize}
\end{definition}
A boundary triple for $S^*$ exists if and only if  $S$ admits self-adjoint extensions in $\cH$. From now on we assume that this is satisfied and pick
a boundary triple $\{\cG,\Gamma_0,\Gamma_1\}$. This induces a number of additional objects. First, the operator
\[
B_0 := S^* \upharpoonright \ker \Gamma_0
\]
is self-adjoint, and for any $z\in \res B_0$ one has the direct sum decomposition
\begin{equation}\label{deco33}
  \dom S^* = \dom B_0 \dsum \ker (S^* - z) = \ker \Gamma_0 \dsum \ker (S^* - z),
\end{equation}
showing that $\Gamma_0 \upharpoonright \ker (S^*-z)$ is bijective. This allows to define the $\gamma$-field $G$ and the Weyl function $M$
associated to $\{\cG,\Gamma_0,\Gamma_1\}$ by
\begin{align*}
  \res B_0 \ni z &\mapsto G_z := \big(\Gamma_0 \upharpoonright \ker (S^*-z)\big)^{-1}:\cG \to \cH,\\
  \res B_0 \ni z &\mapsto M_z :=\Gamma_1 \,G_z:\cG\to \cG.
\end{align*}
For $z\in\res B_0$ the operators $G_z$ and $M_z$ are bounded, and $z\mapsto G_z$ and $z\mapsto M_z$ are holomorphic in $z \in \res B_0$.
The adjoints of $G_z$ and $M_z$ are given by
\begin{equation*} 
  G_z^* = \Gamma_1 (B_0 - \overline z)^{-1} \quad \text{and} \quad M_z^* = M_{\bar{z}}.
\end{equation*}

\kp{
Let $\cG_\Pi$ be a closed subspace of $\cG$ viewed as a Hilbert space when endowed with the induced inner product.
Let $\Pi:\cG\to \cG_\Pi$ be the orthogonal projection, then $\Pi^*:\cG_\Pi\rightarrow\cG$ is the canonical embedding.
Let $\Theta$ be a linear operator in $\cG_\Pi$.
In the following we are interested in the operator $B_{\Pi,\Theta}$ defined as the restriction
of $S^*$ onto the set
}
\begin{equation*} 
  \label{def_A_theta_abstract}
  \dom B_{\Pi,\Theta} = \big\{ f \in \dom S^*: \Pi \Gamma_1 f = \Theta \Pi\Gamma_0 f,\,(\one-\Pi^* \Pi)\Gamma_0 f=0  \big\},
\end{equation*}
where the boundary condition $\Pi \Gamma_1 f = \Theta \Pi\Gamma_0 f$ in $\dom B_{\Pi,\Theta}$ also contains the condition $\Pi \Gamma_0 f \in \dom \Theta$.
A number of properties of $B_{\Pi,\Theta}$ appear to be encoded in $\Theta$. The most important of them for our purposes are summarized in the following theorem:

\begin{thm} \label{theorem_boundary_triple_abstract}
	The operator $B_{\Pi, \Theta}$ is (essentially) self-adjoint in $\mathcal{H}$
  if and only if $\Theta$ is  (essentially) self-adjoint in $\cG_\Pi$. 
  Furthermore, if $\Theta$ is self-adjoint and $z \in \res B_0$, then the following assertions hold:
  \begin{itemize}
    \item[\textup{(i)}] $z\in \spec B_{\Pi,\Theta}$ if and only if $0 \in \spec( \Theta-\Pi M_z\Pi^*)$.
    \item[\textup{(ii)}] $z\in \spec_\textup{p} B_{\Pi,\Theta}$ if and only if $0 \in \spec_\textup{p}( \Theta-\Pi M_z\Pi^*)$, and in that case
		the eigenspaces are related by $\ker (B_{\Pi,\Theta}-z) = G_z \Pi^* \ker  ( \Theta-\Pi M_z\Pi^*)$.
    \item[\textup{(iii)}] $z\in \spec_\ess B_{\Pi,\Theta}$ if and only if $0 \in \spec_\ess(\Theta-\Pi M_z\Pi^*)$.
    \item[\textup{(iii)}] For all $z \in \res B_{\Pi, \Theta} \cap \res B_0$ one has
    \begin{equation*} 
        (B_{\Pi,\Theta}-z)^{-1}=(B_0-z)^{-1}+G_z\Pi^*( \Theta-\Pi M_z\Pi^*)^{-1} \Pi G_{\Bar{z}}^*.
    \end{equation*}
  \end{itemize}
\end{thm}

Finally we recall a special approach for the construction
of boundary triples using abstract trace maps developed in \cite{Pos04} and \cite{Pos08}, see also \cite[Section 1.4.2]{BGP}.
Let $B$ be a self-adjoint operator in the Hilbert space $\cH$, 
let $\cG$ be another Hilbert space, and assume that
  \begin{equation*}
    \cT: \dom B \rightarrow \cG
  \end{equation*}
is a \emph{surjective} linear operator which is bounded with respect to the graph norm of $B$
and such that $\ker \cT$ is a dense subspace of the initial Hilbert space $\cH$. Then
\begin{equation*}
S := B\upharpoonright \ker \cT
\end{equation*}
is a densely defined closed symmetric operator.
Next, define for any  $z\in\res B$ the injective operator 
\begin{equation}
  \label{eq-gz1}
	G_z:= \big(\cT(B - \Bar{z})^{-1}\big)^*,
\end{equation}
which is bounded from $\mathcal{G}$ to $\mathcal{H}$.
Then one has $\ran G_z=\ker(S^*-z)$ for  $z\in\res B$ and \eqref{deco33} leads to the direct sum decomposition
\begin{equation}\label{deco234}
\dom S^* = \dom B \dot{+} \ran G_z,\quad z\in\res B,
\end{equation}
which shows that for all $f \in \dom S^*$ there exist unique $f_z \in \dom B$ and $\xi \in \mathcal{G}$ 
such that $f = f_z + G_z \xi$;
one can show that the component~$\xi$ is independent of the choice of $z$.
Having these notations in hand we can formulate now the following proposition:

\begin{prop} \label{proposition_boundary_triple_singular_perturbation}
  Let $\zeta \in \res B$ be fixed and define the mappings $\Gamma_0, \Gamma_1: \dom S^* \rightarrow \mathcal{G}$ 
  for $f = f_\zeta + G_\zeta \xi = f_{\bar{\zeta}} + G_{\bar{\zeta}} \xi \in \dom S^*$ by
  \begin{equation*}
    \Gamma_0 f := \xi \quad \text{and} \quad \Gamma_1 f := \frac{1}{2} \,\mathcal{T} (f_\zeta + f_{\bar{\zeta}}).
  \end{equation*}
  Then $\{\cG,\Gamma_0,\Gamma_1\}$ is a boundary triple for $S^*$ with $S^* \upharpoonright \ker \Gamma_0 = B$. Moreover, the $\gamma$-field and the 
  Weyl function are given by~\eqref{eq-gz1} and
	\[
	M_z = \mathcal{T}\Big(G_z - \frac{1}{2} (G_\zeta + G_{\bar{\zeta}}) \Big),\quad z\in\res B.
	\]
\end{prop}

\section{The free Dirac operator and a boundary triple for its singular perturbations} \label{section_boundary_triple}

In this section we first recall the definition of the free Dirac operator in~$\mathbb{R}^2$, a minimal and a maximal realization of the Dirac 
operator in $\mathbb{R}^2 \setminus \Sigma$, and we introduce and study some families of integral operators which will play an important role in 
our analysis in Section~\ref{section_delta_op}. Afterwards, we define a boundary triple which is useful in the
treatment of Dirac operators with singular $\delta$-interactions.

\subsection{The free, the minimal, and the maximal Dirac operators and associated integral operators} \label{section_int_op}

For $m \in \mathbb{R}$ the free Dirac operator in $\mathbb{R}^2$ is defined by 
\begin{equation} \label{def_free_op}
  A_0 f =-\rmi \sum_{j=1}^2 \sigma_j \partial_j f + m\sigma_3 f = -\rmi\,\sigma\cdot\nabla f + m\sigma_3 f, \quad \dom A_0 = H^1(\mathbb{R}^2; \mathbb{C}^2),
\end{equation}
where $\sigma := (\sigma_1, \sigma_2)$ and $\sigma_3$ are the $\mathbb{C}^{2 \times 2}$-valued Pauli spin matrices
in~\eqref{def_Pauli_matrices}.
First, we recall some basic properties of $A_0$.
Using the Fourier transform and~\eqref{anti_commutation} one verifies that
$A_0$ is self-adjoint in $L^2(\mathbb{R}^2; \mathbb{C}^2)$
and that its spectrum is purely essential,
\begin{equation*}
  \spec A_0 = \spec_\ess A_0 = \big(-\infty,-|m|\big]\,\cup\big[|m|,+\infty\big).
\end{equation*}
In particular, $\spec A_0=\RR$ for $m=0$.
Due to the identity 
\begin{equation*}
  (A_0-z)(A_0+z)  = (-\Delta + m^2-z^2)\sigma_0
\end{equation*}
one can express the resolvent of $A_0$ through the resolvent of the free Laplacian.
Recall that for $z\notin \spec(-\Delta) = [0,+\infty)$ the resolvent $(-\Delta-z)^{-1}$ is the integral operator
\begin{equation*}
  (-\Delta-z)^{-1} f(x)=\dfrac{1}{2\pi}\int_{\RR^2} K_0\big(\sqrt{-z}|x-y|\big)\, f(y)\, dy,
\end{equation*}
where $K_j$ stands for the modified Bessel function of second kind of order~$j$, and
 we take the principal square root function, i.e. for $z \in \mathbb{C} \setminus [0, \infty)$ the number $\sqrt{z}$ is determined by 
 $\textup{Re}\, \sqrt{z} > 0$.
For $z\in\res A_0$ one gets 
\begin{equation*}
  (A_0 - z)^{-1}=(A_0 + z) \big(-\Delta - (z^2-m^2)\big)^{-1}\sigma_0,
\end{equation*}
which leads to 
\begin{equation*}
	(A_0 - z)^{-1} f(x) = \int_{\RR^2}\phi_z(x-y) f(y) \dd y, \quad f \in L^2(\RR^2;\CC^2),
\end{equation*}
where 
\begin{equation} \label{def_Green_function}
\begin{split}
 	\phi_z(x) = \rmi\frac{\sqrt{m^2 - z^2}}{2\pi} K_1\big(\sqrt{m^2 - z^2}|x|\big)\bigg(\sigma\cdot \frac{x}{|x|}\bigg)\qquad \qquad \qquad&\\
	 + \frac1{2\pi}K_0\big(\sqrt{m^2 -z^2}|x|\big)\big(m\sigma_3 + z\sigma_0\big)&.
  \end{split}
\end{equation}

\kp{
Next we introduce a symmetric operator which is suitable for our purposes. More precisely, denote by
$S$ be the restriction of $A_0$ to the functions vanishing at $\Sigma$, i.e.
\begin{equation} \label{def_S}
S f = (-\rmi \sigma \cdot \nabla + m \sigma_3) f,\quad \dom S=H^1_0(\RR^2\setminus\Sigma; \mathbb{C}^2).
\end{equation}
Then the operator $A_{\eta, \tau}$ defined in~\eqref{dirdelta} is an extension of $S$.
The standard theory implies that the adjoint $S^*$ is the maximal realization of the same differential expression
in $\RR^2\setminus\Sigma$, i.e.
\begin{equation} \label{def_S_star}
  \begin{split}
    S^* f &= (-\rmi \sigma \cdot \nabla + m \sigma_3) f_+ \oplus (-\rmi \sigma \cdot \nabla + m \sigma_3) f_-, \\
    \dom S^* &= \big\{ f = f_+ \oplus f_- \in L^2(\Omega_+; \mathbb{C}^2) \oplus L^2(\Omega_-; \mathbb{C}^2):\, f_\pm\in H(\sigma,\Omega_\pm)\big\},
  \end{split}
\end{equation}
and we recall that
\begin{equation} \label{def_H_sigma}
  H(\sigma,\Omega_\pm)=\big\{ f_\pm \in L^2(\Omega_\pm;\CC^2):  (-\rmi \sigma \cdot \nabla + m \sigma_3) f_\pm\in L^2(\Omega_\pm;\CC^2)\big\}.
\end{equation}
The set $H(\sigma, \Omega_\pm)$ endowed with the norm
\begin{equation*}
  \| f_\pm \|_{H(\sigma, \Omega_\pm)}^2 := \| f \|_{L^2(\Omega_\pm; \mathbb{C}^4)}^2 +  \big\| (-\rmi \sigma \cdot \nabla + m \sigma_3) f_\pm \big\|_{L^2(\Omega_\pm; \mathbb{C}^2)}^2
\end{equation*}
is a Hilbert space, which is obviously independent of $m$. 
For our further considerations, it is useful to extend the Dirichlet trace operator onto $H(\sigma; \Omega_\pm)$. 
In the following lemma we summarize several known results; we refer to \cite[Lemma~2.3 and Lemma~2.4]{BFSV}
for compact proofs:

\begin{lem} \label{lemma_trace}
  The trace map
	\[
	\mathcal{T}_{\pm, 0}^D: H^1(\Omega_\pm; \mathbb{C}^2) \rightarrow H^{\frac{1}{2}}(\Sigma; \mathbb{C}^2),
	\quad
	\mathcal{T}_{\pm, 0}^D f = f|_\Sigma, 
	\]
  extends uniquely to a bounded linear operator
  \begin{equation*} 
    \mathcal{T}_\pm^D: H(\sigma,\Omega_\pm) \rightarrow H^{-\frac{1}{2}}(\Sigma;\CC^2).
  \end{equation*}
  Moreover, if $\mathcal{T}_\pm^D f \in H^{\frac{1}{2}}(\Sigma; \mathbb{C}^2)$ for $f \in H(\sigma, \Omega_\pm)$, then $f \in H^1(\Omega_\pm; \mathbb{C}^2)$.
\end{lem}
}

Now we introduce some families of integral operators corresponding to the Green function $\phi_z$ associated to $A_0$ given by~\eqref{def_Green_function}. 
Let us denote the Dirichlet trace operator on $H^1(\mathbb{R}^2; \mathbb{C}^2)$ 
by $\mathcal{T}^D: H^1(\mathbb{R}^2; \mathbb{C}^2) \rightarrow H^{\frac{1}{2}}(\Sigma; \mathbb{C}^2)$. 
It is well-known that $\mathcal{T}^D$ is bounded, surjective, and $\ker \mathcal{T}^D = H^1_0(\mathbb{R}^2 \setminus \Sigma; \mathbb{C}^2)$; 
  cf. \cite[Theorems~3.37 and~3.40]{M00}.
For $z \in \res A_0$ we first define the bounded operator
\begin{equation} \label{def_Phi_z_prime}
  \Phi_z' := \mathcal{T}^D (A_0 - \Bar{z})^{-1}: L^2(\mathbb{R}^2; \mathbb{C}^2) \rightarrow H^{\frac{1}{2}}(\Sigma; \mathbb{C}^2)
\end{equation}
and its anti-dual
\begin{equation} \label{def_Phi_z}
  \Phi_z := \big( \mathcal{T}^D (A_0 - \overline{z})^{-1} \big)': H^{- \frac{1}{2}}(\Sigma; \mathbb{C}^2) \rightarrow L^2(\mathbb{R}^2; \mathbb{C}^2).
\end{equation}
The basic properties of $\Phi_z$ are stated in the following proposition:

\begin{prop} \label{proposition_Phi_z}
  Let $z \in \res A_0$ and consider the operator $\Phi_z$ in \eqref{def_Phi_z}. Then  for $\varphi \in L^2(\Sigma; \mathbb{C}^2)$ one has
  \begin{equation*}
    \Phi_z \varphi (x) = \int_\Sigma \phi_z(x-y) \varphi(y) \dd s(y) \quad \text{for a.e. }x \in \mathbb{R}^2\setminus\Sigma.
  \end{equation*}
  Moreover, $\Phi_z$ is a bounded bijective operator from $ H^{- \frac{1}{2}}(\Sigma; \mathbb{C}^2)$ onto  $\ker (S^* - z)$.
\end{prop}
\begin{proof}
  First, due to the properties of the trace map it is clear that $\Phi_z'$ defined by~\eqref{def_Phi_z_prime} is surjective and
  \begin{equation*}
    \ker \Phi_z' = \big\{ f \in L^2(\mathbb{R}^2; \mathbb{C}^2): (A_0 - \Bar{z})^{-1} f \in H^1_0(\mathbb{R}^2 \setminus \Sigma; \mathbb{C}^2) \big\} = 
    \ran (S- \Bar{z}),
  \end{equation*}
  as $S = A_0 \upharpoonright H^1_0(\mathbb{R}^2 \setminus \Sigma; \mathbb{C}^2)$. Using the closed range theorem,
  $(\ran\Phi_z)^\bot=\ker\Phi_z'$, and the fact that $\ker(S^*-z)=(\ran (S- \Bar{z}))^\bot$
  is closed we conclude that
  \begin{equation*}
    \Phi_z: H^{-\frac{1}{2}}(\Sigma; \mathbb{C}^2) \rightarrow \ker (S^* - z)
  \end{equation*}
  is a bounded bijective operator. To prove the integral representation 
  consider $\varphi \in L^2(\Sigma; \mathbb{C}^2)$ and $f \in L^2(\mathbb{R}^2; \mathbb{C}^2)$. A direct computation using Fubini's theorem shows  
  \begin{align*}
    (f, \Phi_z \varphi )_{L^2(\RR^2; \mathbb{C}^2)} &= ( \Phi_z' f, \varphi )_{L^2(\Sigma;\CC^2)}\\ 
    &=
    \big(\cT^D (A_0- \Bar{z})^{-1} f, \varphi \big)_{L^2(\Sigma;\CC^2)}\\
    &=\int_\Sigma \bigg( \int_{\RR^2} \phi_{\Bar{z}}(x-y) f(y)\dd y, \varphi(x) \bigg)_{\CC^2}\dd s(x)\\
    &= \int_{\RR^2} \bigg( f(y),\int_\Sigma \phi_{\Bar{z}}(x-y)^* \varphi(x)\dd s(x) \bigg)_{\CC^2}\dd y\\
    &=\int_{\RR^2} \bigg( f(y),\int_\Sigma \phi_z(y-x)\varphi(x)\dd s(x)\bigg)_{\CC^2}\dd y,
  \end{align*}
  where the symmetry property $\phi_{\Bar{z}}(x-y)^* = \phi_z(y-x)$ was used in the last equality.
  This implies the representation for $\Phi_z \varphi$, $\varphi \in L^2(\Sigma; \mathbb{C}^2)$, and completes the proof of this proposition.
\end{proof}

We will also need a family of boundary integral operators with integral kernel~$\phi_z$. To introduce these operators, we study first the structure of the Green function~$\phi_z$ in more detail:

\begin{lem} \label{lemma_phi_z}
  Let $z \in \res A_0$ and consider the function $\phi_z$ in \eqref{def_Green_function}. Then there exist scalar analytic functions $g_1, g_2, g_3$, 
  and $g_4$ and a constant $c_1 < 0$ such that
  \begin{equation} \label{decomposition_phi_z1}
    \begin{aligned}
 	   	\phi_z(x)	=\,& \dfrac{\rmi}{2\pi }\, \sigma\cdot \dfrac{x}{|x|^2	}
	  -\dfrac{1}{2\pi}\Big(\log |x| +\log \sqrt{m^2 - z^2}+c_1\Big)(m\sigma_3 + z\sigma_0)\\
	  & + \dfrac{\rmi}{2\pi} (m^2-z^2) \Big[
	  g_1\big((m^2-z^2)|x|^2\big)\big(\log \sqrt{m^2 - z^2} + \log |x|\big)\\
	  &\hskip 6.8cm + g_2\big((m^2-z^2)|x|^2\big)\Big](\sigma\cdot x)\\
	  +\dfrac{1}{2\pi}(m^2&-z^2)|x|^2\Big[
	  g_3\big((m^2-z^2)|x|^2\big)\big(\log \sqrt{m^2 - z^2} +\log |x|\big)\\
	  &\hskip 5.8cm +g_4\big((m^2-z^2)|x|^2\big)\Big](m\sigma_3 + z\sigma_0).
    \end{aligned}
  \end{equation}
  In particular, there exist $C^\infty$-smooth matrix valued functions $f_1$ and $f_2$ such that
  \begin{equation} \label{decomposition_phi_z2}
    \phi_z(x) = \dfrac{\rmi}{2\pi} \begin{pmatrix}
      0 & \dfrac{1}{x_1+\rmi x_2}\\
      \dfrac{1}{x_1-\rmi x_2} & 0
    \end{pmatrix} + f_1(x) \log|x| +f_2(x).
  \end{equation}
\end{lem}
\begin{proof}
  In order to prove the claimed results, let us recall the series representations of $K_j$ from, e.g., \S 10.25.2, 10.31.1, and 10.31.2 in \cite{NIST}, which read
  \begin{align*}
    I_\mu(t)&=\dfrac{t^\mu}{2^\mu} \sum_{k=0}^\infty \dfrac{t^{2k}}{4^k k!\Gamma(\mu+k+1)},
    \quad \mu\in\{0,1\},\\
    K_1(t)&= \dfrac{1}{t} +(\log t - \log 2) I_1(t) -\dfrac{t}{4} \sum_{k=0}^\infty
    \big( \psi(k+1)+\psi(k+2)\big) \dfrac{t^{2k}}{4^k k! (k+1)!},\\
    K_0(t)&=-(\log t - \log 2 + \gamma) I_0(t) +\sum_{k=1}^\infty \sum_{j=1}^k \dfrac{1}{j} \dfrac{t^{2k}}{4^k (k!)^2},
  \end{align*}
  with $\psi(t)=\dfrac{\Gamma'(t)}{\Gamma(t)}$ and $\gamma=-\psi(1)<\log 2$. This implies first that 
  \begin{equation*}
    I_0(t)=1+t^2h_0(t^2) \quad \text{and} \quad I_1(t)=t h_1(t^2)
  \end{equation*}
  with some analytic functions $h_0$ and $h_1$. Furthermore, with some analytic functions $k_0$ and $k_1$ we have
  \begin{align*}
    K_1(t)&= \dfrac{1}{t} +(\log t- \log 2) I_1(t) +t k_1(t^2)\\
    &= \dfrac{1}{t} + t h_1(t^2) \log t +t \big(k_1(t^2)- h_1(t^2)\log 2 \big),\\
    K_0(t)&=-(\log t - \log 2 + \gamma) I_0(t) +t^2 k_0(t^2)\\
    &=-\log t-c_1 -t^2 h_0(t^2) \log t -c_1 t^2 h_0(t^2)+t^2k_0(t^2)
  \end{align*}
  with $c_1:=\gamma-\log 2<0$. This can be rewritten in a simplified form as
  \begin{align*}
    K_1(t)&=\dfrac{1}{t}+t g_1(t^2)\log t + t g_2(t^2), \\ 
    K_0(t)&=-\log t -c_1+t^2 g_3(t^2) \log t + t^2 g_4(t^2), 
  \end{align*}
  where $g_1, g_2, g_3,$ and $g_4$ are analytic functions and $c_1<0$.
  Using now the explicit expression for $\phi_z$ we decompose
\begin{align*}
 	&\phi_z(x) = \rmi\frac{\sqrt{m^2 - z^2}}{2\pi} K_1\big(\sqrt{m^2 - z^2}|x|\big)
	\bigg(\sigma\cdot \frac{x}{|x|}\bigg)+ \frac1{2\pi}K_0\big(\sqrt{m^2 -z^2}|x|\big)
	\big(m\sigma_3 + z\sigma_0\big)\\
	&\,\,\, = \rmi\frac{\sqrt{m^2 - z^2}}{2\pi} \bigg\{\dfrac{1}{\sqrt{m^2 - z^2}|x|}\\
  &\hskip 2.9cm+\sqrt{m^2 - z^2}|x| g_1\big((m^2-z^2)|x|^2\big)\log \big( \sqrt{m^2 - z^2} |x|\big)\\
	&\hskip 6.1cm+ \sqrt{m^2 - z^2}|x| g_2\big((m^2-z^2)|x|^2\big)   \bigg\}\bigg(\sigma\cdot \frac{x}{|x|}\bigg)\\
	& \qquad +\frac1{2\pi}\Big\{
	-\log \big( \sqrt{m^2 - z^2} |x|\big)-c_1\\
	&\hskip 2cm +(m^2-z^2)|x|^2 g_3\big((m^2-z^2)|x|^2\big)\log \big( \sqrt{m^2 - z^2} |x|\big)\\
	&\hskip 5.2cm +(m^2-z^2)|x|^2 g_4\big((m^2-z^2)|x|^2\big)	\Big\} \big(m\sigma_3 + z\sigma_0\big),
	\end{align*}
  which leads to the decomposition \eqref{decomposition_phi_z1}. The representation \eqref{decomposition_phi_z2} follows from \eqref{decomposition_phi_z1}
  after noting that
  \begin{equation*}
   \dfrac{\rmi}{2\pi }\, \sigma\cdot \dfrac{x}{|x|^2}=\dfrac{\rmi}{2\pi} \begin{pmatrix}
      0 & \dfrac{1}{x_1+\rmi x_2}\\
      \dfrac{1}{x_1-\rmi x_2} & 0
    \end{pmatrix}.
  \end{equation*}

\end{proof}

For $z \in \res A_0$ we introduce the operator
\begin{equation} \label{def_C_z}
  \mathcal{C}_z \varphi(x) := \pv \int_\Sigma \phi_z(x-y) \varphi(y) \text{d} s(y), \quad \varphi \in C^\infty(\Sigma; \mathbb{C}^2), ~x \in \Sigma.
\end{equation}
The basic properties of $\mathcal{C}_z$ are stated in the following proposition. For the formulation of the result, recall the definition 
of the operator $\Lambda$ from~\eqref{def_Lambda} and of the Cauchy transform $C_\Sigma$ and its dual $C_\Sigma'$ from~\eqref{def_Cauchy_transform} 
and~\eqref{def_Cauchy_transform_dual}, respectively.

\begin{prop} \label{proposition_C_z}
  Let $z \in \res A_0$ and consider the operator $\mathcal{C}_z$ in \eqref{def_C_z}. Then $\mathcal{C}_z \in \Psi_\Sigma^0$ and, 
  in particular, $\mathcal{C}_z$ gives rise to a bounded operator in $H^s(\Sigma; \mathbb{C}^2)$ for any $s \in \mathbb{R}$. 
  The realization in $L^2(\Sigma; \mathbb{C}^2)$ satisfies $\mathcal{C}_z^* = \mathcal{C}_{\Bar{z}}$. Moreover, 
  if $\bf t = ({\bf t}_1, {\bf t}_2)$ is the tangent vector field at $\Sigma$ and $T = {\bf t}_1 + \rmi {\bf t}_2$, $\overline T= {\bf t}_1 - \rmi {\bf t}_2$, then one has
  \begin{equation} \label{Lambda_C_z_Lambda}
    \Lambda \mathcal{C}_z \Lambda=
    \dfrac{1}{2}\,\begin{pmatrix}
     0 & \Lambda C_\Sigma \overline{T} \Lambda\\
     \Lambda T C_\Sigma' \Lambda & 0
    \end{pmatrix}
    + \dfrac{\ell}{4\pi} \begin{pmatrix} (z+m)\one & 0 \\ 0 & (z-m)\one \end{pmatrix} +\Psi
  \end{equation}
  with $\Psi\in\Psi^{-1}_\Sigma$.
\end{prop}

\begin{proof}
We make use of \eqref{decomposition_phi_z1} to decompose $\phi_z$ in the form 
  \begin{equation*}
    \phi_z(x) = \chi_1(x) + \chi_2(x) + \chi_3(x),
  \end{equation*}
  where
  \begin{align*}
    \chi_1(x)&= \dfrac{\rmi}{2\pi} \begin{pmatrix}
    0 & \dfrac{1}{x_1+\rmi x_2}\\
    \dfrac{1}{x_1-\rmi x_2} & 0
     \end{pmatrix},\\
    \chi_2(x)&=-
    \dfrac{1}{2\pi} \begin{pmatrix}
    z+m & 0\\
    0 & z-m
    \end{pmatrix}
    \log |x|\\
    \chi_3(x)&=	  \big[h_1\big(|x|^2\big) \log |x|+ h_2\big(|x|^2\big)\big]\,(\sigma\cdot x)\\
&\qquad\quad + \big[|x|^2 h_3\big(|x|^2\big)\log|x| + h_4\big(|x|^2\big)\big]\,(m\sigma_3 +z\sigma_0),
  \end{align*}
  and  $h_1, h_2, h_3$, and $h_4$ are analytic functions. In the following we will use the corresponding decomposition 
  $\mathcal{C}_z = P_1 + P_2 + P_3$, where
  \begin{equation*}
  \begin{split}  
    (P_1 \varphi)(x)&=\pv \int_\Sigma \chi_1(x-y) \varphi(y)\dd s(y), \\
    (P_2 \varphi)(x)&=\int_\Sigma \chi_2(x-y) \varphi(y)\dd s(y), \\
    (P_3 \varphi)(x)&=\int_\Sigma \chi_3(x-y) \varphi(y)\dd s(y).
  \end{split}
  \end{equation*}
 Here we have removed the principal value from the integral operators $P_2$ and $P_3$, since these integrals converge almost everywhere by \cite[Proposition~3.10]{F95}.
 
 Let us discuss the operator $P_1$ first. With the help of \eqref{relation_Cauchy_transform} and~\eqref{relation_dual_Cauchy_transform} we obtain
  \begin{equation} \label{equation_decomposition1}
    P_1  = \dfrac{1}{2}\,\begin{pmatrix}
     0 & C_\Sigma \overline{T} \\
     T C_\Sigma'  & 0
    \end{pmatrix}
 \end{equation}
 and since $T,\overline T\in \Psi_\Sigma^0$ we conclude $P_1 \in\Psi_\Sigma^0$ from Proposition~\ref{proposition_Cauchy_transform}. 
  
  Next, we claim that the integral operator $P_2$ admits the representation
 \begin{equation} \label{equation_decomposition2}
    P_2 = \dfrac{\ell}{4\pi} \begin{pmatrix} (z+m)\Lambda^{-2} & 0 \\ 0 & (z-m)\Lambda^{-2} \end{pmatrix} + \Psi_1 
  \end{equation}
  with some $\Psi_1 \in \Psi^{-2}_\Sigma$ and $\Lambda^{-2}=U^{-1}L^{-2}U\in\Psi_\Sigma^{-1}$, so that $P_2\in\Psi_\Sigma^{-1}$.
  In fact, using a parametrization $\gamma: [0, \ell] \rightarrow \mathbb{R}^2$ of $\Sigma$ we find 
  \begin{equation*}
    (U P_2 f)(t)=-\dfrac{\ell}{2\pi} \begin{pmatrix}z+m & 0\\0 & z-m\end{pmatrix} \int_\TT \log \big|\gamma(\ell t)-\gamma(\ell s)\big|\,
    f\big(\gamma(\ell s)\big)\dd s
  \end{equation*}
  for $f\in C^\infty(\Sigma)$. Therefore, with $f=U^{-1}u$ and $\rho(\cdot)=\gamma_1(\ell\cdot)+\rmi\gamma_2(\ell\cdot)\equiv\gamma(\ell\cdot)$ we conclude
  \begin{equation*}
  \begin{split}
    (U P_2 U^{-1}u)(t)&=-\dfrac{\ell}{2\pi} \begin{pmatrix}z+m & 0\\0 & z-m\end{pmatrix} \int_\TT \log \big|\rho(t)-\rho(s)\big|\,
    u(s)\dd s\\
    &=-\dfrac{\ell}{2\pi} \begin{pmatrix}z+m & 0\\0 & z-m\end{pmatrix} H_0u(t)
   \end{split}
  \end{equation*}
  with $H_0$ as in Proposition~\ref{prop15}.
  Now it follows from Proposition~\ref{prop15} (with $m=0$, $a\equiv 1$, and $\rho$ as above) that 
  $H_0\in\Psi^{-1}$ and $\one + 2 L H_0 L \in\Psi^{-1}$. Furthermore, 
  Proposition~\ref{proposition_properties_PPDO}~(ii) and $L^{-1}\in\Psi^{-\frac{1}{2}}$ yield $\frac{1}{2}L^{-2} +   H_0 \in\Psi^{-2}$ and
   hence
 \begin{equation*}
  -\dfrac{\ell}{4\pi} \begin{pmatrix}(z+m)L^{-2} & 0\\0 & (z-m)L^{-2}\end{pmatrix}
   + U P_2 U^{-1}  \in\Psi^{-2}.
  \end{equation*}
  We then conclude
  \begin{equation*}
  -\dfrac{\ell}{4\pi} \begin{pmatrix}(z+m)\Lambda^{-2} & 0\\0 & (z-m)\Lambda^{-2}\end{pmatrix}
   +  P_2   \in\Psi_\Sigma^{-2},
 \end{equation*}
  which leads to \eqref{equation_decomposition2}. 
  
  It will be shown now that $P_3 \in\Psi_\Sigma^{-2}$. Indeed, setting again $\rho(\cdot)=\gamma_1(\ell\cdot)+\rmi\gamma_2(\ell\cdot)\equiv\gamma(\ell\cdot)$ we see that $\chi_3$ can be written in the form
  \begin{equation*}
    \begin{split}
      \chi_3&(\rho(t) - \rho(s)) = \log |\rho(t)-\rho(s)| a_1(t, s) \begin{pmatrix} 0 & \overline{\rho(t)-\rho(s)} \\ \rho(t)-\rho(s) & 0 \end{pmatrix}
      + a_2(t, s)
    \end{split}
  \end{equation*}
  with the $C^\infty$-smooth matrix-valued functions 
  \begin{align*}
      a_1(t, s) &:= h_1\big(|\rho(t)-\rho(s)|^2\big) \sigma_0 \\
      &\quad +  h_3\big(|\rho(t) - \rho(s)|^2\big)(m\sigma_3 +z\sigma_0) \begin{pmatrix} 0 & \overline{\rho(t)-\rho(s)} \\ \rho(t)-\rho(s) & 0 \end{pmatrix},\\ 
      a_2(t, s) &:= h_2\big(|\rho(t)-\rho(s)|^2\big)\big]\,\begin{pmatrix} 0 & \overline{\rho(t)-\rho(s)} \\ \rho(t)-\rho(s) & 0 \end{pmatrix}\qquad& \\
       &\quad+ h_4\big(|\rho(t) - \rho(s)|^2\big)\,(m\sigma_3 +z\sigma_0).&
  \end{align*}
  Hence, it follows as above in the proof of~\eqref{equation_decomposition2} with Proposition~\ref{prop15} applied in the case $m=1$ that 
  $U P_3 U^{-1} = H_1 \in \Psi^{-2}$, so that $P_3 \in \Psi^{-2}_\Sigma$.
  Together with~\eqref{equation_decomposition1} and~\eqref{equation_decomposition2} this implies first 
  $\mathcal{C}_z \in \Psi_\Sigma^0$ and in a second step, together with Proposition~\ref{proposition_properties_PPDO}~(i) 
  and $\Lambda \in \Psi_\Sigma^{\frac{1}{2}}$, that also~\eqref{Lambda_C_z_Lambda} is true.

  Finally, since $\phi_z(y-x)^*=\phi_{\overline{z}}(x-y)$, we find that the 
  realization of $\mathcal{C}_z$ in $L^2(\Sigma; \mathbb{C}^2)$ satisfies $\mathcal{C}_z^*=\mathcal{C}_{\Bar{z}}$. Hence, all claims have been shown.
\end{proof}
  
%
%

Finally, we prove a result on how $\Phi_z$ and $\mathcal{C}_z$ are related to each other by taking traces. 
Recall that $\mathcal{T}_\pm^D$ is the Dirichlet trace operator on $H(\sigma, \Omega_\pm)$, see Lemma~\ref{lemma_trace}.

\begin{prop} \label{proposition_Plemelj_Sokhotskii}
  For $\varphi \in H^{-\frac{1}{2}}(\Sigma; \mathbb{C}^2)$ one has
  \begin{equation}
   \label{ps00}
    \cT_\pm^D \Phi_z \varphi = \mp\,\dfrac{\rmi}{2}\, (\sigma\cdot \nu)\,\varphi + \mathcal{C}_z \varphi.
  \end{equation}
\end{prop}
\begin{proof}
First we note that it suffices to prove~\eqref{ps00} for $\varphi \in C^\infty(\Sigma; \mathbb{C}^2)$; by continuity 
this implies the claim for any $\varphi \in H^{-\frac{1}{2}}(\Sigma;\CC^2)$.
The assertion essentially follows from the classical Plemelj-Sokhotskii formula, see, e.g., \cite[Theorem~4.1.1]{SV},
which states that the holomorphic function
\[
\CC\setminus\Sigma\ni \xi\mapsto \Phi(\xi)=\dfrac{1}{2\pi \rmi}\int_\Sigma \dfrac{\varphi(\zeta)}{\zeta - \xi}\, \dd \zeta
\]
satisfies
\begin{equation}
    \label{eq-tpm}
\cT_\pm^D \Phi(\xi)=\dfrac{1}{2\pi \rmi}\pv \int_\Sigma \dfrac{\varphi(\zeta)}{\zeta - \xi}\, \dd\zeta\pm \dfrac{1}{2}\,\varphi(\xi), \quad \xi\in\Sigma.
\end{equation}
In order to use it, recall that by~\eqref{decomposition_phi_z2} we can write 
$\phi_z(x)=\chi_1(x)+\widetilde{\chi}_2(x)$ with
\[
\chi_1(x)=-\dfrac{1}{2\pi \rmi} \begin{pmatrix}
0 & \dfrac{1}{x_1+\rmi x_2}\\
\dfrac{1}{x_1-\rmi x_2} & 0
\end{pmatrix}\quad\text{and}
\quad \widetilde{\chi}_2(x)= f_1(x) \log|x| +f_2(x),
\]
where $f_1$ and $f_2$ are $C^\infty$-smooth matrix functions. In a corresponding way we decompose $\Phi_z = \Psi_1 + \Psi_2$ with
\begin{equation*}
  \Psi_1 \varphi(x) = \int_\Sigma \chi_1(x-y)\, \varphi(y)\dd  s(y) \quad \text{and} \quad
  \Psi_2 \varphi(x) = \int_\Sigma \widetilde{\chi}_2(x-y)\, \varphi(y)\dd  s(y),
\end{equation*}
and $\mathcal{C}_z = P_1 + P_2$ with
\begin{equation*}
  P_1 \varphi(x) = \pv \int_\Sigma \chi_1(x-y)\, \varphi(y)\dd s(y) \quad \text{and} \quad
  P_2 \varphi(x) = \int_\Sigma \widetilde{\chi}_2(x-y)\, \varphi(y)\dd s(y).
\end{equation*}
As in the proof of Proposition~\ref{proposition_C_z} we have removed the principal value from the integral operator $P_2$, since the integral exists almost everywhere.
One sees easily that $\Psi_2 \varphi$ is continuous on $\mathbb{R}^2$, and its value on $\Sigma$
coincides with $P_2 \varphi$, i.e.
\begin{equation}
   \label{ps1}
	\cT_\pm^D \Psi_2 \varphi = P_2 \varphi.
\end{equation}
In order to find the relation between $\Psi_1 \varphi$ and $P_1 \varphi$, we write the normal vector field as a complex number $N=\nu_1+\rmi \nu_2=\gamma_2'-\rmi \gamma_1'$ 
and use the relation
$\dd (y_1+\rmi y_2)= \rmi N(y) \dd s(y)$ of the complex and the classical line element on $\Sigma$.
With $\varphi=(\varphi_1, \varphi_2)$ we get then
\begin{align*}
\Psi_1 \varphi(x)&=\dfrac{1}{2\pi\rmi} \!\!\!\bigintss_\Sigma \! \begin{pmatrix}
0 & \dfrac{1}{(y_1+\rmi y_2)-(x_1+\rmi x_2)}\\
\dfrac{1}{(y_1-\rmi y_2)-(x_1-\rmi x_2)} & 0
\end{pmatrix}\!\!\begin{pmatrix}\varphi_1(y)\\ \varphi_2(y)\end{pmatrix}\!\dd s(y)\\
&=
\begin{pmatrix}
\dfrac{1}{2\pi\rmi} \displaystyle\int_\Sigma \dfrac{\varphi_2(y)}{(y_1+\rmi y_2)-(x_1+\rmi x_2)}\dd s(y)\\[\bigskipamount]
\dfrac{1}{2\pi\rmi} \displaystyle \overline{\int_\Sigma \dfrac{\overline{\varphi_1(y)}}{(y_1+\rmi y_2)-(x_1+\rmi x_2)} \dd s(y) }
\end{pmatrix}\\
&=
\begin{pmatrix}
\dfrac{1}{2\pi\rmi} \displaystyle\int_\Sigma \dfrac{-\rmi \overline{N(y)} \varphi_2(y)}{(y_1+\rmi y_2)-(x_1+\rmi x_2)}\dd(y_1+\rmi y_2)\\[\bigskipamount]
- \overline{\dfrac{1}{2\pi\rmi} \displaystyle\int_\Sigma\dfrac{-\rmi \overline{N(y)} \overline{\varphi_1(y)}}{(y_1+\rmi y_2)-(x_1+\rmi x_2)} \dd(y_1+\rmi y_2)}
\end{pmatrix}.
\end{align*}
Applying now~\eqref{eq-tpm} to each component of this vector we find that
\begin{align*}
\cT_\pm^D \Psi_1 \varphi(x)&=\!
\begin{pmatrix}
\dfrac{1}{2\pi\rmi} \displaystyle\pv\int_\Sigma \dfrac{-\rmi \overline{N(y)} \varphi_2(y)}{(y_1+\rmi y_2)-(x_1+\rmi x_2)}\dd(y_1+\rmi y_2)\\[\bigskipamount]
- \overline{\dfrac{1}{2\pi\rmi} \displaystyle\pv \int_\Sigma \dfrac{-\rmi \overline{N(y)\varphi_1(y)}}{(y_1+\rmi y_2) - (x_1+\rmi x_2)} d(y_1+\rmi y_2)}
\end{pmatrix}\!\pm\!\frac{1}{2}\,\!
\begin{pmatrix}
-\rmi \overline{N(x)} \varphi_2(x)\\
-\rmi N(x) \varphi_1(x)
\end{pmatrix}\\
&=\begin{pmatrix}
-\dfrac{1}{2\pi\rmi} \pv \displaystyle\int_\Sigma \dfrac{ \varphi_2(y)}{(x_1+\rmi x_2)-(y_1+\rmi y_2)}\dd s(y)\\[\bigskipamount]
-\dfrac{1}{2\pi\rmi} \pv \displaystyle\int_\Sigma\dfrac{\varphi_1(y)}{(x_1-\rmi x_2)-(y_1-\rmi y_2)} \dd s(y)
\end{pmatrix}\mp\dfrac{\rmi }{2}\begin{pmatrix}
 \overline{N(x)} \varphi_2(x)\\
N(x) \varphi_1(x)
\end{pmatrix}\\
&=P_1 \varphi(x)\mp\dfrac{\rmi }{2}\,(\sigma\cdot \nu(x))\, \varphi(x).
\end{align*}
A combination of this and~\eqref{ps1} leads to the claim of this proposition.
\end{proof}

\subsection{A boundary triple for Dirac operators with singular interactions supported on a loop} \label{section_boundary_triple_sing_interaction}

In this section we follow the strategy from Section~\ref{ssec-bt} to introduce a boundary triple 
which is suitable to study Dirac operators in $L^2(\mathbb{R}^2; \mathbb{C}^2)$ with singular interactions supported on 
the loop $\Sigma$. To get an explicit representation of the boundary mappings the results from Section~\ref{section_int_op} play 
an important role. We remark that the obtained boundary triple is closely related to the one used in~\cite{BH} to study Dirac operators in the three dimensional case.

Recall the definitions of the free Dirac operator $A_0$, the symmetric operator $S$, and its adjoint $S^*$ from~\eqref{def_free_op}, \eqref{def_S}, 
and~\eqref{def_S_star}, respectively. Moreover, $\mathcal{T}_\pm^D$ is the Dirichlet trace operator defined on $\dom S^*$ from Lemma~\ref{lemma_trace}, 
the integral operators $\Phi_z$ and $\mathcal{C}_z$ are introduced for $z \in \res A_0$ in~\eqref{def_Phi_z} 
and~\eqref{def_C_z}, respectively. The operator $\Lambda \in \Psi_\Sigma^{\frac{1}{2}}$ is given by~\eqref{def_Lambda} and will sometimes be viewed as
an isomorphism from $L^2(\Sigma;\mathbb{C}^2)$ to $H^{-\frac{1}{2}}(\Sigma;\mathbb{C}^2)$ or from $H^{\frac{1}{2}}(\Sigma;\mathbb{C}^2)$ to $L^2(\Sigma;\mathbb{C}^2)$, 
and is also regarded as an unbounded strictly positive
self-adjoint operator in $L^2(\Sigma;\mathbb{C}^2)$.

\begin{prop}\label{prop23}\label{prop-bt3}
Let $\zeta\in\res A_0$ be fixed. Define $\Gamma_0, \Gamma_1: \dom S^* \rightarrow L^2(\Sigma; \mathbb{C}^2)$ by
\begin{equation} \label{def_Gamma}
  \begin{split}
    \Gamma_0 f&=\rmi\Lambda^{-1} (\sigma\cdot \nu)\big(\cT_+^D f_+ -\cT_-^D f_-),\\
    \Gamma_1 f&=\dfrac{1}{2}\,\Lambda \Big((\cT_+^D f_+ + \cT_-^D f_-) -(\mathcal{C}_\zeta+\mathcal{C}_{\Bar\zeta}) \Lambda \Gamma_0 f \Big), 
    \qquad f = f_+ \oplus f_- \in \dom S^*.
  \end{split}
\end{equation}
Then $\{ L^2(\Sigma; \mathbb{C}^2), \Gamma_0, \Gamma_1 \}$ is a boundary triple for $S^*$ such that $A_0=S^*\upharpoonright\ker\Gamma_0$. Moreover, the 
corresponding $\gamma$-field is
  \begin{equation*}
   \res A_0\ni z \mapsto G_z = \Phi_z \Lambda
  \end{equation*}
  and the Weyl function is
  \begin{equation*}
    \res A_0\ni z\mapsto M_z = \Lambda \Big(\mathcal{C}_z -\dfrac{1}{2} \big(\mathcal{C}_\zeta+\mathcal{C}_{\Bar\zeta} \big) \Big)\Lambda.
  \end{equation*}
\end{prop}
\begin{proof}
  Recall that the Dirichlet trace operator $\mathcal{T}^D: H^1(\mathbb{R}^2; \mathbb{C}^2) \rightarrow H^{\frac{1}{2}}(\Sigma; \mathbb{C}^2)$
  is bounded, surjective, and one has $\ker \mathcal{T}^D = H^1_0(\mathbb{R}^2 \setminus \Sigma; \mathbb{C}^2)$. Hence, 
  \begin{equation*}
    \mathcal{T} := \Lambda \mathcal{T}^D: H^1(\mathbb{R}^2; \mathbb{C}^2) = \dom A_0 \rightarrow L^2(\Sigma; \mathbb{C}^2)
  \end{equation*}
  is bounded and surjective with $\ker \mathcal{T} = \dom S$. Following the constructions in Section~\ref{ssec-bt} for $B = A_0$ we consider for $z \in \res A_0$ 
  \begin{equation*}
    \mathcal{T} (A_0 - \bar{z})^{-1} = \Lambda \mathcal{T}^D (A_0 - \bar{z})^{-1} = \Lambda \Phi_z'
  \end{equation*}
  with $\Phi_z'$ given by~\eqref{def_Phi_z_prime}, so that the operator $G_z$ from~\eqref{eq-gz1} in the present context is given by 
  \begin{equation} \label{gamma_our_triple}
    G_z = \Phi_z \Lambda.
  \end{equation}
  Let $\zeta \in \res A_0$ be fixed. Then, by \eqref{deco234} any $f \in \dom S^*$ can be written as
  \begin{equation*}
    f = f_\zeta + G_\zeta \xi = f_{\bar{\zeta}} + G_{\bar{\zeta}} \xi
  \end{equation*}
  for some $\xi \in L^2(\Sigma; \mathbb{C}^2)$ and $f_\zeta, f_{\bar{\zeta}} \in H^1(\mathbb{R}^2; \mathbb{C}^2)$, and according to 
  Proposition~\ref{proposition_boundary_triple_singular_perturbation}
  \begin{equation*}
    \Gamma_0 f = \xi \quad \text{and} \quad \Gamma_1 f = \frac{1}{2} \big( \mathcal{T} f_\zeta + \mathcal{T} f_{\bar{\zeta}} \big)
  \end{equation*}
  defines a boundary triple for $S^*$ such that  $A_0=S^*\upharpoonright\ker\Gamma_0$.
  
  Next we show that the above boundary maps coincide with the more explicit representations of $\Gamma_0$ and $\Gamma_1$ stated in the proposition. Let $f = f_\zeta + G_\zeta \xi = f_\zeta + \Phi_\zeta \Lambda \xi$ with $\xi \in L^2(\Sigma; \mathbb{C}^2)$ and $f_\zeta \in H^1(\mathbb{R}^2; \mathbb{C}^2)$ be fixed.
  Using that the jump of the trace of $f_\zeta \in H^1(\mathbb{R}^2; \mathbb{C}^2)$ at $\Sigma$ is zero 
  and the trace formula from 
  Proposition~\ref{proposition_Plemelj_Sokhotskii}
  we find 
  \begin{equation*}
  \begin{split}
   \mathcal{T}^D_+ f_+ - \mathcal{T}^D_- f_-&=\mathcal{T}^D_+\bigl(f_\zeta + \Phi_\zeta \Lambda \xi\bigr)_+-\mathcal{T}^D_-\bigl(f_\zeta + \Phi_\zeta \Lambda \xi\bigr)_-\\
   &=\mathcal{T}^D_+\bigl(\Phi_\zeta \Lambda \xi\bigr)_+-\mathcal{T}^D_-\bigl(\Phi_\zeta \Lambda \xi\bigr)_-\\
   &=-\frac{\rmi}{2}(\sigma \cdot \nu) \Lambda \xi + \mathcal{C}_\zeta \Lambda \xi -\frac{\rmi}{2}(\sigma \cdot \nu) \Lambda \xi - \mathcal{C}_\zeta \Lambda \xi \\
   &= -\rmi (\sigma \cdot \nu) \Lambda \xi.
  \end{split}
  \end{equation*}
  Hence,
  \begin{equation*}
   \Gamma_0 f = \xi =\rmi \Lambda^{-1} (\sigma \cdot \nu) \big( \mathcal{T}^D_+ f_+ - \mathcal{T}^D_- f_- \big),
  \end{equation*}
  which is the claimed formula for $\Gamma_0 f$. Employing again Proposition~\ref{proposition_Plemelj_Sokhotskii} we find
  \begin{equation} \label{equation_Gamma_11}
    \begin{split}
      \mathcal{T}^D f_\zeta &= \frac{1}{2} \big( \mathcal{T}_+^D f_{\zeta, +} + \mathcal{T}^D_- f_{\zeta, -} \big)\\
          &= \frac{1}{2} \big( \mathcal{T}_+^D (f - \Phi_\zeta \Lambda \xi)_+ + \mathcal{T}^D_- (f - \Phi_\zeta \Lambda \xi)_- \big) \\
      &= \frac{1}{2} \left( \mathcal{T}^D_+ f_+- \mathcal{C}_\zeta \Lambda \xi + \frac{\rmi}{2} (\sigma \cdot \nu) \Lambda \xi + \mathcal{T}^D_- f_- - \mathcal{C}_\zeta \Lambda \xi - \frac{\rmi}{2} (\sigma \cdot \nu) \Lambda \xi \right) \\
      &= \frac{1}{2} \left( \mathcal{T}^D_+ f_+ + \mathcal{T}^D_- f_-\right) - \mathcal{C}_\zeta \Lambda \xi \\
      &= \frac{1}{2} \left( \mathcal{T}^D_+ f_+ + \mathcal{T}^D_- f_-\right) - \mathcal{C}_\zeta \Lambda \Gamma_0 f
    \end{split}
  \end{equation}
  and analogously
  \begin{equation} \label{equation_Gamma_12}
    \mathcal{T}^D f_{\Bar{\zeta}} = \frac{1}{2} \left( \mathcal{T}^D_+ f_+ + \mathcal{T}^D_- f_-\right) - \mathcal{C}_{\Bar{\zeta}} \Lambda \Gamma_0 f.
  \end{equation}
  By summing up the last two formulae~\eqref{equation_Gamma_11} and~\eqref{equation_Gamma_12} we find
  \begin{equation*}
    \Gamma_1 f = \frac{1}{2} \big( \mathcal{T} f_\zeta + \mathcal{T} f_{\bar{\zeta}} \big) = \frac{1}{2} \Lambda \big( \mathcal{T}^D f_\zeta + \mathcal{T}^D f_{\bar{\zeta}} \big)
        = \frac{1}{2} \Lambda \Big((\cT_+^D f_+ + \cT_-^D f_-) -(\mathcal{C}_\zeta+\mathcal{C}_{\Bar\zeta}) \Lambda \Gamma_0 f \Big),
  \end{equation*}
  which is the claimed formula for $\Gamma_1$ in~\eqref{def_Gamma}.

  Finally, the claimed representation of the $\gamma$-field follows from Proposition~\ref{proposition_boundary_triple_singular_perturbation} and~\eqref{gamma_our_triple}. Using again Proposition~\ref{proposition_Plemelj_Sokhotskii}, we can simplify the formula for the Weyl function $M_z$ from Proposition~\ref{proposition_boundary_triple_singular_perturbation} and get for $\varphi \in L^2(\Sigma; \mathbb{C}^2)$
  \begin{equation*}
    \begin{split}
      M_z \varphi &= \mathcal{T} \left( G_z - \frac{1}{2} (G_\zeta + G_{\bar{\zeta}}) \right) \varphi\\
       &= \Lambda \mathcal{T}^D_+ \left( \Phi_z - \frac{1}{2} (\Phi_\zeta + \Phi_{\bar{\zeta}}) \right) \Lambda \varphi \\
      &= \Lambda \left( \mathcal{C}_z - \frac{\rmi}{2} (\sigma \cdot \nu) - \frac{1}{2} \left(\mathcal{C}_\zeta - \frac{\rmi}{2} (\sigma \cdot \nu) + \mathcal{C}_{\bar{\zeta}} - \frac{\rmi}{2} (\sigma \cdot \nu) \right) \right) \Lambda \varphi \\
      &= \Lambda \left( \mathcal{C}_z - \frac{1}{2} \left(\mathcal{C}_\zeta + \mathcal{C}_{\bar{\zeta}} \right) \right) \Lambda \varphi.
    \end{split}
  \end{equation*}
  Remark that in the above computation we used the well-known regularization property $(G_z - \frac{1}{2} (G_\zeta + G_{\bar{\zeta}})) \varphi\in\dom A_0=H^1(\mathbb{R}^2; \mathbb{C}^2) $, which holds automatically 
  by the abstract theory (see the formula for the Weyl function in 
  Proposition~\ref{proposition_boundary_triple_singular_perturbation}), and hence $\mathcal{T}^D$ and $\mathcal{T}^D_+$ lead to the same trace in the second equality above.
  Therefore, all claimed statements have been shown.
\end{proof}

Finally, we state an auxiliary regularity result that will be used later.

\begin{lem}\label{lem23}
Let $f \in \dom S^*$. Then $f \in H^1(\mathbb{R}^2 \setminus \Sigma; \mathbb{C}^2)$ if and only if $\Gamma_0 f \in H^1(\Sigma; \mathbb{C}^2)$.
\end{lem}

\begin{proof}
First, if $f = f_+ \oplus f_- \in H^1(\mathbb{R}^2 \setminus \Sigma; \CC^2)$, then one has $\cT_\pm^D f_\pm \in H^{\frac{1}{2}}(\Sigma; \CC^2)$
implying $\cT_+^D f_+ - \cT_-^D f_- \in H^{\frac{1}{2}}(\Sigma; \CC^2)$.
As $\sigma\cdot \nu$ is a  $C^\infty$-matrix function it follows that
$\rmi (\sigma\cdot \nu)\big( \cT_+^D f_+ - \cT_-^D f_- \big)\in H^{\frac{1}{2}}(\Sigma; \CC^2)$. Using that $\Lambda$ is a bijection from $H^s(\Sigma)$ to $H^{s-\frac{1}{2}}(\Sigma)$ for all $s \in \mathbb{R}$, this yields
\begin{equation*}
  \Gamma_0 f = \rmi \Lambda^{-1} (\sigma\cdot \nu)\big( \cT_+^D f_+ - \cT_-^D f_- \big)\in H^1(\Sigma; \CC^2).
\end{equation*}

Conversely, let $f = f_+ \oplus f_- \in\dom S^*$ with $\Gamma_0 f\in H^{1}(\Sigma;\CC^2)$. Since $\Lambda: H^1(\Sigma) \rightarrow H^{\frac{1}{2}}(\Sigma)$ 
is bijective and the $C^\infty$-matrix function $\sigma \cdot \nu$ is invertible we conclude from 
the definition of $\Gamma_0$ that 
\begin{equation}\label{formel11}
\cT_+^D f_+ - \cT_-^D f_- \in H^{\frac{1}{2}}(\Sigma; \CC^2).
\end{equation}
By Proposition~\ref{proposition_C_z} the operators $\mathcal{C}_\zeta$ and $\mathcal{C}_{\Bar{\zeta}}$ are bounded in $H^{\frac{1}{2}}(\Sigma; \mathbb{C}^2)$,
which gives $(\mathcal{C}_\zeta + \mathcal{C}_{\bar\zeta})\Lambda \Gamma_0 f\in H^{\frac{1}{2}}(\Sigma; \CC^2)$. In addition, $\Gamma_1 f\in L^2(\Sigma;\CC^2)$
implies $\Lambda^{-1}\Gamma_1\in H^{\frac{1}{2}}(\Sigma; \CC^2)$. With the definition of $\Gamma_1$ this yields
\begin{equation*}
  \frac{1}{2} \big( \cT_+^D f_+ + \cT_-^D f_-\big)=\Lambda^{-1}\Gamma_1 f
+\frac{1}{2} (\mathcal{C}_\zeta+\mathcal{C}_{\Bar\zeta}) \Lambda \Gamma_0 f \in H^{\frac{1}{2}}(\Sigma; \mathbb{C}^2).
\end{equation*}
Hence, together with \eqref{formel11} this implies $\cT_\pm^D f_\pm \in H^{\frac{1}{2}}(\Sigma;\CC^2)$.
Finally, Lemma~\ref{lemma_trace} shows $f_\pm\in H^1(\Omega_\pm;\CC^2)$. 
\end{proof}

\subsection{\kp{Some basic properties of self-adjoint extensions}}

\kp{
In this subsection we prove two results which are valid for the essential and discrete spectra of a large class of self-adjoint extensions of $S$ defined in~\eqref{def_S}
and which are independent of the preceding construction of a boundary triple. These properties will be used
later for a more detailed spectral analysis of $A_{\eta, \tau}$.
}

For the essential spectrum of we have the following result:
\begin{prop}\label{prop-ess}
For any self-adjoint extension $A$ of $S$ one has the inclusion
\[
\big(-\infty,-|m|\big]\,\cup\,\big[|m|,+\infty\big)\subset \spec_\ess A.
\]
\end{prop}

\begin{proof}
The proof is in the standard way by constructing for each $z \in (-\infty, -|m|) \cup (|m|, \infty)$
	a sequence of functions $f_n\in \dom A_{\eta, \tau}$ satisfying $\big\|(A_{\eta, \tau}-z)f_n\big\|/\|f_n\| \rightarrow 0$ as $n\to \infty$.
	For example, following \cite[Theorem~5.7~(i)]{BH} for the three-dimensional analog we define
  \begin{equation*} 
    f_n(x_1, x_2) := \frac{1}{n} \chi\left(\frac{1}{n} |x-y_n| \right) e^{\rmi \sqrt{z^2 - m^2} x_1} \big(\sqrt{z^2 - m^2} \sigma_1 + m \sigma_3 + z \sigma_0\big) \zeta,
  \end{equation*}
  where $\chi:\RR\to [0, 1]$ is a $C^\infty$-function such that $\chi(t)=1$ for $|t|\le \frac{1}{2}$ and 
$\chi(t)=0$ for $|t|\ge 1$, the vector $\zeta \in \mathbb{C}^2$ is chosen such that $(\sqrt{z^2 - m^2} \sigma_1 + m \sigma_3 + z \sigma_0) \zeta \neq 0$, 
  the number $R > 0$ is sufficiently large to have $\RR^2 \setminus \overline{B(0, R)} \subset \Omega_-$, and we denote $y_n := (R + n^2, 0)$, $n \in \NN$. 
  Then $f_n \in \dom S \subset \dom A$ and one can show as in \cite[Theorem~5.7~(i)]{BH} that $\big\|(A_{\eta, \tau}-z)f_n\big\|/\|f_n\| \rightarrow 0$ for $n\to \infty$. Since $z \in (-\infty, -|m|) \cup (|m|, \infty)$ was arbitrary, the claimed result follows.
  \end{proof}

Some information about the discrete spectrum can be obtained under an additional regularity assumption:
\begin{prop} \label{prop-discr}  
  Let $A$ be a self-adjoint extension of the 
  symmetric operator $S$ in $L^2(\mathbb{R}^2; \mathbb{C}^2)$ satisfying the inclusion $\dom A \subset H^s(\mathbb{R}^2 \setminus \Sigma; \mathbb{C}^2)$ for some $s > 0$.
	Then the spectrum of $A$ in $\big(-|m|, |m|\big)$ is purely discrete and finite.
\end{prop}
\begin{proof}
  It is sufficient to show that $A^2$ has at most finitely many eigenvalues in $(-\infty, m^2)$. For that, consider the quadratic form
\[
a[f,f]=\int_{\RR^2} |A f|^2 \dd x, \quad \dom a= \dom A.
\]
Since $A$ is self-adjoint and hence closed, also the densely defined nonnegative form $a$ is closed. The self-adjoint operator associated to  $a$ via the first representation theorem  is $A^2$.
Next, take $0<r<R$ with $r$ chosen sufficiently large, such that the open ball $B_r=\{x\in\RR^2:\vert x\vert<r\}$ 
contains $\overline{\Omega_+}$ in its interior, 
and choose $\varphi_1, \varphi_2\in C^\infty(\RR^2)$ which satisfy
\[
0\le \varphi_1, \varphi_2\le 1, \quad
\varphi_1^2+\varphi_2^2=1,
\quad
\varphi_1=1 \text{ in $B_r$},
\quad
\varphi_2=1 \text{ in $\RR^2\setminus B_R$.}
\]
Let $f\in \dom A$ be fixed. Then by construction one has $\varphi_j f \in \dom A$
and
\[
A (\varphi_j f)=\varphi_j A f -\rmi \sigma\cdot (\nabla\varphi_j) f.
\]
In particular, we note that $\varphi_2 f \in H(\sigma, \Omega_-)$ with $\mathcal{T}_-^D f = 0 \in H^{\frac{1}{2}}(\Sigma; \mathbb{C}^2)$. Thus, 
it follows from Lemma~\ref{lemma_trace} that $\varphi_2 f \in H^1(\Omega_-; \mathbb{C}^2)$.

Next, we remark that $\nabla\varphi_j$ is supported in $\overline{B_R}\setminus B_r$. Hence, we have for $j \in \{ 1, 2\}$
\[
a[\varphi_j f,\varphi_j f]=\int_{\RR^2} \big(\varphi_j^2|A f|^2+ |\rmi \sigma\cdot (\nabla\varphi_j) f|^2\big)\dd x + \mathcal{I}_j,
\]
where
\begin{equation*}
\begin{split}
\mathcal{I}_j&=\int_{B_R\setminus B_r} 2 \, \textup{Re}\,  \big(\varphi_j 
(-\rmi \sigma \cdot \nabla + m \sigma_3) f, -\rmi \sigma\cdot (\nabla\varphi_j) f\big)_{\CC^2}\dd x\\
&=\int_{B_R\setminus B_r} 2 \, \textup{Re}\, \big( (-\rmi \sigma \cdot \nabla + m \sigma_3) f, -\rmi \sigma\cdot (\varphi_j \nabla\varphi_j) f\big)_{\CC^2} \dd x\\
&=\int_{B_R\setminus B_r} \textup{Re}\, \big( (-\rmi \sigma \cdot \nabla + m \sigma_3) f, -\rmi \sigma\cdot  \nabla(\varphi_j^2) f\big)_{\CC^2} \dd x.
\end{split}
\end{equation*}
From $\varphi_1^2+\varphi_2^2=1$ we obtain $\nabla(\varphi_1^2)=-\nabla(\varphi_2^2)$ and hence $\mathcal{I}_1=-\mathcal{I}_2$. Moreover, using 
\eqref{anti_commutation} one verifies $|\rmi \sigma\cdot (\nabla\varphi_j) f|^2= |\nabla\varphi_j\vert^2\vert f|^2$ for $j \in \{ 1, 2\}$. 
Therefore, it follows that
\begin{equation*}
\begin{split}
&a[\varphi_1 f,\varphi_1 f]+a[\varphi_2 f,\varphi_2 f]\\
&\qquad =\int_{\RR^2} ( \varphi_1^2+\varphi_2^2) |A f|^2\dd x
 +\int_{\RR^2} \big(|\nabla\varphi_1|^2+|\nabla\varphi_2|^2\big)|f|^2\big)\dd x\\
&\qquad = \int_{\RR^2} |A f|^2\dd x +\int_{\RR^2} V|f|^2\dd x,
\end{split}
\end{equation*}
where we have used the abbreviation $V:=|\nabla\varphi_1|^2+|\nabla\varphi_2|^2$ in the last step; note that $V$ is supported in $\overline{B_R}\setminus B_r$.
This leads to
\begin{equation}
   \label{buu}
a[f,f]=a[\varphi_1 f,\varphi_1 f]-\int_{\RR^2} V |\varphi_1 f|^2\dd x
+a[\varphi_2 f,\varphi_2 f]-\int_{\RR^2} V |\varphi_2 f|^2\dd x.
\end{equation}

In the following we will often restrict functions in $\dom a$ to $B_R$ or $\RR^2\setminus \overline{B_r}$ and view them as elements in $L^2(B_R;\CC^2)$ or 
$L^2(\RR^2\setminus \overline{B_r};\mathbb{C}^2)$, or we will extend $L^2$-functions on 
$B_R$ or $\RR^2\setminus \overline{B_r}$ by zero onto $\RR^2$ and view them as elements in $L^2(\RR^2;\CC^2)$. We find it convenient to use 
the same letter for the original and the restricted or extended function.

Let $a_1$ be the quadratic form in $L^2(B_R;\CC^2)$ defined by
\[
\dom a_1=\big\{g \in \dom a : \supp g \subset \overline{B_R}\big\},
\quad
a_1[g,g]=a[g,g]-\int_{B_R} V |g|^2\dd x.
\]
As $V$ is bounded and $a$ is nonnegative it follows that $a_1$ is semibounded from below. It is also clear that $a_1$ is densely defined in $L^2(B_R;\CC^2)$. To see that $a_1$ is closed consider 
$g_n\in \dom a_1$ such that $g_n \rightarrow g$ in $L^2(B_R;\CC^2)$ for $n\rightarrow\infty$ and $a_1(g_n - g_m, g_n-g_m) \rightarrow 0$
for $n,m\rightarrow\infty$. Since $V$ is bounded it follows that the zero extensions $g_n$ and $g$ satisfy 
$g_n \rightarrow g$ in $L^2(\RR^2;\CC^2)$ for $n\rightarrow\infty$ and $a(g_n - g_m, g_n-g_m) \rightarrow 0$
for $n,m\rightarrow\infty$. As $a$ is closed we conclude $g\in\dom a$ and $a(g_n - g, g_n-g)\rightarrow 0$ for $n\rightarrow\infty$.
Furthermore, as $\supp g \subset \overline{B_R}$ we have $g\in\dom a_1$ and $a_1(g_n - g, g_n-g)\rightarrow 0$ for $n\rightarrow\infty$,
thus $a_1$ is closed. Let $A_1$ be the self-adjoint operator in $L^2(B_R; \mathbb{C}^2)$ corresponding to $a_1$.
Then $A_1$ has a compact resolvent since the form domain $\dom a_1\subset H^s(B_R\setminus\Sigma;\CC^2)$
is compactly embedded in $L^2(B_R; \mathbb{C}^2)$ for $s>0$. Hence, the number of eigenvalues $\mathcal{N}(A_1, m^2)$ of $A_1$ below $m^2$ is finite, that is, 
$\mathcal{N}(A_1, m^2)<\infty$.

Next, let $a_2$ be the quadratic form in $L^2(\RR^2\setminus \overline{B_r};\CC^2)$ defined by
\begin{gather*}
\dom a_2=H^1_0\big(\RR^2\setminus \overline{B_r};\CC^2\big), \quad
a_2[g,g]=a[g,g]-\int_{\RR^2\setminus \overline{B_r}} V |g|^2\dd x.
\end{gather*}
As above it is clear that $a_2$ is densely defined and semibounded from below.
Using integration by parts and~\eqref{anti_commutation} one sees for $g\in C^\infty_0(\RR^2\setminus \overline{B_r};\CC^2)$ that
\begin{equation*}
\begin{split}
a[g,g]&=\int_{\RR^2\setminus \overline{B_r}} |(-\rmi \sigma \cdot \nabla + m \sigma_3) g|^2 \dd x\\
&=\int_{\RR^2\setminus \overline{B_r}} \big( g, (-\rmi \sigma \cdot \nabla + m \sigma_3)^2 g \big)_{\CC^2}\dd x\\
&= \int_{\RR^2\setminus \overline{B_r}} \big( g, (-\Delta+m^2) g\big)_{\CC^2}\dd x\\
&=\int_{\RR^2\setminus \overline{B_r}} \big( |\nabla g|^2 +m^2 |g|^2\big)\dd x,
\end{split}
\end{equation*}
which then extends by density to all $g\in H^1_0\big(\RR^2\setminus \overline{B_r};\CC^2\big)$. Therefore,
the form $a_2$ is closed and the self-adjoint operator associated to $a_2$ is $A_2=-\Delta^D +m ^2-V$, where $-\Delta^D$ denotes 
the Dirichlet Laplacian in $\RR^2\setminus \overline{B_r}$.

\kp{
Let us prove that $\cN(A_2, m^2)<\infty$. Recall that $V$ is bounded and that its support
is contained in $\overline{B_R}$. Consider the following
closed sesquilinear forms $a_3$ in $L^2(B_R\setminus \overline{B_r})$ and $a_4$ in $L^2(\RR^2\setminus \overline{B_R})$,
\begin{align*}
a_3[g,g]&= \int_{B_R\setminus \overline{B_r}} \big( |\nabla g|^2 +(m^2-V) |g|^2\big)\dd x,\\
& \quad \dom a_3=\big\{g\in H^1(B_R\setminus \overline{B_r};\CC^2):\, g=0 \text{ on } \partial B_r \big\},\\
a_4[g,g]&= \int_{\RR^2\setminus \overline{B_R}} \big( |\nabla g|^2 +m^2 |g|^2\big)\dd x, \quad \dom a_4=H^1(\RR^2\setminus \overline{B_R}).
\end{align*}
For any $g\in \dom a_2$ one has
\[
f_3:=g\upharpoonright B_R\setminus \overline{B_r}\in \dom a_3,
\quad
f_4:=g\upharpoonright \RR^2\setminus \overline{B_R} \in \dom a_4,
\]
with $a_2(g,g)=a_3(f_3,f_3)+a_4(f_4,f_4)$.
Therefore, if $A_3$ is the self-adjoint operator in $L^2(B_R\setminus \overline{B_r})$ generated by $a_3$
and $A_4$ is the self-adjoint operator in $L^2(\RR^2\setminus \overline{B_R})$ generated by $a_4$, then it follows
by the min-max principle that the eigenvalues of $a_2$ are bounded from below by the respective eigenvalues of $A_3\oplus A_4$.
In particular, $\cN(A_2, m^2)\le \cN(A_3, m^2)+\cN(A_4, m^2)$. One clearly has $\cN(A_4, m^2)=0$. On the other hand, the operator
$A_3$ is semibounded from below and has a compact resolvent, hence, $\cN(A_3, m^2)<\infty$. This implies $\cN(A_2, m^2)<\infty$.
}

Now, we can conclude that $A^2$ has only finitely many eigenvalues below $m^2$. For this consider
\begin{equation*}
  J:L^2(\RR^2;\CC^2)\to L^2(B_R;\CC^2)\oplus L^2(\RR^2\setminus \overline{B_r};\CC^2), \qquad 
  J f= \varphi_1 f \oplus \varphi_2 f.
\end{equation*}
Due to the properties of $\varphi_1$ and $\varphi_2$ we get that $J$ is an isometry. Moreover, with the above considerations we see $J(\dom a) \subset \dom a_1\oplus \dom a_2$, and with the equality \eqref{buu}
we obtain
\begin{equation*}
  \frac{a[f,f]}{\| f \|_{L^2(\mathbb{R}^2; \mathbb{C}^2)}^2} = \frac{(a_1\oplus a_2) [J f, J f]}{\| J f \|_{L^2(B_R;\CC^2)\oplus L^2(\RR^2\setminus \overline{B_r};\CC^2)}^2}.
\end{equation*} 
It follows from the min-max principle that
\begin{equation*}
  \cN(A^2, m^2)\le \cN(A_1\oplus A_2, m^2) = \cN(A_1, m^2)+\cN(A_2, m^2).
\end{equation*}
As we have seen above, the quantity on the right hand side is finite and hence $\cN(A^2, m^2)< \infty$. This completes the proof.
\end{proof}

\section{Dirac operators with singular interactions} \label{section_delta_op}

In this section we study the Dirac operator $A_{\eta, \tau}$ introduced in~\eqref{dirdelta} and we prove
the main results of this paper. First, in Section~\ref{section_def_op} we show how $A_{\eta, \tau}$ is related to the boundary triple 
$\{ L^2(\Sigma; \mathbb{C}^2), \Gamma_0, \Gamma_1 \}$ from Proposition~\ref{prop-bt3}. Then, in Section~\ref{sec-noncrit}, we show the 
self-adjointness of $A_{\eta, \tau}$ for non-critical interaction strengths, i.e. when $\eta^2 - \tau^2 \neq 4$, and investigate the spectral properties 
of $A_{\eta, \tau}$ in this setting. In 
Section~\ref{sec-crit} we the study the self-adjointness and the spectral properties of $A_{\eta, \tau}$ in the case of critical interaction strengths.
Finally, in Section~\ref{secsevloops} we provide a sketch of the proof of Theorem~\ref{sevloops}.

\subsection{Definition of $A_{\eta, \tau}$ via the boundary triple} \label{section_def_op}

Recall the definition of the space $H(\sigma, \Omega_\pm)$ from~\eqref{def_H_sigma}, the trace maps $\mathcal{T}_\pm^D$ on $H(\sigma, \Omega_\pm)$ in 
Lemma~\ref{lemma_trace}, and that the operator $A_{\eta, \tau}$ in \eqref{dirdelta} is defined by
\begin{equation}\label{eqn:defdiracdelta}
  \begin{split}
    A_{\eta, \tau} f &= (-\rmi \sigma \cdot \nabla + m \sigma_3) f_+ \oplus (-\rmi \sigma \cdot \nabla + m \sigma_3) f_-, \\
    \dom A_{\eta, \tau} &= \Big\{ f = f_+ \oplus f_- \in H(\sigma,\Omega_+) \oplus H(\sigma,\Omega_-):\\ 
	& \qquad \quad {}-\rmi (\sigma\cdot \nu)\big( \cT_+^D f_+ - \cT_-^D f_-\big) = \frac12(\eta \sigma_0 + \tau \sigma_3)\big(\cT_+^D f_+ + \cT_-^D f_-\big)\Big\}.
	  \end{split}
\end{equation}

Before analyzing the properties of $A_{\eta, \tau}$ we would like to mention that for special values of the interaction strengths $A_{\eta, \tau}$ decouples 
in Dirac operators in $L^2(\Omega_+;\CC^2)$ and $L^2(\Omega_-;\CC^2)$ subject to certain 
boundary conditions. Similar effects are known from dimension three, see \cite[Section~V]{DES89}, \cite[Section~5]{AMV15}, and \cite[Lemma~3.1]{BEHL19}. 
The result reads as follows:

\begin{lem} \label{lemma_confinement}
  Let $\eta, \tau \in \mathbb{R}$. Then the following holds:
  \begin{itemize}
    \item[\textup{(i)}] If $\eta^2 - \tau^2 \neq -4$, then there is an invertible matrix $M$, which is explicitly given below in~\eqref{def_M}, such that $f = f_+ \oplus f_- \in \dom A_{\eta, \tau}$ if and only if
    \begin{equation*}
      \mathcal{T}_+^D f_+ = M \mathcal{T}_-^D f_-.
    \end{equation*}
    \item[\textup{(ii)}] If $\eta^2 - \tau^2 = -4$, then $A_{\eta, \tau} = A_+ \oplus A_-$, where $A_\pm$ is a Dirac operator in $L^2(\Omega_\pm; \mathbb{C}^2)$ and $f_\pm \in \dom A_\pm$ if and only if
    \begin{equation} \label{boundary_condition}
      \mathcal{T}_\pm^D f_\pm = \pm \frac{\rmi}{2} (\sigma \cdot \nu) \left( \eta \sigma_0 + \tau \sigma_3 \right) \mathcal{T}_\pm^D f_\pm.
    \end{equation}
  \end{itemize}
\end{lem}

\begin{remark}
  Assume that $\eta^2 - \tau^2 = -4$, which is equivalent to $\frac{\eta^2}{\tau^2} + \frac{4}{\tau^2} = 1$.  Thus, there exists $\vartheta \in [0, 2 \pi] \setminus \{ \frac{\pi}{2}, \frac{3 \pi}{2} \}$ such that
\begin{equation*}
  \frac{\eta}{\tau} = -\sin \vartheta \quad \text{and} \quad \frac{2}{\tau} = \cos \vartheta.
\end{equation*}  
Using~\eqref{anti_commutation} we see that~\eqref{boundary_condition} for $f_+$ is equivalent to
\begin{equation*}
  \begin{split}
    0&=\frac{2i}{\tau} \sigma_3 (\sigma \cdot \nu) \left( \sigma_0 - \frac{\rmi}{2} (\sigma \cdot \nu) \left( \eta \sigma_0 + \tau \sigma_3 \right) \right) \mathcal{T}^D_+ f_+ \\
    &=  \big( \sigma_0 + \rmi \sigma_3 (\sigma \cdot \nu) \cos \vartheta - \sin\vartheta \sigma_3 \big) \mathcal{T}_+^D f_+,
  \end{split}
\end{equation*}
i.e. the operators $A_+$ in the bounded domain $\Omega_+$ are exactly those investigated in~\cite{BFSV}. 
The case $\vartheta=0$ corresponds to the well-known infinite mass boundary condition, which is the two dimensional analog of the MIT bag boundary condition,
studied in~\cite{ALTMR,MOBP,SW}.
We would like to point out that our results on $A_{\eta, \tau}$ obtained later in Section~\ref{sec-noncrit} can be used for a deeper understanding for $A_\pm$.
\end{remark}

\begin{proof}[Proof of Lemma~\ref{lemma_confinement}]
The transmission condition in the definition of $A_{\eta, \tau}$ can be written in the form
\[
  \Big(\rmi  (\sigma\cdot \nu) +  \frac12(\eta \sigma_0 + \tau \sigma_3)\Big)\, \mathcal{T}^D_+ f_+
	=
	\Big(\rmi  (\sigma\cdot \nu) -  \frac12(\eta \sigma_0 + \tau \sigma_3)\Big)\, \mathcal{T}_-^Df_-.
\]
Multiplying this equation with $-\rmi\,(\sigma\cdot\nu)$ we obtain the equivalent form
\begin{equation} \label{tran4}
  (\sigma_0 -  R)\, \mathcal{T}_+^D f_+ 	= 	(\sigma_0 +R)\, \mathcal{T}_-^D f_-
\end{equation}
with
\begin{equation*}
	R:=\dfrac{\rmi}{2}\,(\sigma\cdot\nu)(\eta \sigma_0 + \tau \sigma_3) = \dfrac{\rmi}{2}\,(\eta \sigma_0 - \tau \sigma_3)(\sigma\cdot\nu), \nonumber
\end{equation*}
where \eqref{anti_commutation} was used. One computes
\[
R^2=\dfrac{\rmi}{2}\,(\eta \sigma_0 - \tau \sigma_3)(\sigma\cdot\nu)\,\dfrac{\rmi}{2}\,(\sigma\cdot\nu)(\eta \sigma_0 + \tau \sigma_3)=-\dfrac{\eta^2-\tau^2}{4}\,\sigma_0,
\]
which implies
$$(\sigma_0 -  R)(\sigma_0 +R)=\sigma_0 - R^2= \sigma_0 + \dfrac{\eta^2-\tau^2}{4}\,\sigma_0.$$ 
Assume now $\eta^2-\tau^2\ne -4$. Then both $\sigma_0\pm R$ are invertible with 
\[
(\sigma_0 \pm R)^{-1} = \frac{4}{4 + \eta^2 - \tau^2} \,(\sigma_0 \mp R).
\]
Therefore, the transmission condition can be equivalently rewritten as
\begin{equation} \label{def_M}
\mathcal{T}_+^D f_+=(\sigma_0 -R)^{-1}(\sigma_0+R)\, \mathcal{T}^D_- f_-\quad \text{or}\quad \mathcal{T}^D_- f_-=(\sigma_0 +R)^{-1}(\sigma_0-R)\, \mathcal{T}^D_+ f_+,
\end{equation}
which shows assertion~(i).
On the other hand, for $\eta^2-\tau^2 = -4$ one has $R^2=\sigma_0$ and multiplying \eqref{tran4} by $\sigma_0 -  R$ or $\sigma_0 +  R$ leads to the two conditions
\[
\mathcal{T}_\pm^D f_\pm=\pm R \mathcal{T}_\pm^D f_\pm.
\]
It follows that the operator $A_{\eta, \tau}$ decouples in an orthogonal sum of operators $A_\pm$ acting in $\Omega_\pm$ and hence, also statement~(ii) has been shown. 
\end{proof}

Let us represent $A_{\eta, \tau}$ using the boundary triple $\{ L^2(\Sigma; \CC^2),\Gamma_0,\Gamma_1\}$
constructed in Proposition~\ref{prop-bt3}. Note that the definition of $\Gamma_0$ and $\Gamma_1$ can be rewritten as
\begin{align}\label{eqn:transm1}
\rmi (\sigma\cdot \nu)\,\big( \cT_+^D f_+ - \cT_-^D f_-\big)&=\Lambda \Gamma_0 f,\\
\label{eqn:transm2}
\dfrac{1}{2}\, \big( \cT_+^D f_+ + \cT_-^D f_-\big)&=\Lambda^{-1}\Gamma_1 f
+\dfrac{1}{2}\,(\mathcal{C}_\zeta+\mathcal{C}_{\Bar\zeta}) \Lambda \Gamma_0 f.
\end{align}

\begin{prop}\label{prop-bgamma}
 Let $\eta, \tau \in \mathbb{R}$. Then the following holds:
\begin{enumerate}
\item[\textup{(i)}] Assume $|\eta|\ne |\tau|$. Let  $\Theta$ be the linear operator in $L^2(\Sigma;\CC^2)$ obtained as the maximal realization of the periodic pseudodifferential operator $\theta\in \Psi^1_\Sigma$ given by
\begin{equation} \label{def_Theta}
\theta=-\Lambda\bigg[\frac1{\eta^2 - \tau^2}(\eta\sigma_0 - \tau\sigma_3) + \frac12\,(\mathcal{C}_\zeta + \mathcal{C}_{\bar\zeta})\bigg]\Lambda,
\end{equation}
i.e. $\dom \Theta=\big\{ \varphi \in L^2(\Sigma;\CC^2):\, \theta \varphi\in L^2(\Sigma;\CC^2)\big\}$ and $\Theta \varphi=\theta \varphi$. Then
\begin{equation}
    \label{domb1}
\dom A_{\eta, \tau} = \big\{ f \in \dom S^* : \Gamma_0 f \in \dom \Theta,\  \Gamma_1 f = \Theta \Gamma_0 f
\big\}.
\end{equation}
\item[\textup{(ii)}] Assume $\eta=\tau\neq 0$, let 
$$\Pi_+:L^2(\Sigma; \mathbb{C}^2)\rightarrow L^2(\Sigma),\quad \begin{pmatrix} \varphi_1\\ \varphi_2\end{pmatrix}\mapsto \varphi_1,$$ 
and let $\Theta_+$ be the linear operator in $L^2(\Sigma)$ obtained as the maximal realization of the periodic pseudodifferential 
operator $\theta_+\in \Psi^1_\Sigma$ given by
\begin{equation} \label{def_Theta_plus}
\theta_+=-\Lambda\Big(\dfrac{1}{2\eta} + \Pi_+\, \dfrac{1}{2}(\mathcal{C}_\zeta+\mathcal{C}_{\Bar\zeta})\Pi_+^*\Big) \Lambda,
\end{equation}
i.e. $\dom \Theta_+ = \big\{ \varphi\in L^2(\Sigma): \theta_+ \varphi \in L^2(\Sigma)\big\}$ and $\Theta_+ \varphi = \theta_+ \varphi$.
Then
\begin{equation} \label{domb2}
\dom A_{\eta, \tau}=\Big\{
f \in \dom S^* : \Pi_+\Gamma_1 f = \Theta_{+} \Pi_+\Gamma_0 f,\, (\sigma_0-\Pi_+^* \Pi_+)\Gamma_0 f=0
\Big\}.
\end{equation}
\item[\textup{(iii)}] Assume $\eta=-\tau\neq 0$, let 
$$\Pi_-:L^2(\Sigma; \mathbb{C}^2)\rightarrow L^2(\Sigma),\quad \begin{pmatrix} \varphi_1\\ \varphi_2\end{pmatrix}\mapsto \varphi_2,$$ 
and let $\Theta_-$ be the linear operator in $L^2(\Sigma)$ obtained as the maximal realization of the periodic pseudodifferential 
operator $\theta_-\in \Psi^1_\Sigma$ given by
\begin{equation} \label{def_Theta_minus}
\theta_-=-\Lambda\Big(\dfrac{1}{2\eta} + \Pi_-\, \dfrac{1}{2} (\mathcal{C}_\zeta+\mathcal{C}_{\Bar\zeta})\Pi_-^*\Big) \Lambda,
\end{equation}
i.e. $\dom \Theta_- = \big\{ \varphi\in L^2(\Sigma): \theta_- \varphi \in L^2(\Sigma)\big\}$ and $\Theta_- \varphi = \theta_- \varphi$.
Then
\begin{equation} \label{domb3}
\dom A_{\eta, \tau}=\Big\{
f \in \dom S^* : \Pi_-\Gamma_1 f = \Theta_- \Pi_-\Gamma_0 f,\, (\sigma_0-\Pi_-^* \Pi_-)\Gamma_0 f=0
\Big\}.
\end{equation}
\end{enumerate}
\end{prop}

Note that the case $\eta = \tau = 0$ is not discussed in the previous statement because $A_{\eta,\tau}$ simply 
becomes the free Dirac operator $A_0$ introduced in \eqref{def_free_op}.

\begin{rem}\label{rem37}
\begin{itemize}
\item[$\textup{(i)}$] The operators $\Theta$ and $\Theta_\pm$ in Proposition~\ref{prop-bgamma} are well-defined due to the fact that $\theta$ and $\theta_\pm$ are periodic pseudodifferential operators of order $1$.
For example $\theta \varphi$ makes sense as an element of $H^{-1}(\Sigma; \CC^2)$ for any $\varphi\in L^2(\Sigma; \CC^2)$, and $H^1(\Sigma; \CC^2)\subset \dom \Theta$.
\item[$\textup{(ii)}$] In assertions {\rm (ii)} and {\rm (iii)} of Proposition \ref{prop-bgamma}
we decomposed $\mathcal{G} = L^2(\Sigma; \mathbb{C}^2) = \mathcal{G}_{\Pi_+} \oplus \mathcal{G}_{\Pi_-}$, where
\begin{align*}
  \mathcal{G}_{\Pi_+} &:= \left\{ \varphi = (\varphi_1, \varphi_2) \in L^2(\Sigma; \mathbb{C}^2): \varphi_2=0 \right\} \simeq L^2(\Sigma),\\
  \mathcal{G}_{\Pi_-} &:= \left\{ \varphi = ( \varphi_1, \varphi_2) \in L^2(\Sigma; \mathbb{C}^2): \varphi_1=0 \right\} \simeq L^2(\Sigma).
\end{align*}
\end{itemize}
\end{rem}

\begin{proof} 
With the help of \eqref{eqn:transm1} and \eqref{eqn:transm2} the transmission condition in \eqref{eqn:defdiracdelta} can be rewritten as
\begin{equation}
   \label{trangamma}
-\Lambda\Gamma_0 f = (\eta\sigma_0 + \tau \sigma_3)\Big(\Lambda^{-1}\Gamma_1 f
+\dfrac{1}{2}\,(\mathcal{C}_\zeta+\mathcal{C}_{\Bar\zeta}) \Lambda \Gamma_0 f\Big).
\end{equation}
Now let us distinguish between several cases.

(i) For $|\eta|\ne |\tau|$ the matrix $\eta\sigma_0 + \tau \sigma_3$ is invertible with
\[
(\eta\sigma_0 + \tau \sigma_3)^{-1}= \dfrac{1}{\eta^2-\tau^2}(\eta\sigma_0 - \tau \sigma_3).
\]
Hence, we can rewrite the equality \eqref{trangamma} as
\begin{equation*}
\Gamma_1 f=-\Lambda\bigg[\frac1{\eta^2 - \tau^2}(\eta\sigma_0 - \tau\sigma_3) + \frac12\,(\mathcal{C}_\zeta + \mathcal{C}_{\bar\zeta})\bigg]\Lambda \Gamma_0 f = \Theta \Gamma_0 f,
\end{equation*}
which gives the claimed representation in~\eqref{domb1}

The cases (ii) are and (iii) are almost identical, so we only give a proof for~(ii). By \eqref{trangamma} we have that $f \in \dom A_{\eta, \tau}$ if and only if
\begin{equation*}
  \begin{split}
-\Lambda\Gamma_0 f &= (\eta\sigma_0 + \tau \sigma_3)\Big(\Lambda^{-1}\Gamma_1 f
+\dfrac{1}{2}\,(\mathcal{C}_\zeta+\mathcal{C}_{\Bar\zeta}) \Lambda \Gamma_0 f\Big)\\
&= \begin{pmatrix} 2\eta  & 0 \\ 0 & 0 \end{pmatrix} \Big(\Lambda^{-1}\Gamma_1 f
+\dfrac{1}{2}\,(\mathcal{C}_\zeta+\mathcal{C}_{\Bar\zeta}) \Lambda \Gamma_0 f\Big) \\
&= 2 \eta \Pi_+^* \Pi_+ \Big(\Lambda^{-1}\Gamma_1 f
+\dfrac{1}{2}\,(\mathcal{C}_\zeta+\mathcal{C}_{\Bar\zeta}) \Lambda \Gamma_0 f\Big).
\end{split}
\end{equation*}
Writing this equation in components it follows
that this boundary condition is equivalent to the conditions
\begin{equation*}
(\sigma_0-\Pi_+^*\Pi_+)\Gamma_0 f =0
\end{equation*}
and 
\begin{equation*}
\begin{split}
\Pi_+ \Gamma_1 f&=-\Lambda\Big(\dfrac{1}{2\eta} +\Pi_+\, \frac{1}{2}(\mathcal{C}_\zeta+\mathcal{C}_{\Bar\zeta})\Big) \Lambda \Gamma_0 f \\
&=-\Lambda\Big(\dfrac{1}{2\eta} +\Pi_+\, \frac{1}{2} (\mathcal{C}_\zeta+\mathcal{C}_{\Bar\zeta})\Pi_+^*\Big) \Lambda \Pi_+\Gamma_0 f\\
&
= \Theta_+ \Pi_+ \Gamma_0 f.
\end{split}
\end{equation*}
Hence, we find that \eqref{domb2} is true.
\end{proof}

In view of the general theory of boundary triples, see Subsection~\ref{ssec-bt}, many properties of $A_{\eta, \tau}$ can be deduced 
from the respective properties of the operators $\Theta$
and $\Theta_\pm$ from Proposition~\ref{prop-bgamma}.
We prefer to consider separately the non-critical case $\eta^2-\tau^2\ne 4$ and the critical case $\eta^2-\tau^2=4$, where the latter one is more involved.

\subsection{Non-critical case}\label{sec-noncrit}

Throughout this subsection we assume that
\[
\eta^2-\tau^2\ne 4.
\]
In order to show the self-adjointness of $A_{\eta, \tau}$ we use Theorem~\ref{theorem_boundary_triple_abstract}. For that it is necessary to 
investigate the operators $\Theta$ and $\Theta_\pm$ in Proposition~\ref{prop-bgamma}.

\begin{lem}\label{lem3132}
Let $\eta, \tau \in \mathbb{R}$ with $\eta^2 - \tau^2 \neq 4$. Then the following holds:
\begin{enumerate}
\item[\textup{(i)}]
If $\eta^2-\tau^2\ne 0$, then $\dom \Theta=H^1(\Sigma;\CC^2)$ and $\Theta$ is self-adjoint in $L^2(\Sigma;\CC^2)$.
\item[\textup{(ii)}] If $\eta=\pm\tau$, then $\dom \Theta_\pm = H^1(\Sigma)$ and
$\Theta_\pm$ is self-adjoint in $L^2(\Sigma)$.
\end{enumerate}
\end{lem}

\begin{proof}
  (i) Let us consider the restriction $\Theta_1 := \Theta \upharpoonright H^1(\Sigma; \mathbb{C}^2)$. Since $\theta \in \Psi^1_\Sigma$, the operator $\Theta_1$ is well-defined
  as an operator in $L^2(\Sigma;\CC^2)$.
  We show $\Theta = \Theta_1$ and that $\Theta_1$ is self-adjoint in $L^2(\Sigma;\CC^2)$.
  
  First, it follows from Proposition~\ref{proposition_C_z} that $(\mathcal C_\zeta+\mathcal C_{\bar\zeta})^*=\mathcal C_{\bar\zeta}+\mathcal C_\zeta$  and hence
  $\Theta_1$ is a symmetric operator in $L^2(\Sigma;\CC^2)$. Moreover, since $\Theta_1$ is a symmetric extension of the symmetric operator 
  $\Theta_\infty:= \Theta \upharpoonright C^\infty(\Sigma; \mathbb{C}^2)$,
   Lemma~\ref{sympsi} implies $\Theta_1^* \subset \Theta_\infty^*=\Theta$. Hence, $\Theta = \Theta_1$ and $\Theta_1=\Theta_1^*$ follows if we show 
   $\Theta \subset \Theta_1$, for which it suffices to check the inclusion
   \begin{equation}\label{showthis}
    \dom \Theta\subset \dom \Theta_1=H^1(\Sigma; \mathbb{C}^2).
   \end{equation}
   To see \eqref{showthis} fix some $\varphi \in \dom \Theta$. Then $\theta \varphi \in L^2(\Sigma; \mathbb{C}^2)$. 
   Using Proposition~\ref{proposition_C_z} we find that 
  \begin{equation*}
    \theta \varphi = -\dfrac{1}{2}\Lambda P \Lambda \varphi + \widehat{\Psi} \varphi,\quad\text{where}\,\, P=\begin{pmatrix}
     \dfrac{2}{\eta + \tau} &  C_\Sigma \overline{T} \\
      T C_\Sigma' & \dfrac{2}{\eta - \tau}
    \end{pmatrix}\,\,\,\text{and }\,\,\,\widehat{\Psi} \in \Psi_\Sigma^0.
  \end{equation*}
  Hence, $\Lambda P \Lambda \varphi \in L^2(\Sigma; \mathbb{C}^2)$ 
  and as $\Lambda :H^{\frac{1}{2}}(\Sigma;\CC^2)\rightarrow L^2(\Sigma;\CC^2)$ is bijective, this amounts to $P \Lambda \varphi \in H^{\frac{1}{2}}(\Sigma; \mathbb{C}^2)$. 
  Since $C_\Sigma, C_\Sigma' \in \Psi_\Sigma^0$ by Proposition~\ref{proposition_Cauchy_transform}, 
  these operators give rise to bounded operators in $H^{\frac{1}{2}}(\Sigma; \mathbb{C}^2)$, which implies that
  \begin{equation*}
   \begin{split}
    &\begin{pmatrix}
     \dfrac{2}{\eta - \tau} &  -C_\Sigma \overline{T} \\
      -T C_\Sigma' & \dfrac{2}{\eta + \tau}
    \end{pmatrix}
    \begin{pmatrix}
     \dfrac{2}{\eta + \tau} &  C_\Sigma \overline{T} \\
      T C_\Sigma' & \dfrac{2}{\eta - \tau}
    \end{pmatrix} \Lambda \varphi  \\
    &\qquad \qquad =
    \begin{pmatrix}
     \dfrac{4}{\eta^2 - \tau^2} - C_\Sigma \overline{T} T C_\Sigma' & 0 \\
      0 & \dfrac{4}{\eta^2 - \tau^2} - T C_\Sigma' C_\Sigma \overline{T}
    \end{pmatrix} \Lambda \varphi 
    \in H^{\frac{1}{2}}(\Sigma; \mathbb{C}^2).
   \end{split}
  \end{equation*}
  Now we use that $\overline TT=T\overline T$ is the multiplication operator with the constant function $\one$ and that 
  $C_\Sigma C_\Sigma' - \one, C_\Sigma' C_\Sigma - \one \in \Psi_\Sigma^{-\infty}$ by Proposition~\ref{proposition_Cauchy_transform}. We then obtain from the last line that
  \begin{equation*}
    \frac{4 - \eta^2 + \tau^2}{\eta^2 - \tau^2} \Lambda \varphi + \widetilde{\Psi} \varphi \in H^{\frac{1}{2}}(\Sigma; \mathbb{C}^2)
  \end{equation*}
  with some $\widetilde{\Psi} \in \Psi_\Sigma^{-\infty}$ and hence $\frac{4 - \eta^2 + \tau^2}{\eta^2 - \tau^2} \Lambda \varphi \in H^{\frac{1}{2}}(\Sigma; \mathbb{C}^2)$. 
  Since $\eta^2 - \tau^2 \neq 4$ by assumption, this implies $\Lambda \varphi \in H^{\frac{1}{2}}(\Sigma; \mathbb{C}^2)$ and thus, $\varphi \in H^1(\Sigma; \mathbb{C}^2)$. 
  We have shown \eqref{showthis}. This completes the proof of (i)
  
  (ii)
We consider the case $\eta=\tau$, the other one being similar.
Recall that $\Theta_+$ is the maximal operator in $L^2(\Sigma)$ associated to the periodic pseudodifferential operator
\begin{equation*}
\theta_{+}=-\dfrac{1}{2}\,\Lambda\Big(\dfrac{1}{\eta} + \Pi_+ (\mathcal{C}_\zeta+\mathcal{C}_{\Bar\zeta})\Pi_+^*\Big) \Lambda.
\end{equation*}
Using Proposition~\ref{proposition_C_z} we find for $\varphi  \in \dom \Theta_+ $ that
\begin{equation*}
  \Theta_+ \varphi = -\frac{1}{2 \eta} \Lambda^2 \varphi - \dfrac{1}{2} \Pi_+ \begin{pmatrix}
     0 & \Lambda C_\Sigma \overline{T} \Lambda \\
     \Lambda T C_\Sigma' \Lambda & 0
    \end{pmatrix}  \Pi_+^* \varphi  + \widehat{\Psi} \varphi = -\frac{1}{2 \eta} \Lambda^2 \varphi + \widehat{\Psi} \varphi
\end{equation*}
with some symmetric operator $\widehat{\Psi} \in \Psi_\Sigma^0$. This implies $\dom \Theta_+ = \dom \Lambda^2  = H^1(\Sigma; \mathbb{C})$
and since $ \Lambda^2$ is self-adjoint we conclude that also $\Theta_+$ is self-adjoint in $L^2(\Sigma)$.
\end{proof}

After the preparatory considerations in Lemma~\ref{lem3132} we are now ready to show the self-adjointness 
of $A_{\eta, \tau}$ for non-critical interaction strengths. To formulate the result we recall the definitions of the 
free Dirac operator $A_0$ from~\eqref{def_free_op}, of $\Phi_z$ and $\Phi_z'$ from~\eqref{def_Phi_z} and~\eqref{def_Phi_z_prime}, 
and of $\mathcal{C}_z$ in~\eqref{def_C_z}, respectively.

\begin{thm} \label{theorem_self_adjoint_noncritical}
  Assume that $\eta, \tau \in \mathbb{R}$ with $\eta^2 - \tau^2 \neq 4$ and $(\eta,\tau)\neq(0,0)$. Then the operator $A_{\eta, \tau}$ is self-adjoint in
  $L^2(\RR^2;\CC^2)$
  with $\dom A_{\eta, \tau} \subset H^1(\RR^2\setminus\Sigma; \CC^2)$. Moreover, for all $z \in \res A_{\eta, \tau} \cap \res A_0$ 
  the operator $\sigma_0 + (\eta \sigma_0 + \tau \sigma_3) \mathcal{C}_z$ is bounded and boundedly invertible in $H^{\frac{1}{2}}(\Sigma; \mathbb{C}^2)$ and
  \begin{equation} \label{krein_noncritical}
    (A_{\eta, \tau} - z)^{-1} = (A_0 - z)^{-1} - \Phi_z \big( \sigma_0 + (\eta \sigma_0 + \tau \sigma_3) \mathcal{C}_z \big)^{-1} (\eta \sigma_0 + \tau \sigma_3) \Phi_{\Bar{z}}'
  \end{equation}
  holds. 
\end{thm}

\begin{proof}
First, according to Theorem~\ref{theorem_boundary_triple_abstract} the self-adjointness of $\Theta$ and $\Theta_\pm$ in $L^2(\Sigma; \mathbb{C}^2)$ and 
$L^2(\Sigma)$, respectively,
implies the self-adjointness of $A_{\eta, \tau}$ in $L^2(\RR^2;\CC^2)$. In addition, since $\dom \Theta = H^1(\Sigma;\CC^2)$ and 
$\dom \Theta_\pm = H^1(\Sigma)$, Lemma~\ref{lem23}
yields $\dom A_{\eta, \tau} \subset H^1(\RR^2\setminus\Sigma;\CC^2)$.

It remains to show the Krein type resolvent formula in~\eqref{krein_noncritical}. First, for $|\eta| \neq |\tau|$ we have by 
Theorem~\ref{theorem_boundary_triple_abstract} that $\Theta - M_z$, $z \in \res A_{\eta, \tau} \cap \res A_0$,  is boundedly invertible in $L^2(\Sigma; \mathbb{C}^2)$ and
\begin{equation*}
  (A_{\eta, \tau} - z)^{-1} = (A_0 - z)^{-1} + G_z \big( \Theta - M_z \big)^{-1} G_{\Bar{z}}^*.
\end{equation*}
Taking the special form of $\Theta$ and $M_z = \Lambda \big(\mathcal{C}_z -\frac{1}{2} \big(\mathcal{C}_\zeta+\mathcal{C}_{\Bar\zeta} \big) \big)\Lambda$ 
into account and using $\frac1{\eta^2 - \tau^2}(\eta\sigma_0 - \tau\sigma_3) = (\eta \sigma_0 + \tau \sigma_3)^{-1}$, we find
\begin{equation} \label{Birman_Schwinger_kernel}
  \begin{split}
    \Theta - M_z &
    = -\Lambda\bigg[\frac1{\eta^2 - \tau^2}(\eta\sigma_0 - \tau\sigma_3) + \frac12\,(\mathcal{C}_\zeta + \mathcal{C}_{\bar\zeta})\bigg]\Lambda \\
      &\quad - \Lambda \Big(\mathcal{C}_z -\dfrac{1}{2} \big(\mathcal{C}_\zeta+\mathcal{C}_{\Bar\zeta} \big) \Big)\Lambda \\
   &   = -\Lambda\bigg[\frac1{\eta^2 - \tau^2}(\eta\sigma_0 - \tau\sigma_3) + \mathcal{C}_z \bigg]\Lambda \\
    &= -\Lambda (\eta \sigma_0 + \tau \sigma_3)^{-1} \big( \sigma_0 + (\eta \sigma_0 + \tau \sigma_3) \mathcal{C}_z \big) \Lambda.
  \end{split}
\end{equation}
As $\Theta - M_z$ is a bijective operator in $L^2(\Sigma; \mathbb{C}^2)$ defined on
$\dom \Theta = H^1(\Sigma; \mathbb{C}^2)$ this implies that $\sigma_0 + (\eta \sigma_0 + \tau \sigma_3) \mathcal{C}_z$ is bijective in 
$H^{\frac{1}{2}}(\Sigma; \mathbb{C}^2)$. In particular, the inverse 
$(\sigma_0 + (\eta \sigma_0 + \tau \sigma_3) \mathcal{C}_z)^{-1}$ is well-defined and bounded in $H^{\frac{1}{2}}(\Sigma; \mathbb{C}^2)$.
Using $G_z = \Phi_z \Lambda$ and $G_{\Bar{z}}^* = \Lambda \Phi_{\Bar{z}}'$ we get
\begin{equation} \label{Birman_Schwinger_kernel1}
  \begin{split}
    G_z \big( \Theta - M_z \big)^{-1} G_{\Bar{z}}^* &= - \Phi_z \Lambda \Lambda^{-1} \big( \sigma_0 + (\eta \sigma_0 + \tau \sigma_3) \mathcal{C}_z \big)^{-1} (\eta \sigma_0 + \tau \sigma_3) \Lambda^{-1} \Lambda \Phi_{\Bar{z}}' \\
    &= - \Phi_z \big( \sigma_0 + (\eta \sigma_0 + \tau \sigma_3) \mathcal{C}_z \big)^{-1} (\eta \sigma_0 + \tau \sigma_3)  \Phi_{\Bar{z}}',
  \end{split}
\end{equation}
which leads to~\eqref{krein_noncritical}.

The proof of~\eqref{krein_noncritical} for $|\eta| = |\tau| \neq 0$ is similar as above. First, one notes in the same way as in~\eqref{Birman_Schwinger_kernel} that
\begin{equation} \label{Birman_Schwinger_kernel_pm}
 \Theta_\pm-\Pi_\pm M_z \Pi_\pm^*=-\Lambda \bigg(\dfrac{1}{2\eta} + \Pi_\pm \mathcal{C}_z\Pi_\pm^*  \bigg)\Lambda
 = -\frac{1}{2 \eta} \Pi_\pm \Lambda \big(\sigma_0 + 2 \eta \Pi_\pm^* \Pi_\pm \mathcal{C}_z  \big) \Lambda \Pi_\pm^*,
\end{equation}
which implies with $2 \eta \Pi_\pm^* \Pi_\pm = \eta \sigma_0 + \tau \sigma_3$
\begin{equation*}
  \begin{split}
    \Pi_\pm^* (\Theta_{\pm} - \Pi_\pm M_z \Pi_\pm^*)^{-1} \Pi_\pm &= \Lambda^{-1} \Pi_\pm^* \big(\Pi_\pm (\sigma_0 + 2 \eta \Pi_\pm^* \Pi_+ \mathcal{C}_z) \Pi_\pm^*  \big)^{-1} 2 \eta \Pi_\pm \Lambda^{-1}\\
    &= \Lambda^{-1}  \big(\Pi_\pm^* \Pi_\pm (\sigma_0 + 2 \eta \Pi_\pm^* \Pi_+ \mathcal{C}_z)  \big)^{-1} 2 \eta \Pi_\pm^* \Pi_\pm \Lambda^{-1}\\
    &= \Lambda^{-1} \big( \sigma_0 + (\eta \sigma_0 + \tau \sigma_3) \mathcal{C}_z \big)^{-1} (\eta \sigma_0 + \tau \sigma_3) \Lambda^{-1}.
  \end{split}
\end{equation*}
With this observation and the same ideas as above one shows~\eqref{krein_noncritical} also in the case $|\eta|=|\tau|$. This finishes the proof of this theorem.
\end{proof}

In the following proposition we discuss the basic spectral properties of $A_{\eta, \tau}$:

\begin{thm} \label{proposition_spectral_properties_noncritical}
  Let $\eta, \tau \in \mathbb{R}$ be such that $\eta^2 - \tau^2 \neq 4$. Then the following holds:
  \begin{itemize}
    \item[\textup{(i)}] We have $\spec_{\textup{ess}} A_{\eta,\tau}=\big(-\infty, -|m|\big] \,\cup\,\big[|m|,+\infty\big)$.
    In particular, for $m=0$ we have $\spec A_{\eta, \tau} = \spec_\ess A_{\eta,\tau}= \RR$.
    \item[\textup{(ii)}] Assume $m\neq 0$. Then $z \in (-|m|, |m|)$ is a discrete eigenvalue of $A_{\eta, \tau}$ if and only if there exists $\varphi \in H^{\frac{1}{2}}(\Sigma; \mathbb{C}^2)$ such that $\big( \sigma_0 + (\eta \sigma_0 + \tau \sigma_3) \mathcal{C}_z \big) \varphi = 0$.
    \item[\textup{(iii)}] If $m\ne 0$, then $A_{\eta, \tau}$ has at most finitely many eigenvalues in $\big(-|m|,|m|\big)$.
  \end{itemize}
\end{thm}
\begin{proof}
  By Proposition~\ref{prop-ess}, the set $\big(-\infty, -|m|\big] \,\cup\,\big[|m|,+\infty\big)$ is contained in the essential spectrum of $A_{\eta, \tau}$.
	Moreover, we have shown in Theorem~\ref{theorem_self_adjoint_noncritical} that the inclusion $\dom A_{\eta,\tau} \subset H^1(\mathbb{R}^2 \setminus \Sigma; \mathbb{C}^2)$ holds, which implies by Proposition~\ref{prop-discr}
	that the spectrum of $ A_{\eta,\tau}$ in $\big(-|m|,|m|\big)$ is discrete and finite. This proves the items (i) and (iii).
	%
	%
  
  It remain to prove~(ii). Assume first that $|\eta| \neq |\tau|$. By Theorem~\ref{theorem_boundary_triple_abstract} a number 
  $z \in \res A_0$ is an eigenvalue of $A_{\eta, \tau}$ if and only if zero is an eigenvalue of $\Theta - M_z$. 
  Using~\eqref{Birman_Schwinger_kernel} this means that $z \in \res A_0$ is an eigenvalue of $A_{\eta, \tau}$ if and only 
  if there exists $\psi \in \dom \Theta = H^1(\Sigma; \mathbb{C}^2)$ such that
  \begin{equation*}
    -\Lambda (\eta \sigma_0 + \tau \sigma_3)^{-1} \big( \sigma_0 + (\eta \sigma_0 + \tau \sigma_3) \mathcal{C}_z \big) \Lambda \psi = 0,
  \end{equation*}
  i.e. if and only if $\varphi := \Lambda \psi \in H^{\frac{1}{2}}(\Sigma; \mathbb{C}^2)$ satisfies 
  $$\big( \sigma_0 + (\eta \sigma_0 + \tau \sigma_3) \mathcal{C}_z \big) \varphi = 0.$$ 
  The proof of~(ii) for $|\eta|=|\tau|$ is similar, one just has to use~\eqref{Birman_Schwinger_kernel_pm} instead of~\eqref{Birman_Schwinger_kernel}.
\end{proof}  

Finally, we provide some 
symmetry relations for the point spectrum of $A_{\eta, \tau}$, which can be seen as consequences of 
commutator relations of $A_{\eta, \tau}$.
The following results are the two-dimensional analogues of \cite[Proposition~4.2]{BEHL19}.

\begin{prop} \label{proposition_symmetry_relations_noncritical}
  Let $\eta, \tau \in \mathbb{R}$ and assume that $\eta^2 - \tau^2 \neq 4$. Then the following holds:
  \begin{itemize}
    \item[\textup{(i)}] If $|\eta| \neq |\tau|$, then $z \in \spec_\textup{p} A_{-\frac{4 \eta}{\eta^2 - \tau^2}, -\frac{4 \tau}{\eta^2 - \tau^2}}$ if and only if $z \in \spec_\textup{p} A_{\eta, \tau}$.
    \item[\textup{(ii)}] $z \in \spec_\textup{p} A_{\eta, \tau}$ if and only if $(-z) \in \spec_\textup{p} A_{-\eta, \tau}$.
  \end{itemize}
\end{prop}
\begin{proof}
  (i) Consider the unitary and self-adjoint operator
  \begin{equation*}
    U: L^2(\Omega_+; \mathbb{C}^2) \oplus L^2(\Omega_-; \mathbb{C}^2) \rightarrow L^2(\Omega_+; \mathbb{C}^2) \oplus L^2(\Omega_-; \mathbb{C}^2), \quad U (f_+ \oplus f_-) = f_+ \oplus(-f_-).
  \end{equation*}
  We claim that 
  \begin{equation} \label{unitary_equivalence}
    A_{\eta, \tau} = U A_{-\frac{4 \eta}{\eta^2 - \tau^2}, -\frac{4 \tau}{\eta^2 - \tau^2}} U.
  \end{equation}
  For this purpose we note first that $f = f_+ \oplus f_- \in H^1(\Omega_+; \mathbb{C}^2) \oplus H^1(\Omega_-; \mathbb{C}^2)$ belongs to $\dom A_{\eta, \tau}$, if and only if 
  \begin{equation} \label{transmission_condition_temp}
    -\rmi (\sigma\cdot \nu)\big( \cT_+^D f_+ - \cT_-^D f_-\big) = \frac12(\eta \sigma_0 + \tau \sigma_3)\big(\cT_+^D f_+ + \cT_-^D f_-\big),
  \end{equation}
  which is equivalent to
  \begin{equation*}
    -\rmi (\sigma\cdot \nu)\big( \cT_+^D (U f)_+ + \cT_-^D (U f)_-\big) = \frac12(\eta \sigma_0 + \tau \sigma_3)\big(\cT_+^D (U f)_+ - \cT_-^D (U f)_-\big).
  \end{equation*}
  By multiplying the last equation with $(\eta \sigma_0 + \tau \sigma_3)^{-1} = \frac{1}{\eta^2 - \tau^2} (\eta \sigma_0 - \tau \sigma_3)$ and using \eqref{anti_commutation} we find that $f \in \dom A_{\eta, \tau}$ if and only if
  \begin{equation*}
    -\rmi (\sigma\cdot \nu) \frac{1}{\eta^2 - \tau^2} (\eta \sigma_0 + \tau \sigma_3) \big( \cT_+^D (U f)_+ + \cT_-^D (U f)_-\big) = \frac12 \big(\cT_+^D (U f)_+ - \cT_-^D (U f)_-\big),
  \end{equation*}
  which is equivalent to 
  \begin{equation*}
    - \frac{4}{\eta^2 - \tau^2} (\eta \sigma_0 + \tau \sigma_3) \frac{1}{2} \big( \cT_+^D (U f)_+ + \cT_-^D (U f)_-\big) = - \rmi (\sigma\cdot \nu) \big(\cT_+^D (U f)_+ - \cT_-^D (U f)_-\big)
  \end{equation*}
  i.e. $U f \in \dom A_{-4 \eta/(\eta^2 - \tau^2), -4 \tau/(\eta^2 - \tau^2)}$. Hence, we have shown the equality $\dom A_{\eta, \tau} = \dom A_{-4 \eta/(\eta^2 - \tau^2), -4 \tau/(\eta^2 - \tau^2)} U$. Moreover, a straightforward calculation  shows $U A_{\eta, \tau} f = A_{-4 \eta/(\eta^2 - \tau^2), -4 \tau/(\eta^2 - \tau^2)} U f$ for any $f \in \dom A_{\eta, \tau}$. This gives~\eqref{unitary_equivalence}, which yields (i).
  
  (ii) Define the antilinear charge conjugation operator
  \begin{equation*}
    C f = \sigma_1 \overline{f}, \qquad f \in L^2(\mathbb{R}^2; \mathbb{C}^2).
  \end{equation*}
  Then we see immediately $C^2 f = f$ for all $f \in L^2(\mathbb{R}^2; \mathbb{C}^2)$. We claim that 
  \begin{equation} \label{charge_conjugation}
    C A_{\eta, \tau} = -A_{- \eta, \tau} C,
  \end{equation}
  which yields then the claim of statement~(ii). To prove~\eqref{charge_conjugation}, we note first by taking the complex 
  conjugate of equation~\eqref{transmission_condition_temp} that $f \in \dom A_{\eta, \tau}$ if and only if
  \begin{equation}\label{haarpracht}
    \rmi (\overline{\sigma}\cdot \nu)\big( \cT_+^D \overline{f_+} - \cT_-^D \overline{f_-}\big) 
    = \frac12(\eta \sigma_0 + \tau \sigma_3)\big(\cT_+^D \overline{f_+} + \cT_-^D \overline{f_-}\big),
  \end{equation}
  where $\overline\sigma =(\overline{\sigma_1},\overline{\sigma_2})$ and $\overline{\sigma_j}$ is the matrix  with the complex conjugate entries of $\sigma_j$.
  By multiplying this equation with $\sigma_1$ and using~\eqref{anti_commutation}, $\overline{\sigma_1} = \sigma_1$, and $\overline{\sigma_2} = -\sigma_2$ 
  we find that \eqref{haarpracht} is equivalent to
  \begin{equation*}
    \rmi (\sigma \cdot \nu)\big( \cT_+^D (\sigma_1 \overline{f_+}) - \cT_-^D (\sigma_1 \overline{f_-})\big) = \frac12(\eta \sigma_0 - \tau \sigma_3)\big(\cT_+^D (\sigma_1 \overline{f_+}) + \cT_-^D (\sigma_1\overline{f_-})\big),
  \end{equation*}
  i.e. $Cf \in \dom A_{-\eta, \tau}$. Moreover, using again~\eqref{anti_commutation} and $\overline{\sigma_2} = -\sigma_2$
  we get
  \begin{equation*}
    \begin{split}
      (-\rmi \sigma \cdot \nabla + m \sigma_3) C f &= (-\rmi \sigma \cdot \nabla + m \sigma_3) \sigma_1 \overline{f}\\
      &= \sigma_1 (-\rmi \overline{\sigma} \cdot \nabla - m \sigma_3) \overline{f} \\
      &= -\sigma_1 \overline{(-\rmi \sigma \cdot \nabla + m \sigma_3) f} \\
      &= - C \big( -\rmi \sigma \cdot \nabla + m \sigma_3) f \big),
    \end{split}
  \end{equation*}
  which implies \eqref{charge_conjugation}.
\end{proof}

\subsection{Critical case}\label{sec-crit}

In this subsection we study the self-adjointness and the spectral properties of $A_{\eta, \tau}$ 
for the critical interaction strengths, i.e. when $\eta^2 - \tau^2  = 4$. 
To show the self-adjointness of $A_{\eta, \tau}$ we prove that the corresponding operator 
$\Theta$ in Proposition~\ref{prop-bgamma} is self-adjoint in $L^2(\Sigma; \mathbb{C}^2)$.

\begin{lem}\label{self-theta}
  Let $\eta, \tau \in \mathbb{R}$ be such that $\eta^2 - \tau^2 = 4$. Then the operator 
  $\Theta$ is self-adjoint in $L^2(\Sigma; \mathbb{C}^2)$ and the restriction of $\Theta$ onto $H^1(\Sigma;\CC^2)$ is essentially self-adjoint 
  in $L^2(\Sigma; \mathbb{C}^2)$.
\end{lem}

\begin{remark}
  According to Lemma~\ref{self-theta} the operator
  $\Theta$ is essentially self-adjoint on $H^1(\Sigma;\CC^2)$. It will turn out later in 
  the proof of Proposition~\ref{prop54} that $\spec_\ess \Theta$ is non-empty. Hence, one has $\dom \Theta \not\subset H^s(\Sigma; \mathbb{C}^2)$ for all $s>0$.
\end{remark}

\begin{proof}[Proof of Lemma~\ref{self-theta}]
  As in the proof of Lemma~\ref{lem3132} we consider the restriction $\Theta_1 := \Theta \upharpoonright H^1(\Sigma; \mathbb{C}^2)$. 
  It follows in the same way as in the proof of Lemma~\ref{lem3132} that $\Theta_1$ is a symmetric operator in $L^2(\Sigma; \mathbb{C}^2)$ and together with
  Lemma~\ref{sympsi} we see $\overline\Theta_1\subset\Theta_1^* \subset \Theta$. 
  To see $\Theta \subset \overline{\Theta_1}$, which then implies the claims, we will show (the slightly stronger fact) that
  \begin{equation}\label{doms}
   \dom\Theta=\dom\overline{\Theta_1}.
  \end{equation}
For this we consider the associated periodic pseudodifferential operator $\theta$ defined in~\eqref{def_Theta} and recall that with the aid of Proposition~\ref{proposition_C_z} we have
  \begin{equation}\label{upsi}
    \theta=-\dfrac{1}{2}\,\upsilon + \Psi, \quad\text{where}\,\,\,
    \upsilon= \begin{pmatrix} \dfrac{2}{\eta+\tau} \Lambda^2 & \Lambda C_\Sigma \overline{T} \Lambda\\
         \Lambda T C_\Sigma' \Lambda & \dfrac{2}{\eta-\tau}\, \Lambda^2
    \end{pmatrix},
  \end{equation}
  with some operator $\Psi \in \Psi^0_\Sigma$, which is symmetric and hence self-adjoint in $L^2(\Sigma; \mathbb{C}^2)$. 
  In the following we denote by $\Upsilon$ the maximal realization of $\upsilon$ in $L^2(\Sigma; \mathbb{C}^2)$, that is
  \begin{equation*}
    \Upsilon \varphi = \upsilon \varphi, \quad \dom \Upsilon = \big\{ \varphi \in L^2(\Sigma; \mathbb{C}^2): \upsilon \varphi \in L^2(\Sigma; \mathbb{C}^2) \big\} = \dom \Theta,
  \end{equation*}
  and $\Upsilon_1 = \Upsilon \upharpoonright H^1(\Sigma; \mathbb{C}^2)$. Note that $\dom \overline{\Upsilon_1} = \dom \overline{\Theta_1}$.  
  In the same way as in Subsection~\ref{ssec-schur} we use the Schur complement to decompose $\upsilon$ (on a formal level in the sense of periodic pseudodifferential operators
  without specification of the operator domains) as
  \begin{equation} \label{factorization_xi}
    \upsilon=\begin{pmatrix}
\one & 0\\
\dfrac{\eta+\tau}{2} \, \Lambda T C_\Sigma' \Lambda^{-1}& \one
\end{pmatrix}
\begin{pmatrix}
\dfrac{2}{\eta+\tau}\,\Lambda^2 & 0 \\
0 & \mathcal{S}(\upsilon)
\end{pmatrix}
\begin{pmatrix}
\one & \dfrac{\eta+\tau}{2}\Lambda^{-1} C_\Sigma \overline{T} \Lambda\\
0 & \one
\end{pmatrix},
\end{equation}
where the Schur complement has the form
\begin{equation*}
  \mathcal{S}(\upsilon)=\dfrac{2}{\eta-\tau}\, \Lambda^2-\dfrac{\eta+\tau}{2} \Lambda T C_\Sigma' \Lambda 
(\Lambda^2)^{-1} \Lambda  C_\Sigma \overline{T} \Lambda=\dfrac{2}{\eta-\tau}\, \Lambda^2-\dfrac{\eta+\tau}{2} \Lambda T C_\Sigma'  C_\Sigma \overline{T} \Lambda.
\end{equation*}
Using that $C_\Sigma' C_\Sigma=\one+R$ with $R\in\Psi^{-\infty}_\Sigma$, see Proposition~\ref{proposition_Cauchy_transform},
we can rewrite this expression as
\begin{equation*}
    \mathcal{S}(\upsilon)= \frac{2}{\eta-\tau} \Lambda^2-\frac{\eta+\tau}{2} \Lambda T \overline{T} \Lambda - \frac{\eta+\tau}{2} \Lambda T R \overline{T} \Lambda 
      =- \frac{\eta+\tau}{2} \Lambda T R \overline{T} \Lambda \in \Psi^{-\infty}_\Sigma,
\end{equation*}
where we used in the last step that $T \overline{T}$ is the multiplication operator with the constant function $\one$ and $\eta^2 - \tau^2 = 4$. 
From this, \eqref{factorization_xi}, and $\dom\Lambda^2= H^1(\Sigma)$ we obtain now
\begin{equation*}
  \dom \Theta = \dom \Upsilon =\Big\{
	( \varphi_1, \varphi_2) \in L^2(\Sigma;\CC^2):
	\varphi_1+ \dfrac{\eta+\tau}{2}\Lambda^{-1} C_\Sigma \overline{T} \Lambda \varphi_2 \in H^1(\Sigma)	\Big\}.
\end{equation*}
Let us now consider the operator realizations $\Theta_1,\Upsilon_1$ of $\theta,\upsilon$ and their closures $\overline{\Theta_1},\overline{\Upsilon_1}$ in $L^2(\Sigma;\CC^2)$.
We leave it to the reader to check that the assumptions in Proposition~\ref{prop-block} are satisfied when each entry of the pseudodifferential operators in the matrix representation
of $\upsilon$ in \eqref{upsi} is defined on $H^1(\Sigma)$; in particular, note that the upper left corner is a boundedly invertible self-adjoint operator 
in $L^2(\Sigma)$ with domain $H^1(\Sigma)$. Then it follows 
from Proposition~\ref{prop-block} that $\dom \overline{\mathcal{S}(\Xi_1)} = L^2(\Sigma)$ and
\begin{equation*}
  \begin{split}
    \dom \overline{\Theta_1} &= \dom \overline{\Xi_1} \\
    &=\big\{
	( \varphi_1, \varphi_2) \in L^2(\Sigma;\CC^2):
	\varphi_1 + \dfrac{\eta+\tau}{2}\Lambda^{-1} C_\Sigma \overline{T} \Lambda \varphi_2 \in H^1(\Sigma)	\big\}
	= \dom \Theta
  \end{split}
\end{equation*}
hold. 
Hence, we have shown \eqref{doms}, which finishes the proof of this proposition.
\end{proof}

With Lemma~\ref{self-theta} we are now ready to show the self-adjointness of $A_{\eta, \tau}$ for critical interaction strengths. 
To formulate the result we recall the definitions of the free Dirac operator $A_0$ from~\eqref{def_free_op}, of $\Phi_z$ and $\Phi_z'$ 
from~\eqref{def_Phi_z} and~\eqref{def_Phi_z_prime}, and of $\mathcal{C}_z$ in~\eqref{def_C_z}, respectively.

\begin{thm} \label{theorem_self_adjoint_critical}
Let $\eta, \tau \in \mathbb{R}$ with $\eta^2 - \tau^2 = 4$. Then the operator $A_{\eta, \tau}$ is self-adjoint
and its restriction to $\dom A_{\eta, \tau} \cap H^1(\RR^2\setminus\Sigma; \CC^2)$ is essentially self-adjoint in
  $L^2(\RR^2;\CC^2)$. Moreover, for all $z \in \res A_{\eta, \tau} \cap \res A_0$ 
  the operator $\sigma_0 + (\eta \sigma_0 + \tau \sigma_3) \mathcal{C}_z$ 
  admits a bounded inverse from $H^{\frac{1}{2}}(\Sigma; \mathbb{C}^2)$ to $H^{-\frac{1}{2}}(\Sigma; \mathbb{C}^2)$, and
  \begin{equation} \label{krein_critical}
    (A_{\eta, \tau} - z)^{-1} = (A_0 - z)^{-1} - \Phi_z \big( \sigma_0 + (\eta \sigma_0 + \tau \sigma_3) \mathcal{C}_z \big)^{-1} (\eta \sigma_0 + \tau \sigma_3) \Phi_{\Bar{z}}'.
  \end{equation}
\end{thm}

\begin{proof}
First, according to Theorem~\ref{theorem_boundary_triple_abstract} the self-adjointness of $\Theta$ in $L^2(\Sigma; \mathbb{C}^2)$
implies the self-adjointness of $A_{\eta, \tau}$ in $L^2(\RR^2;\CC^2)$, and the essential self-adjointness of $\Theta_1=\Theta\upharpoonright H^1(\Sigma;\CC^2)$ 
in $L^2(\Sigma; \mathbb{C}^2)$ implies the essential self-adjointness of the restriction of $A_{\eta, \tau}$ to $\dom A_{\eta, \tau} \cap H^1(\RR^2\setminus\Sigma; \CC^2)$ 
in $L^2(\RR^2;\CC^2)$. For the latter observation we have also used that
by Lemma~\ref{lem23}
\begin{equation*}
  S^* \upharpoonright \ker (\Gamma_1 - \Theta_1 \Gamma_0) = A_{\eta, \tau} \upharpoonright \big( \dom A_{\eta, \tau} \cap H^1(\RR^2\setminus\Sigma; \CC^2) \big).
\end{equation*}

It remains to verify the Krein type resolvent formula in~\eqref{krein_critical}. By Theorem~\ref{theorem_boundary_triple_abstract}  we have 
that $\Theta - M_z$ is boundedly invertible in $L^2(\Sigma; \mathbb{C}^2)$ and
\begin{equation*}
  (A_{\eta, \tau} - z)^{-1} = (A_0 - z)^{-1} + G_z \big( \Theta - M_z \big)^{-1} G_{\Bar{z}}^*.
\end{equation*}
Taking the special form of $\Theta$ and $M_z = \Lambda \big(\mathcal{C}_z -\frac{1}{2} \big(\mathcal{C}_\zeta+\mathcal{C}_{\Bar\zeta} \big) \big)\Lambda$ 
into account we find with a similar calculation as in~\eqref{Birman_Schwinger_kernel}-\eqref{Birman_Schwinger_kernel1} that
\begin{equation*} 
  \begin{split}
    (\Theta - M_z)^{-1}
    &= -\Lambda^{-1}  \big( \sigma_0 + (\eta \sigma_0 + \tau \sigma_3) \mathcal{C}_z \big)^{-1} (\eta \sigma_0 + \tau \sigma_3) \Lambda^{-1}.
  \end{split}
\end{equation*}
As $(\Theta - M_z)^{-1}$ is bounded in $L^2(\Sigma; \mathbb{C}^2)$ we deduce that $( \sigma_0 + (\eta \sigma_0 + \tau \sigma_3) \mathcal{C}_z )^{-1}$ is bounded from $H^{\frac{1}{2}}(\Sigma; \mathbb{C}^2)$ to $H^{-\frac{1}{2}}(\Sigma; \mathbb{C}^2)$.
Using $G_z = \Phi_z \Lambda$ and $G_{\Bar{z}}^* = \Lambda \Phi_{\Bar{z}}'$ we get
\begin{equation*} 
  \begin{split}
    G_z \big( \Theta - M_z \big)^{-1} G_{\Bar{z}}^* &= - \Phi_z \Lambda \Lambda^{-1} \big( \sigma_0 + (\eta \sigma_0 + \tau \sigma_3) \mathcal{C}_z \big)^{-1} (\eta \sigma_0 + \tau \sigma_3) \Lambda^{-1} \Lambda \Phi_{\Bar{z}}' \\
    &= - \Phi_z \big( \sigma_0 + (\eta \sigma_0 + \tau \sigma_3) \mathcal{C}_z \big)^{-1} (\eta \sigma_0 + \tau \sigma_3)  \Phi_{\Bar{z}}',
  \end{split}
\end{equation*}
and thus~\eqref{krein_critical}.
\end{proof}

In the next proposition we analyze the essential spectrum of the self-adjoint operator $\Theta$.
Note that our assumption $\eta^2-\tau^2=4$ implies $|\tau|<|\eta|$, and hence 
$-\frac{\tau}{\eta}\,m\in (-|m|,|m|)$.

\begin{prop}\label{prop54}
Let $\eta, \tau \in \mathbb{R}$ be such that $\eta^2 - \tau^2 = 4$ and let $m \neq 0$. Then for $z\in (-|m|,|m|)$ one has
$0\in \spec_\ess (M_z-\Theta)$  if and only if $z=-\frac{\tau}{\eta}\,m$.
\end{prop}

\begin{proof}
Throughout the proof we assume that $z\in (-|m|,|m|)$. In particular, $M_z$ is a bounded self-adjoint operator in $L^2(\Sigma;\CC^2)$. 
Recall that 
\[
M_z-\Theta = \Lambda\frac1{\eta^2 - \tau^2}(\eta\sigma_0 - \tau\sigma_3)\Lambda
+\Lambda \mathcal{C}_z \Lambda,
\]
and using Proposition~\ref{proposition_C_z} we decompose this self-adjoint operator in $M_z-\Theta=\Xi_1+\Xi_2$, where
\[
\Xi_1:=
\begin{pmatrix}
\dfrac{1}{\eta+\tau}\,\Lambda^2+\dfrac{\ell}{4\pi} (z+m)\one
& \dfrac{1}{2}\Lambda C_\Sigma \overline{T} \Lambda\\
\dfrac{1}{2}  \Lambda T C_\Sigma' \Lambda & 
\dfrac{1}{\eta-\tau}\,\Lambda^2+\dfrac{\ell}{4\pi} (z-m)\one
\end{pmatrix}
\]
and $\Xi_2\in \Psi^{-1}_\Sigma$ is a compact self-adjoint operator in $L^2(\Sigma; \mathbb{C}^2)$. We note that $\Xi_1$ defined on $\dom(M_z-\Theta)=\dom\Theta$
is a self-adjoint operator in $L^2(\Sigma;\CC^2)$. 
It follows that
$\spec_\ess(M_z-\Theta)=\spec_\ess \Xi_1$ and, in particular, 
\begin{equation*}
0\in\spec_\ess (M_z-\Theta)
\text{ if and only if } 0 \in \spec_\ess \Xi_1.
\end{equation*}

In the following we will show that $0\in\spec_\ess \Xi_1$ if and only if $z=-\frac{\tau}{\eta}\,m$. For this, the Schur complement of $\Xi_1$ and Proposition~\ref{prop-block}
will be used.
To proceed, we shall use the operator $\Lambda \in \Psi^{\frac{1}{2}}_\Sigma$ from \eqref{def_Lambda} (see also \eqref{def_L}). Recall also that 
$\Lambda^2 \geq c_0^2$ for $c_0>0$. Now we choose $c_0$ such that $c_0^2 > \frac{|m| \ell}{2 \pi} |\eta+\tau|$. Then the upper left corner of $\Xi_1$, 
\begin{equation*}
  \dfrac{1}{\eta+\tau}\,\Lambda^2+\dfrac{\ell}{4\pi} (z+m)\one,
\end{equation*}
is boundedly invertible in $L^2(\Sigma)$. We leave it to the reader to check that the other assumptions in Proposition~\ref{prop-block}
are also satisfied for the block operator matrix $\Xi_1$. Therefore,
we have $0\in \spec_\ess \Xi_1$ if and only if $0\in\spec_\ess \mathcal{S}$,
where $\mathcal{S} := \mathcal{S}(\Xi_1)$ is the Schur complement
\begin{align*}
\mathcal{S}=&\dfrac{1}{\eta-\tau}\,\Lambda^2+\dfrac{\ell(z-m)}{4\pi}\one
 -\dfrac{\eta+\tau}{4}\,\Lambda T C_\Sigma' \Lambda
\left( \Lambda^2+\dfrac{\ell (z+m)(\eta+\tau)}{4\pi}\one\right)^{-1}
\Lambda C_\Sigma \overline{T} \Lambda.
\end{align*}
To simplify the last summand in the above expression of $\mathcal{S}$ we use the identity
\begin{equation}
  \label{reseq1}
( \Lambda^2+a \one)^{-1} =\Lambda^{-2}
-a \Lambda^{-1}(\Lambda^2+a\one)^{-1}\Lambda^{-1}=\Lambda^{-2}
-a \Lambda^{-2}(\Lambda^2+a\one)^{-1}
\end{equation}
and rewrite $\mathcal{S}=\mathcal{S}_1+\mathcal{S}_2$ with
\begin{equation*}
\mathcal{S}_1=\dfrac{1}{\eta-\tau}\,\Lambda^2+\dfrac{\ell (z-m)}{4\pi}\one-\dfrac{\eta+\tau}{4}\,\Lambda T C_\Sigma' C_\Sigma \overline{T} \Lambda
\end{equation*}
and
\begin{equation*}
\mathcal{S}_2=\dfrac{(\eta+\tau)^2}{4}\cdot\dfrac{\ell (z+m)}{4\pi}
\Lambda  T C_\Sigma'
\left( \Lambda^2+\dfrac{\ell (z+m)(\eta+\tau)}{4\pi}\one\right)^{-1}
C_\Sigma \overline{T} \Lambda.
\end{equation*}
By Proposition~\ref{proposition_Cauchy_transform} one has
$C_\Sigma' C_\Sigma=\one+K_1$ with $K_1\in\Psi^{-\infty}_\Sigma$, so 
\[
\dfrac{\eta+\tau}{4}\,\Lambda T C_\Sigma' C_\Sigma \overline{T} \Lambda=
\dfrac{\eta+\tau}{4}\,\Lambda^2 + K_2
\]
with $K_2 \in \Psi^{-\infty}_\Sigma$. This gives because of $\eta^2 - \tau^2 = 4$
\begin{align*}
\mathcal{S}_1&=
\dfrac{1}{\eta-\tau}\,\Lambda^2+\dfrac{\ell (z-m)}{4\pi}\one
-\dfrac{\eta+\tau}{4}\,\Lambda^2-K_2=\dfrac{\ell (z-m)}{4\pi}\one-K_2.
\end{align*}
In order to deal with $\mathcal{S}_2$ we use again the identity \eqref{reseq1}, which gives
\begin{align*}
\dfrac{4}{(\eta+\tau)^2}\cdot\dfrac{4\pi}{\ell (z+m)}\mathcal{S}_2&=
\Lambda T C_\Sigma' \left( \Lambda^2+\dfrac{\ell (z+m)(\eta+\tau)}{4\pi}\one\right)^{-1}
C_\Sigma \overline{T} \Lambda=K_3+K_4,\\
\end{align*}
where
\begin{align*}
K_3&=\Lambda T C_\Sigma' \Lambda^{-2} C_\Sigma \overline{T} \Lambda,\\
K_4&=-\dfrac{\ell (z+m)(\eta+\tau)}{4\pi}
\Lambda T C_\Sigma' \Lambda^{-2}\left( \Lambda^2+\dfrac{\ell (z+m)(\eta+\tau)}{4\pi}\one\right)^{-1}
C_\Sigma \overline{T} \Lambda.
\end{align*}
Using Proposition~\ref{proposition_properties_PPDO} one finds that $K_4 \in \Psi^{-1}_\Sigma$ 
and hence this operator is compact in $L^2(\Sigma; \mathbb{C}^2)$. In order to simplify $K_3$ we note first that 
\[
K_5:=T C_\Sigma' \Lambda^{-2}-\Lambda^{-2} T C_\Sigma' \in \Psi^{-2}_\Sigma
\]
by
Proposition~\ref{proposition_properties_PPDO}~(ii). Hence,
\begin{equation*}
  \begin{split}
    K_3=\Lambda \Lambda^{-2} T C_\Sigma' C_\Sigma \overline{T} \Lambda + \Lambda K_5 C_\Sigma\overline{T} \Lambda
    =: \Lambda \Lambda^{-2} T C_\Sigma' C_\Sigma \overline{T} \Lambda + K_6 
  \end{split}
\end{equation*}
with $K_6 \in \Psi_\Sigma^{-1}$.
Using again $C_\Sigma' C_\Sigma-\one\in\Psi^{-\infty}$, see Proposition~\ref{proposition_Cauchy_transform}, we arrive 
at $K_3=\one + K_7$ with $K_7 \in \Psi^{-1}_\Sigma$. With this we find
\[
\mathcal{S}_2=\dfrac{(\eta+\tau)^2}{4}\cdot\dfrac{\ell (z+m)}{4\pi} (K_3+K_4)=\dfrac{(\eta+\tau)^2}{4}\cdot\dfrac{\ell (z+m)}{4\pi}\one + K_8
\]
with $K_8 \in \Psi^{-1}_\Sigma$.
Using this in the expression of the Schur complement $\mathcal{S}$ we conclude, with some $K_9\in \Psi^{-1}_\Sigma$, that
\begin{align*}
\mathcal{S}&=\mathcal{S}_1+\mathcal{S}_2\\
&=
\Big(
\dfrac{\ell (z-m)}{4\pi}
+\dfrac{(\eta+\tau)^2}{4}\cdot\dfrac{\ell (z+m)}{4\pi}
\Big)\one + K_9\\
&=\dfrac{\ell}{4\pi}\Big[
\Big( \dfrac{(\eta+\tau)^2}{4}+1\Big)\, z + \Big( \dfrac{(\eta+\tau)^2}{4}-1\Big)\,m
\,\Big]\one + K_9.
\end{align*}
As $K_9$ is compact and symmetric, it does not influence the essential spectrum, and we have
\[
0\in \spec_\ess \mathcal{S} \text{ if and only if } z=-\dfrac{(\eta+\tau)^2-4}{(\eta+\tau)^2+4}\,m.
\]
With $\eta^2 - \tau^2 = 4$ we can simplify the last expression to
\[
\dfrac{(\eta+\tau)^2-4}{(\eta+\tau)^2+4}=\dfrac{\eta^2+\tau^2+2\eta\tau-\eta^2 +\tau^2}{\eta^2+\tau^2+2\eta\tau+\eta^2 -\tau^2}
=\dfrac{2\tau^2+2\eta\tau}{2\eta^2+2\eta\tau}=\dfrac{2\tau(\eta+\tau)}{2\eta(\eta+\tau)}=\dfrac{\tau}{\eta}.
\]
Hence, $0\in \spec_\ess \mathcal{S}$  if and only if $z=- \frac{\tau}{\eta}\,m$.
This finishes the proof.
\end{proof}


We are now ready to describe the spectral properties of $A_{\eta, \tau}$ for critical interaction strengths. Compared to 
Proposition~\ref{proposition_spectral_properties_noncritical}, the following theorem shows that the spectral properties of $A_{\eta, \tau}$ differ 
significantly from the non-critical case.

\begin{thm} \label{theorem_spectrum_critical}
  Let $\eta, \tau \in \mathbb{R}$ with $\eta^2 - \tau^2 = 4$. Then the following holds:
  \begin{itemize}
    \item[\textup{(i)}] The essential spectrum of $A_{\eta,\tau}$ is
    \[
		\spec_\ess  A_{\eta, \tau}=\big(-\infty,-|m|\big]\,\cup\,
\big\{ -\tfrac{\tau}{\eta}\,m\big\}\,\cup\,\big[|m|,+\infty\big).
\] 
In particular, for $m=0$ we have $\spec A_{\eta, \tau} =\spec_\ess A_{\eta, \tau}= \RR$.
    \item[\textup{(ii)}] Assume $m\neq 0$. Then $z \notin \spec_\textup{ess} A_{\eta, \tau}$ is a discrete eigenvalue of $A_{\eta, \tau}$ if and only if there exists 
    $\varphi \in H^{-\frac{1}{2}}(\Sigma; \mathbb{C}^2)$ such that $\big( \sigma_0 + (\eta \sigma_0 + \tau \sigma_3) \mathcal{C}_z \big) \varphi = 0$.
    \item[\textup{(iii)}] For all $s > 0$ we have $\dom A_{\eta, \tau} \not\subset H^s(\mathbb{R}^2 \setminus \Sigma; \mathbb{C}^2)$.
  \end{itemize}
\end{thm}

\begin{rem}
  Item~(ii) in the above theorem is slightly weaker as Proposition~\ref{proposition_spectral_properties_noncritical}~(ii), 
  since one has to search for eigenfunctions $\varphi$ of the Birman-Schwinger operator $\sigma_0 + (\eta \sigma_0 + \tau \sigma_3) \mathcal{C}_z$ 
  in the larger space $H^{-\frac{1}{2}}(\Sigma; \mathbb{C}^2)$. However, as there is no Sobolev regularity in $\dom A_{\eta, \tau}$ 
  the smoothness of the eigenfunctions of $\sigma_0 + (\eta \sigma_0 + \tau \sigma_3) \mathcal{C}_z$ can not be improved.
\end{rem}

\begin{proof}[Proof of Theorem~\ref{theorem_spectrum_critical}]
  (i) Proposition~\ref{prop-ess} implies the inclusion 
	\[
	\big(-\infty,-|m|\big]\,\cup\,\big[|m|,+\infty\big)\subset \spec_\ess  A_{\eta, \tau}.
	\]
	In addition, due to Theorem~\ref{theorem_boundary_triple_abstract} and Proposition~\ref{prop54}
	one has $\spec_\ess  A_{\eta, \tau} \,\cap \,\big(-|m|,|m|\big)=\big\{-\dfrac{\tau m}{\eta}\big\}$, which gives the claim.
		
  To prove item~(ii) we note first that by Theorem~\ref{theorem_boundary_triple_abstract} a point $z \in \res A_0$ is an eigenvalue of 
  $A_{\eta, \tau}$ if and only if zero is an eigenvalue of $\Theta - M_z$. Using a similar calculation as in~\eqref{Birman_Schwinger_kernel} 
  this shows that $z \in \res A_0$ is an eigenvalue of $A_{\eta, \tau}$ if and only if there exists $\psi \in \dom \Theta \subset L^2(\Sigma; \mathbb{C}^2)$ such that
  \begin{equation*}
    -\Lambda (\eta \sigma_0 + \tau \sigma_3)^{-1} \big( \sigma_0 + (\eta \sigma_0 + \tau \sigma_3) \mathcal{C}_z \big) \Lambda \psi = 0,
  \end{equation*}
  i.e. if and only if $\varphi := \Lambda \psi \in H^{-\frac{1}{2}}(\Sigma; \mathbb{C}^2)$ satisfies $\big( \sigma_0 + (\eta \sigma_0 + \tau \sigma_3) \mathcal{C}_z \big) \varphi = 0$. 
  
  Eventually, since $\dom A_{\eta, \tau}$ is independent of $m$, it suffices to prove statement~(iii) for $m \neq 0$. In this case the claim is a consequence of Proposition~\ref{prop-discr},
	as $\spec_\ess(A_{\eta, \tau}) \cap (-|m|, |m|) \neq \emptyset$. 
\end{proof}

Finally, we state several symmetry relations in the spectrum of $A_{\eta, \tau}$. The following proposition is the counterpart of Proposition~\ref{proposition_symmetry_relations_noncritical} for critical interaction strengths.

\begin{prop} \label{proposition_symmetry_relations_critical}
  Let $\eta, \tau \in \mathbb{R}$ with $\eta^2 - \tau^2 = 4$. Then the following holds:
  \begin{itemize}
    \item[\textup{(i)}] $z \in \spec_\textup{p} A_{\eta, \tau}$ if and only if $z \in \spec_\textup{p} A_{-\eta, -\tau}$.
    \item[\textup{(ii)}] $z \in \spec_\textup{p} A_{\eta, \tau}$ if and only if $(-z) \in \spec_\textup{p}A_{-\eta, \tau}$.
  \end{itemize}
\end{prop}
\begin{proof}
  In the following set $A_{\eta, \tau}^{1} := A_{\eta, \tau} \upharpoonright (\dom A_{\eta, \tau} \cap H^1(\mathbb{R}^2 \setminus \Sigma; \mathbb{C}^2))$. Then by Theorem~\ref{theorem_self_adjoint_critical} the operator $A_{\eta, \tau}^{1}$ is essentially self-adjoint in $L^2(\mathbb{R}^2; \mathbb{C}^2)$ and, in particular, $\overline{A_{\eta, \tau}^{1}} = A_{\eta, \tau}$.

  (i) Consider the unitary and self-adjoint mapping
  \begin{equation*}
    U: L^2(\Omega_+; \mathbb{C}^2) \oplus L^2(\Omega_-; \mathbb{C}^2) \rightarrow L^2(\Omega_+; \mathbb{C}^2) \oplus L^2(\Omega_-; \mathbb{C}^2), \quad U (f_+ \oplus f_-) = f_+ \oplus(-f_-).
  \end{equation*}
  As in the proof of Proposition~\ref{proposition_symmetry_relations_noncritical}~(i) one verifies $A_{\eta, \tau}^{1} = U A_{-\eta, -\tau}^1 U$.
  By taking closures we find $A_{\eta, \tau} = U A_{-\eta, -\tau} U$ and hence the claim follows.
  
  (ii) Consider the nonlinear charge conjugation operator
  \begin{equation*}
    C f = \sigma_1 \overline{f}, \qquad f \in L^2(\mathbb{R}^2; \mathbb{C}^2).
  \end{equation*}
  Then $C^2 f = f$ for $f \in L^2(\mathbb{R}^2; \mathbb{C}^2)$ and in the same way as in the proof of 
  Proposition~\ref{proposition_symmetry_relations_noncritical}~(ii) one obtains $C A_{\eta, \tau}^{1} = -A_{- \eta, \tau}^1 C$.
  Taking closures leads to $C A_{\eta, \tau} = -A_{- \eta, \tau} C$, which implies (ii).
\end{proof}


\subsection{Case of several loops}\label{secsevloops}

To prove Theorem \ref{sevloops} we use similar constructions as in the case of one loop. We give some comments on necessary modifications in this subsection. Let $N\geq1$ and let $\Sigma_j, j\in\{1,\dots,N\},$ be non-intersecting $C^\infty$-smooth loops  with normals $\nu_j$.
We set $\Sigma:=\bigcup_{j=1}^N \Sigma_j$, and for $f \in H(\sigma,\RR^2\setminus\Sigma)$ we denote its Dirichlet traces from Lemma~\ref{lemma_trace}
on the two sides of $\Sigma_j$ by $\mathcal{T}_{\pm, j}^D f$,
where $-$ corresponds to the side to which  $\nu_j$ is directed.
The Sobolev spaces on $\Sigma$ are defined by $H^s(\Sigma):=\bigoplus_{j=1}^N H^s(\Sigma_j)$, and for $\varphi\in H^s(\Sigma)$
we denote by $\varphi_j$ its restriction on $\Sigma_j$. Furthermore, if $\Lambda_j$ denotes the isomorphism defined in~\eqref{def_Lambda} on $\Sigma_j$,
then we set $\Lambda:=\bigoplus_{j=1}^N\Lambda_j$. As in the case of one loop one starts with the symmetric operator $S := A_0 \upharpoonright H^1_0(\mathbb{R}^2 \setminus \Sigma; \mathbb{C}^2)$.
For $z \in \res A_0$ and $\varphi \in L^2(\Sigma; \mathbb{C}^2)$ we introduce
  \begin{equation*}
    \Phi_z \varphi (x) = \int_\Sigma \phi_z(x-y) \varphi(y) \dd s(y), \quad x \in \mathbb{R}^2\setminus\Sigma.
  \end{equation*}
	As for the single loop in Proposition~\ref{proposition_Phi_z} one shows that $\Phi_z$ extends to a bounded map $\Phi_z: H^{-\frac{1}{2}}(\Sigma;\CC^2)\to L^2(\RR^2;\CC^2)$
	with $\ran \Phi_z = \ker (S^* - z)$.	The associated principal value operator ${\mathcal C}_z$,
\[
  \big(\mathcal{C}_z \varphi\big)(x) := \pv \int_\Sigma \phi_z(x-y) \varphi(y) \dd s(y), \quad \varphi \in C^\infty(\Sigma; \mathbb{C}^2), ~x \in \Sigma,
\]
has a block structure of the form
\begin{align}
   \label{czj}
(\mathcal{C}_z \varphi)_j(x)&=\cC^j_z \varphi_j (x) + \sum_{k\ne j} ({\mathcal K}^{j,k}_z \varphi_k) (x),
 \quad \varphi \in C^\infty(\Sigma; \mathbb{C}^2), ~x \in \Sigma_j,\\
(\cC^j_z \varphi_j)(x)&= \pv \int_{\Sigma_j} \phi_z(x-y) \varphi_j (y) \dd s(y), \quad x\in\Sigma_j,\\
({\mathcal K}^{j,k}_z \varphi_k) (x)&= \int_{\Sigma_k} \phi_z(x-y) \varphi_k(y) \dd s(y), \quad x\in\Sigma_j.
\end{align}
The operators $\cC^j_z$ are the same as in the one loop case, while the operators ${\mathcal K}^{j,k}_z$ have smooth integral kernels; hence, they define
bounded operators from $H^s(\Sigma_k,\CC^2)$ to $H^{t}(\Sigma_j,\CC^2)$ for any $s,t\in\RR$. With the help of Proposition~\ref{proposition_Plemelj_Sokhotskii} one can show now the trace equality
\[
  \cT_{\pm,j}^D \Phi_z \varphi = \mp\,\dfrac{\rmi}{2}\, (\sigma\cdot \nu_j)\,\varphi_j + \big(\mathcal{C}_z \varphi\big)_j.
\]
The construction of the boundary triple takes then literally the same form as for a single loop. Let $\zeta\in \res A_0$ be fixed and set $(\cT_\pm^D f):=(\cT_{\pm,j}^D f)_{j=1}^N$. Then $\{ L^2(\Sigma; \mathbb{C}^2), \Gamma_0, \Gamma_1 \}$  with
\begin{align*}
    \Gamma_0 f&=\rmi\Lambda^{-1} (\sigma\cdot \nu)\big(\cT_+^D f -\cT_-^D f),\\
    \Gamma_1 f&=\dfrac{1}{2}\,\Lambda \Big((\cT_+^D f_+ + \cT_-^D f_-) -(\mathcal{C}_\zeta+\mathcal{C}_{\Bar\zeta}) \Lambda \Gamma_0 f \Big),
\end{align*}
is a boundary triple for $S^*$.
The corresponding $\gamma$-field $G$ and  Weyl function $M$ are $z \mapsto G_z = \Phi_z \Lambda$
and
  \begin{equation*}
    z\mapsto M_z = \Lambda \Big(\mathcal{C}_z -\dfrac{1}{2} \big(\mathcal{C}_\zeta+\mathcal{C}_{\Bar\zeta} \big) \Big)\Lambda.
  \end{equation*}

Assume first that $|\eta_j|\ne|\tau_j|$ for all $j\in\{1,\dots,N\}$. Define the linear operator $\Theta$ in $L^2(\Sigma;\CC^2)$ by
\[
\Theta=-\Lambda\bigg[\Xi + \frac12\,(\mathcal{C}_\zeta + \mathcal{C}_{\bar\zeta})\bigg]\Lambda,\quad
(\Xi \varphi)_j :=\frac1{\eta^2_j - \tau^2_j}(\eta_j\sigma_0 - \tau_j\sigma_3)\,\varphi_j,
\]
on its maximal domain in $L^2(\Sigma; \mathbb{C}^2)$. Then the operator $A_{\Sigma,\cP}$ defined in~\eqref{dirdeltamult} corresponds to the boundary condition $\Gamma_1 f= \Theta \Gamma_0 f$.
Using \eqref{czj} one sees that $\Theta$ can be written as $\Theta=\bigoplus_{j=1}^N \Theta_j + \widetilde \Theta$, where
$\Theta_j$ is the operator in $L^2(\Sigma_j;\CC^2)$ acting as
\[
\Theta_j= -\Lambda_j\bigg[\frac1{\eta^2_j - \tau^2_j}(\eta_j\sigma_0 - \tau_j\sigma_3) + \frac12\,(\mathcal{C}^j_\zeta + \mathcal{C}^j_{\bar\zeta})\bigg]\Lambda_j,
\]
with maximal domain, while
$\widetilde\Theta$ is a bounded operator from $H^s(\Sigma,\CC^2)$ to $H^{t}(\Sigma,\CC^2)$ for any $s,t\in\RR$ which is self-adjoint in $L^2(\Sigma; \mathbb{C}^2)$.
Hence, the self-adjointness of $\Theta$ is determined by the self-adjointness of $\bigoplus_{j=1}^N \Theta_j$,
and each $\Theta_j$ is exactly of the form as in the single-loop case. Hence, $\Theta_j$ is self-adjoint by Lemma~\ref{lem3132} and Lemma~\ref{self-theta} and thus, also $\Theta$ is self-adjoint in $L^2(\Sigma; \mathbb{C}^2)$. This implies also the statements concerning the domain regularity.

In order to study the essential spectrum we decompose $M_z$ to blocks as in~\eqref{czj}
and remark that the terms $\cK^{j,k}_z$ produce compact remainders, which do not influence the essential spectrum. Hence, the condition $0\in \spec_\ess (M_z-\Theta)$ is equivalent to 
\[
0\in\spec_\ess \left(\bigoplus_{j=1}^N \bigg( \Lambda_j\frac1{\eta^2_j - \tau^2_j}(\eta_j\sigma_0 - \tau_j\sigma_3)\Lambda_j
+\Lambda_j \mathcal{C}^j_z \Lambda_j\bigg)\right).
\]
As each of the terms on the right-hand side is covered by the analysis of the single-loop case, the statement on the essential spectrum of $M_z - \Theta$ and thus, with the help of Theorem~\ref{theorem_boundary_triple_abstract}, also of $A_{\Sigma, \mathcal{P}}$, follows.

\kp{
If for some $j$ one has $|\eta_j|=|\tau_j|$, then one follows the same technical strategy as the one in Section~\ref{sec-noncrit} for $|\eta|=|\tau|$,
i.e. one has to deal with additional orthogonal projectors, and all other constructions are easily adapted.
}

\end{document}